\documentclass[10pt,a4paper,twoside,reqno]{amsart}

\usepackage[a4paper,
left=1in,right=1in,top=1in,bottom=1.2in,%
footskip=.25in]{geometry}

\usepackage{bm}
\usepackage[utf8]{inputenc}
\usepackage{latexsym, amsmath, amstext, amssymb, amsfonts, amscd, bm, array, multirow, amsbsy, mathrsfs}
\usepackage{amsthm}
\usepackage{t1enc}
\usepackage[mathscr]{eucal}
\usepackage{indentfirst}
\usepackage{pb-diagram}
\usepackage{graphicx}
\usepackage{fancyhdr}
\usepackage{fancybox}
\usepackage{enumerate}
\usepackage{color}
\usepackage{tikz-cd}
\usetikzlibrary{arrows}
\usepackage[all]{xy}
\usepackage{hyperref}
\usepackage{tikz}
\usepackage{xparse}
\hypersetup{colorlinks=false,pdfborderstyle={/S/U/W 0}}
\usepackage{mathtools}
\mathtoolsset{showonlyrefs}
\usetikzlibrary{matrix}
\usepackage{upgreek}

\newcommand{\arxiv}[1]{{\tt
    \href{http://www.arXiv.org/abs/#1}{arXiv:#1}}}

\usepackage{url}
\usepackage[sort&compress,comma]{natbib}
\bibpunct{[}{]}{,}{n}{}{,}
\theoremstyle{plain}
\newtheorem{thm}{Theorem}[section]

\newtheorem{prop}[thm]{Proposition}

\newtheorem{lemma}[thm]{Lemma}
\newtheorem{cor}[thm]{Corollary}

\theoremstyle{definition}
\newtheorem{definition}[thm]{Definition}

\theoremstyle{remark}
\newtheorem{remark}[thm]{Remark}
\newtheorem{example}[thm]{Example}

\newtheorem*{ack}{Acknowledgements}

\newcommand{\Aut}{\mathrm{Aut}}

\DeclareFontFamily{U}{rsf}{}
\DeclareFontShape{U}{rsf}{m}{n}{<5> <6> rsfs5 <7> <8> <9> rsfs7 <10-> rsfs10}{}
\DeclareMathAlphabet\Scr{U}{rsf}{m}{n}

\def\R{\mathbb{R}}

\def\dd{\mathrm{d}}

\def\cQ{\mathcal{Q}}

\def\Diff{\mathrm{Diff}}

\def\Id{\mathrm{Id}}

\def\frw{\mathfrak{w}}

\newcommand{\be}{\begin{equation*}}
\newcommand{\ee}{\end{equation*}}
\newcommand{\ben}{\begin{equation}}
\newcommand{\een}{\end{equation}}
\newcommand{\beqa}{\begin{eqnarray*}}
	\newcommand{\eeqa}{\end{eqnarray*}}
\newcommand{\beqan}{\begin{eqnarray}}
\newcommand{\eeqan}{\end{eqnarray}}

\newcommand{\Tr}{\mathrm{Tr}}

\def\cC{{\mathcal C}}

\newcommand{\Sol}{\mathrm{Sol}}
\newcommand{\Conf}{\mathrm{Conf}}

\def\Spin{\mathrm{Spin}}

\def\Spin{\mathrm{Spin}}

\def\U{\mathrm{U}}

\def\cA{\mathcal{A}}

\def\cI{\mathcal{I}}
\def\cP{\mathcal{P}}

\def\cG{\mathcal{G}}

\def\cP{\mathcal{P}}
\def\cF{\mathcal{F}}
\def\cC{\mathcal{C}}

\def\G_2{\mathrm{G_2}}

\def\cS{\mathcal{S}}

\def\Aut{\mathrm{Aut}}

\def\G{\mathrm{G}}

\def\R{\mathbb{R}}

\def\cL{\mathcal{L}}

\def\dd{\mathrm{d}}

\def\Met{\mathrm{Met}}

\def\Ric{\mathrm{Ric}}

\usepackage{enumitem}
\setlist[itemize]{leftmargin=*}
\setlist[enumerate]{leftmargin=*}

\newcolumntype{P}[1]{>{\centering\arraybackslash}p{#1}}

\usepackage[math]{anttor}
\usepackage[T1]{fontenc}

\begin{document}

\title[The Cauchy problem for gradient generalized Ricci solitons on a bundle gerbe]{The Cauchy problem for gradient generalized Ricci solitons on a bundle gerbe}

\author[S. Bunk]{Severin Bunk} \address{School of Physics, Astronomy and Mathematics, University of Hertfordshire, United Kingdom}
\email{s.bunk@herts.ac.uk}

\author[Miguel Pino Carmona]{Miguel Pino Carmona} 

\author[C. S. Shahbazi]{C. S. Shahbazi} \address{Departamento de Matem\'aticas, Universidad UNED - Madrid, Reino de Espa\~na}
\email{mpino185@alumno.uned.es}
\email{cshahbazi@mat.uned.es}


\begin{abstract}
We prove well-posedness of the analytic Cauchy problem for gradient generalized Ricci solitons on an abelian bundle gerbe and solve the initial data equations on every compact Riemann surface. Along the way, we provide a novel characterization of the self-similar solutions of the generalized Ricci flow by means of families of automorphisms of the underlying abelian bundle gerbe covering families of diffeomorphisms isotopic to the identity.
\end{abstract}
 
\maketitle



\section{Introduction}


\noindent
The generalized Ricci flow \cite{GFStreetsBook,OliynykSuneetaWoolgar} is a natural generalization of the celebrated Ricci flow \cite{Hamilton}. This evolution flow, which originates in the string theory literature as the renormalization group flow of the bosonic string \cite{Polchinski}, describes the evolution of a Riemannian metric $g$ coupled to a \emph{b-field}. Its self-similar solutions, known as generalized Ricci solitons, are expected to foster remarkable applications in the classification and uniformization of compact complex surfaces \cite{Streets1,Streets2,Streets:2012zz}. The main goal of this article is to study the Riemannian Cauchy problem for gradient generalized Ricci solitons, proving its well-posedness for analytic initial data. \medskip

\noindent
The generalized Ricci flow is typically studied either on the bare underlying manifold, as a flow for metrics and closed three-forms, or on an exact Courant algebroid, as a flow of generalized metrics and divergence operators \cite{GarciaFernandez,GFStreetsBook}. Here, motivated by supergravity, we adopt a different perspective: by \emph{Dirac quantization}, the three-form flux occurring in the generalized Ricci flow must be integral, and hence can be interpreted as the curvature of a curving on a bundle gerbe equipped with a connective structure \cite{Murray,BehrendXu}. Consequently, the theory of $\U(1)$ bundle gerbes can be used to give a gauge-theoretic global geometric interpretation of the generalized Ricci flow and, in particular, of the notion of \emph{b-field} as a curving on a bundle gerbe. This brings the generalized Ricci flow into the realm of \emph{higher geometry} and in line with the ongoing modern development of supergravity and string theory as a higher gauge theory \cite{BaezSchreiber,Borstenetall,Tellez}. \medskip

\noindent 
Considering the generalized Ricci flow as an evolution flow on a bundle gerbe with a fixed connective structure, we define its self-similar flows as those evolving by \emph{families} of automorphisms, namely stable isomorphisms \cite{Waldorf0,Waldorf} covering possibly non-trivial diffeomorphisms of the underlying manifold. This is delicate, since it requires a sensible notion of \emph{smooth} family of automorphisms at a level of generality that we have not found in the literature. Recently, an exponential map for a large class of 2-groups, including the 2-group of stable isomorphisms of a bundle gerbe, was obtained in \cite[Theorem 1.4]{Tellez}. However, we cannot directly apply this result to our purposes, as we need to consider automorphisms that cover diffeomorphisms of the underlying manifold. Instead, we propose an explicit natural definition of a smooth family of automorphisms under a technical condition on the submersion underlying the given bundle gerbe, which builds on the general action that we obtain in Theorem \ref{thm:action}. This results in a novel computation of the generalized Ricci soliton system via the action of a smooth family of higher isomorphisms. This should be compared with the computation of the generalized Ricci soliton system carried out in \cite{GFStreetsBook} via families of automorphisms of an exact Courant algebroid. Focusing on \emph{gradient} generalized Ricci solitons, a streamlined version of our main result is the following.

\begin{thm}
The Cauchy problem for the gradient generalized Ricci soliton system on a $\U(1)$ bundle gerbe is well-posed for analytic initial data.
\end{thm}

\noindent
The reader is referred to Theorem \ref{thm:CylindricalWellPosedness} for more details. To prove this result, we first transform the gradient generalized Ricci soliton from the \emph{string frame} to the \emph{Einstein frame}, and then reduce the metric, curving, and dilaton to families of objects that satisfy a system of second-order evolution equations and \emph{time-dependent} constraints on a given hypersurface $\Sigma\subset M$. This requires reducing the abelian bundle gerbe to the Cauchy surface, which results in an abelian bundle gerbe equipped with certain time-dependent extra data. We then apply the Cauchy-Kovalevskaya theorem after an adequate change of variables to a set of local variables that is gauge-invariant with respect to stable isomorphisms. Finally, the result follows from a series of long computations, deferred to the appendix, which guarantee that the local flow preserves the constraint equations. This theorem can be understood as a natural generalization of the classical results of Koiso \cite{Koiso1981} on the Riemannian Cauchy problem for Einstein metrics to the framework of higher geometry. We then consider the constraint equations of the Cauchy problem of the gradient generalized Ricci soliton on a two-dimensional oriented manifold, proving the following result, see Theorem \ref{thm:wellposednessconstraints3d} for more details. 

\begin{thm}
Let $\Sigma$ be a compact and oriented two-dimensional manifold. For every conformal class $[h]$ of Riemannian metrics on $\Sigma$, there exists a solution to the constraint equations of the NS gradient generalized Ricci soliton system whose Riemannian metric belongs to $[h]$.  
\end{thm}

\noindent
As an immediate consequence of this theorem, we obtain the following result.

\begin{cor}
Let $(\Sigma,h)$ be a punctured Riemann surface. Then, there exists an NS gradient generalized Ricci soliton $(g,b,\phi)$ on a bundle gerbe over a three-manifold $M$, and an embedding $\iota \colon \Sigma \hookrightarrow M$ such that $\iota^{\ast}g \in [h]$. In particular, there exist infinitely many different topological types of three-dimensional NS gradient generalized Ricci solitons. 
\end{cor}
 
\noindent
Here, by an \emph{NS gradient generalized Ricci soliton} we refer to a gradient generalized Ricci soliton that also satisfies the \emph{dilaton equation} of NS supergravity and is therefore an honest solution of the bosonic equations of the latter. We hope that these results can be exploited in the future to construct new complete gradient generalized Ricci solitons in the spirit of \cite{PodestaRaffero}. In a different direction, it would be interesting to generalize these results to the case of \emph{Heterotic solitons} \cite{Moroianu:2021kit}, which arise as self-similar solutions to the Heterotic renormalization group flow at second order \cite{Moroianu:2023jof,Papadopoulos:2024tgs}. This flow yields a natural generalization of the notion of generalized Ricci flow that involves a Riemann curvature term \emph{squared}, which makes the analysis involved notably different from the generalized Ricci flow case.

\begin{ack} 
The work of CSS was partially supported by the Leonardo grant LEO22-2-2155 of the BBVA Foundation and the research grant PID2023-152822NB-I00 from the Ministry of Science of the Spanish government. The work of MPC was supported by the UNED-Santander 2024 predoctoral fellowship of the Santander Open Academy foundation. CSS would like to thank D. Michael Roberts and K. Waldorf for very illuminating discussions and explanations, and M. García-Fernández and J. Streets for interesting comments. Furthermore, we have benefited from ongoing discussions and collaborations with Vicente Cortés and O. Schiller, who have obtained analogous results in the smooth category in the Lorentzian signature \cite{Schiller}.
\end{ack}


\section{The generalized Ricci flow on a bundle gerbe}
\label{sec:NSNSsystem}


This section introduces the generalized Ricci flow \cite{OliynykSuneetaWoolgar,Polchinski} on a fixed bundle gerbe $(\cP,A,\pi,\mu)$ over an oriented manifold $M$. Our primary objective is to characterize its self-similar flows, which lead to the notion of \emph{generalized Ricci-soliton} \cite{GFStreetsBook}. Since the notion of a bundle gerbe is not widely encountered in the geometric analysis literature, we devote the first few subsections to introducing its basic theory in a style that we hope is palatable to a differential geometer. For full details, we refer the reader to~\cite{Waldorf0, Waldorf, Murray, BehrendXu, Bunk21}.


\subsection{Why bundle gerbes at all?}


At first glance, we might be inclined to ask a very natural question about the generalized Ricci flow: if the evolving data consists of a Riemannian metric $g_t$ and a closed 3-form $H_t \in \Omega^3(M)$, why do we need the theory of bundle gerbes at all? Why can we not simply consider the evolution system purely as an evolution of the tensors $(g_t, H_t)$?\medskip

In order to propose an answer to this question, it is instructive to look at the analog situation that occurs one categorical level lower, which corresponds to principal $U(1)$-bundles and classical electromagnetism. In Maxwell's theory, the electromagnetic field strength is a closed 2-form $F \in \Omega^2(M)$. Classically, the equations of motion can be written entirely in terms of the field strength $F$. However, this purely tensor-based description is physically incomplete. The Aharonov-Bohm effect, which has been measured experimentally, demonstrates that there exist non-simply-connected physical systems with non-trivial electromagnetic interactions yet a vanishing field strength. Via the Dirac integrality condition, physicists concluded that such non-trivial electromagnetic effects must be described by a flat connection of non-trivial holonomy. Hence, the \emph{correct} variable to describe Maxwell electrodynamics is a connection on a principal $\U(1)$ bundle, and within this interpretation, the classical field strength is simply a \emph{derived} object that loses information. In particular, the notion of connection allows us to define parallel transport along paths, a concept that is remarkably important both physically and mathematically. \medskip

The system we study in the generalized Ricci flow represents a completely analogous setup at a higher field-strength degree: the closed 3-form $H$ is supposed to represent the field strength of an interaction of \emph{electromagnetic type}. Hence, similarly to the Maxwell case, in order to correctly describe the system and all the physical effects and interactions that it describes, we need to properly identify the \emph{gauge variable} underlying the field strength $H$, which again can be done via a Dirac integrality condition. This seemingly minor modification in the degree of the field strength has dramatic consequences for its mathematical gauge-theoretic interpretation. In particular, it requires far more sophisticated mathematical machinery: just as the 2-form $F$ requires a principal $U(1)$-bundle and a connection 1-form to be described globally, the 3-form $H$ requires a bundle gerbe equipped with a connective structure and a curving. The latter is the fundamental gauge-theoretic variable underlying $H$. The notion of curving and connective structure allows us to define parallel transport along surfaces, a tremendously important concept in string theory. The main novelty with respect to the classical Maxwell case is that the gauge transformations of this curving admit gauge transformations themselves. Geometrically, this leads to a higher categorical structure by which bundle gerbes with connective structure assemble into a \emph{2-groupoid}: the objects are the bundle gerbes with connection, the 1-morphisms are principal $U(1)$-bundles (representing the gauge transformations), and the 2-morphisms are bundle isomorphisms (representing the gauge transformations between gauge transformations). \medskip

While formulating the generalized Ricci flow on a bundle gerbe captures these higher gauge symmetries and allows us to define parallel transport, it should be noted that the generalized Ricci flow is most commonly studied in the literature as an evolution of generalized metrics on an exact Courant algebroid \cite{GFStreetsBook}. These two perspectives are, in fact, deeply intertwined. Through Hitchin's generalized tangent bundle construction, every bundle gerbe equipped with a connective structure naturally induces an exact Courant algebroid, and this construction exhausts all integral exact Courant algebroids up to isomorphism \cite{Hitchin}. Both pictures are relevant and complementary. Here, motivated by supergravity and Dirac quantization, we adopt the bundle-gerbe perspective to give a gauge-theoretic global geometric interpretation of the generalized Ricci flow and generalized Ricci solitons.


\subsection{Bundle gerbes}


The main purpose of developing the theory of bundle gerbes is to provide an intrinsic and explicit geometrization of $H^3(M,\mathbb{Z})$ that does not manifestly use higher category theory. To provide such an intrinsic and global description, one must introduce an \textbf{auxiliary fibration}, which we denote by $\pi \colon Y \to M$, over the underlying given manifold $M$, which serves as a flexible geometric generalization of the notion of an open cover. In this coordinate-free global setting, the $n$-fold intersections of an open cover are replaced by the $n$-fold fibered products of $Y$ with itself over $M$, which we denote by $Y^{[n]} = Y \times_M \cdots \times_M Y$. The collection of these fibered products forms a highly structured hierarchy known as the \textbf{\v{C}ech nerve} of the submersion $\pi$. Specifically, the \v{C}ech nerve is a simplicial manifold whose $n$-th level consists of the space $Y^{[n+1]}$, equipped with natural projection maps between the levels. The most important of these are the \textbf{face maps} $d_i \colon Y^{[n]} \to Y^{[n-1]}$ for $i = 1, \dots, n$, which act simply by omitting the $i$-th factor from the Cartesian product. 

\begin{remark}
The terminology \emph{\v{C}ech nerve} unifies concepts from both algebraic topology and category theory. In classical topology, the \emph{nerve of an open cover} $\{U_i\}$ of a space $X$ is an abstract simplicial complex that tracks the combinatorics of the intersections: a $0$-simplex is an open set $U_i$, a $1$-simplex is a non-empty double intersection $U_i \cap U_j$, and generally, an $n$-simplex is a non-empty $(n+1)$-fold intersection $U_{i_0} \cap \dots \cap U_{i_n}$. This forms the combinatorial foundation for classical \v{C}ech cohomology. In category theory, the \emph{nerve of a category} is a construction that yields a simplicial set whose $0$-simplices are the objects, the $1$-simplices are the morphisms, and the $n$-simplices are sequences of $n$ composable morphisms. The \v{C}ech nerve $Y^{[\bullet]}$ of a surjective submersion $\pi \colon Y \to M$ naturally combines these two ideas. From the submersion, one can form the \emph{\v{C}ech groupoid}, where the space of objects is $Y$ and the space of morphisms is the fiber product $Y^{[2]} = Y \times_M Y$. Taking the categorical nerve of this Lie groupoid produces a simplicial manifold where the $n$-th level consists of $n$ composable morphisms. Due to the structure of the \v{C}ech groupoid, the space of $n$ composable morphisms is canonically diffeomorphic to the $(n+1)$-fold fiber product $Y^{[n+1]}$. Thus, the \v{C}ech nerve represents both the topological intersections of our \emph{generalized cover} and the categorical composition of its transition data.
\end{remark}

\noindent
From a differential-geometric point of view, the \v{C}ech nerve construction becomes especially relevant through its action on differential forms. Just as the face maps project spaces down a level, pulling back differential forms along these face maps pushes forms up a level. By taking the alternating sum of these pullbacks, we can define a differential operator $\delta = \sum (-1)^i \dd_i^*$ that maps $p$-forms on $Y^{[k]}$ to $p$-forms on $Y^{[k+1]}$. A standard computation shows that $\delta^2 = 0$. Due to the fact that $\pi \colon Y \to M$ is a surjective submersion, this construction yields an \textbf{exact complex} of differential forms:
\begin{equation}
    0 \to \Omega^p(M) \xrightarrow{\pi^*} \Omega^p(Y) \xrightarrow{\delta} \Omega^p(Y^{[5]}) \xrightarrow{\delta} \Omega^p(Y^{[6]}) \xrightarrow{\delta} \cdots
\end{equation}
The exactness of this complex plays a fundamental role in the theory of bundle gerbes. It guarantees a generalized Poincaré lemma: if a differential form on the auxiliary space $Y$ or its fiber products $Y^{[k]}$ is annihilated by $\delta$, it can be \emph{descended} to a unique, globally defined differential form on $Y^{[k-1]}$. As we will see below, this is precisely the mechanism that allows us to extract the global three-form curvature $H \in \Omega^3(M)$ from the adequate notion of connection on a bundle gerbe, that we call \emph{curving}.

\begin{definition}
A \textit{bundle gerbe} on $M$ consists of a surjective submersion $\pi \colon Y \to M$, a $\U(1)$-bundle $\cP \to Y^{[2]}$, and a bundle isomorphism:
\begin{equation}
\mu \colon \pi^{\ast}_{12}\cP \otimes \pi^{\ast}_{23}\cP \xrightarrow{\simeq} \pi^{\ast}_{13}\cP
\end{equation}
over $Y^{[3]}$, where:
\begin{equation*}
\pi_{ij} \colon Y \times_M Y \times_M Y \to Y\times_M Y\, , \qquad i, j = 1,2,3
\end{equation*}
denotes the projection onto the $i$-th and $j$-th factors. This bundle isomorphism (sometimes called gerbe multiplication) must satisfy an associativity condition on $Y^{[4]}$. A \textit{connective structure} on a bundle gerbe as above consists of a connection on the $\U(1)$-bundle $\cP$ such that the bundle morphism $\mu$ is connection-preserving. A \textit{curving} on a bundle gerbe with connective structure is a $2$-form $b \in \Omega^2(Y)$ such that $\pi_2^*b - \pi_1^*b = F_A$ in $\Omega^2(Y^{[2]})$.
\end{definition}

\noindent
We denote a bundle gerbe with connective structure as a quadruple $(\cP,A,\pi,\mu)$. From now on, all bundle gerbes appearing in this article will be endowed with a connective structure unless stated otherwise. Thus, to save notation, the term bundle gerbe shall always include the choice of a connective structure. A compatible pair $(A,b)$ of a connective structure and a curving on a bundle gerbe is called a \textit{connection} on the bundle gerbe.

\begin{remark}
A more compact way of defining a bundle gerbe is as a central extension of Lie groupoids:
\begin{equation*}
Y\times \U(1) \to (\cP, A) \to Y\times_M Y \rightrightarrows Y \xrightarrow{\pi} M
\end{equation*}
where the $Y\times \U(1)$ is understood as the Lie groupoid whose source and target maps are both equal and given by the canonical projection $Y\times \U(1) \to Y$.
\end{remark}

\noindent
Notice that taking the exterior derivative of the curving condition $\pi_2^*b - \pi_1^*b = F_A$ yields $\pi_2^* db = \pi_1^* db$. Because the exterior derivative of the curving agrees on the fiber overlaps, $db \in \Omega^3(Y)$ descends to a globally defined, closed three-form $H_b \in \Omega^3(M)$ such that $\pi^*H = db$. This three-form $H_b$ is known as the \textit{curvature} of the bundle gerbe curving $b\in \Omega^2(Y)$ and, as we will see below, when appropriately normalized it is \emph{integral}. In the context of the generalized Ricci flow, the curving $b$ appears in the system precisely through its curvature $H_b \in \Omega^3(M)$.


\subsection{Bundle gerbes in terms of local data}


The notion of bundle gerbe that we have reviewed in the previous subsection is designed, as we have explained, to provide an intrinsic and global geometrization of $H^3(M,\mathbb{Z})$. However, for explicit calculations and applications, it is sometimes convenient to rephrase it in terms of local data associated with an explicit good open cover of the underlying manifold $M$, which we denote by $\{U_i\}$. For this, we consider $\{U_i\}$ as equipped with local sections $s_i \colon U_i \to Y$ of the given submersion $\pi \colon Y\to M$. Then, the tuple $(\cP,Y,A,b)$ consisting of a bundle gerbe equipped with a connective structure $\cA$ and a curving $b$ give rise to a \v{C}ech-Deligne cocycle as follows:
\begin{enumerate}
\item The two-form $b \in \Omega^2(Y)$ pulls back to a family of local two-forms:
\begin{equation*}
b_i = s_i^* b \in \Omega^2(U_i)
\end{equation*}

\noindent
Note that $H = \dd b_i$ locally and $\dd b_i = \dd b_j$ on $U_i \cap U_j \neq 0$.

\item The connective structure $A$ on $\cP \to Y\times_M Y$ yields a family of local 1-forms $A_{ij} \in \Omega^1(U_i \cap U_j)$ defined on double overlaps $U_i \cap U_j$. To see how this happens, note that the local sections $s_i \colon U_i \to Y$ define smooth maps of the form:
\begin{equation}
    (s_i, s_j) \colon U_i \cap U_j \to Y \times_M Y \qquad x \mapsto (s_i(x), s_j(x))
\end{equation}

\noindent
defined on non-trivial overlaps $U_i\cap U_j\neq 0$. Pulling the principal bundle $\cP$ back along the smooth map $(s_i, s_j) \colon U_i \cap U_j \to Y\times_M Y$, we obtain a principal $\U(1)$-bundle over the overlap $U_i \cap U_j$. Since these intersections are by assumption contractible, this pullback bundle $(s_i, s_j)^{\ast}\cP$ is trivializable. Hence, we may choose a trivializing section $o_{ij} \colon U_i\cap U_j \to (s_i, s_j)^{\ast}\cP$ for each overlapping pair $U_i$ and $U_j$. Equipping the principal bundle $(s_i, s_j)^{\ast}\cP \to U_i\cap U_j$ with the pull-back connection $(s_i, s_j)^{\ast}\cA$ we can finally pull-back this connection by $o_{ij} \colon U_i\cap U_j \to (s_i, s_j)^{\ast}\cP$, obtaining as a result a family of on-forms $A_{ij}$ defined on double overlaps. More explicitly:
\begin{equation*}
A_{ij} := o_{ij}^{\ast}(s_i, s_j)^{\ast}\cA \in \Omega^1(U_i \cap U_j)
\end{equation*}

\noindent
Similarly, we may pull back the curving condition $\pi_2^*b - \pi_1^*b = F_A$ along the map $(s_i, s_j) \colon U_i\cap U_j \to Y\times_M Y$. On the left-hand side, we use the property that composing the fiber product projections $\pi_1, \pi_2 \colon Y \times_M Y \to Y$ with our combined map $(s_i, s_j)$ simply recovers the original local sections: $\pi_1 \circ (s_i, s_j) = s_i$ and $\pi_2 \circ (s_i, s_j) = s_j$. Thus, the pullback of the curving difference evaluates exactly to the difference of the local $b$-fields:
\begin{equation}
    (s_i, s_j)^*(\pi_2^*b - \pi_1^*b) = s_j^*b - s_i^*b = b_j - b_i.
\end{equation}
On the right-hand side, the pullback of the globally defined curvature 2-form $F_A$ evaluates to the exterior derivative of our local connection one-form: $(s_i, s_j)^* F_A = \dd A_{ij}$. Equating the two sides yields the explicit local tensor transformation rule:
\begin{equation}
b_j - b_i = dA_{ij}
\end{equation}

\noindent
which is the local gluing condition we were after.
    
\item Over triple overlaps $U_i \cap U_j \cap U_k$, the bundle gerbe multiplication $\mu$ over $Y^{[3]}$ yields smooth transition maps $g_{ijk} \colon U_i \cap U_j \cap U_k \to \U(1)$. The requirement that $\mu$ is connection-preserving yields the patching rule for the one-forms:
\begin{equation}
A_{jk} - A_{ik} + A_{ij} = d \log g_{ijk}
\end{equation}
where the right-hand side is the pullback of the Maurer-Cartan form on $\U(1)$. Furthermore, the associativity of $\mu$ over $Y^{[4]}$ ensures these transition functions satisfy the standard \v{C}ech cocycle condition $g_{jkl} g_{ijl}^{-1} g_{ijk} g_{ikl}^{-1} = 1$ on four-fold overlaps. Thus, the collection $\{g_{ijk}\}$ defines a \v{C}ech 2-cocycle with values in the sheaf of smooth $\U(1)$-valued functions, which is typically denoted by $\underline{\U(1)}$. This cocycle represents a well-defined class $[g] \in \check{H}^2(M; \underline{\U(1)})$.
\end{enumerate}
 
We have mentioned earlier that bundle gerbes \emph{geometrize} $H^3(M,\mathbb{Z})$. In particular, every bundle gerbe defines a canonical integral cohomology class of degree three, which can be understood in terms of the class $[g] \in \check{H}^2(M; \underline{\U(1)})$ we just introduced above. The class $[g] \in \check{H}^2(M; \underline{\U(1)})$ can be related to integral cohomology by considering the short exact exponential sequence of sheaves:
\begin{equation}
    0 \to \underline{\mathbb{Z}} \xrightarrow{\iota} \underline{\mathbb{R}} \xrightarrow{\exp(2\pi i \cdot)} \underline{\U(1)} \to 0,
\end{equation}
where $\underline{\mathbb{Z}}$ is the locally constant sheaf of integers and $\underline{\mathbb{R}}$ is the sheaf of smooth real-valued functions. Passing to sheaf cohomology, this short exact sequence induces a long exact sequence. Because the sheaf of smooth real-valued functions $\underline{\mathbb{R}}$ admits partitions of unity, its higher cohomology groups vanish, that is $\check{H}^k(M; \underline{\mathbb{R}}) = 0$ for all $k \geq 1$. Consequently, the connecting homomorphism $\delta$ from the long exact sequence yields a canonical isomorphism:
\begin{equation}
    \delta \colon \check{H}^2(M; \underline{\U(1)}) \xrightarrow{\simeq} \check{H}^3(M; \mathbb{Z})
\end{equation}
The \emph{Dixmier-Douady class} $DD$ is precisely the image of the class $[g] \in \check{H}^2(M; \underline{\U(1)})$ under this isomorphism, that is:
\begin{equation*}
    DD = \delta([g]) \in H^3(M; \mathbb{Z})
\end{equation*}

\noindent
We can construct a representative \v{C}ech 3-cocycle for $DD$ explicitly via the standard diagram chase for $\delta$. Since the open sets $U_i$ and their finite intersections are contractible, we can choose local smooth lifts $h_{ijk} \colon U_i \cap U_j \cap U_k \to \mathbb{R}$ such that:
\begin{equation}
    \exp(2\pi i h_{ijk}) = g_{ijk}.
\end{equation}
Because the original transition functions $g_{ijk}$ form a 2-cocycle, evaluating the \v{C}ech coboundary of our lifts $h_{ijk}$ on a four-fold overlap yields an integer. Specifically, we define:
\begin{equation}
    z_{ijkl} := h_{jkl} - h_{ijl} + h_{ijk} - h_{ikl}.
\end{equation}
Applying the exponential map to both sides reveals that $\exp(2\pi i z_{ijkl}) = 1$, meaning that the smooth function $z_{ijkl}$ is constant and takes values strictly in $\mathbb{Z}$. Furthermore, a direct computation shows that $\check{\delta} z = 0$, and thus the collection $\{z_{ijkl}\}$ defines a class in $\check{H}^3(M,\underline{\mathbb{Z}})$, which in turn is naturally isomorphic to $H^3(M,\mathbb{Z})$. Therefore, we end up with an element in the degree-three singular cohomology of $M$, which we can take by definition to be the Dixmier-Douady class $DD \in H^3(M,\mathbb{Z})$ of the given bundle gerbe. When mapping this integral class into real cohomology via the inclusion $\mathbb{Z} \hookrightarrow \mathbb{R}$, De Rham's theorem tells us it must correspond to a closed differential 3-form. This geometric representative in $H^3_{dR}(M)$ is exactly $[H_b/2\pi]$, where $H_b$ is the globally defined three-form curvature curvature of any curving $b$ on the bundle gerbe.


\subsection{The automorphism 2-group of a bundle gerbe}


The symmetries of a bundle gerbe with a fixed connective structure play a crucial role in characterizing the self-similar solutions of the generalized Ricci flow. Unlike the symmetries of a principal bundle, the symmetries of bundle gerbes do not form a group. Instead, stable isomorphisms conform to a higher algebraic analog of a group called a \emph{2-group}. We begin with the definition of stable isomorphism between given bundle gerbes $(\cP_1,A_1,\pi_1)$ and $(\cP_2,A_2,\pi_2)$. This is the notion of isomorphism adequate to guarantee that \emph{isomorphism classes} of bundle gerbes are in bijection with $H^3(M,\mathbb{Z})$ \cite{MurrayStevenson}.

\begin{definition}
An \emph{isomorphism} between $(\cP_1,A_1,\pi_1)$ and $(\cP_2,A_2,\pi_2)$ is a tuple:
\begin{eqnarray*}
 (\cQ,\Theta,\xi,\Psi)   \colon (\cP_1,A_1,\pi_1) \to (\cP_2,A_2,\pi_2)
\end{eqnarray*}

\noindent
consisting of a principal $\U(1)$ bundle $\cQ \to Z$ with connection $\Theta$ defined on the total space of the smooth submersion $\xi \colon Z \to Y_1 \times_M Y_2$ and an isomorphism of $\U(1)$ bundles with connection over $Z^{[2]}$ (see the master commutative diagram below for the notation):
\begin{equation*}
\Psi\colon (\varrho^{\ast}_1\cP_1, \varrho^{\ast}_1A_1) \otimes (\mathfrak{z}_{2}^{\ast} \cQ , \mathfrak{z}_{2}^{\ast} \Theta ) \to  (\mathfrak{z}_{1}^{\ast}  Q , \mathfrak{z}_{1}^{\ast} \Theta ) \otimes (\varrho^{\ast}_2\cP_2, \varrho^{\ast}_2A_2)
\end{equation*}

\noindent
fitting into the following commutative diagram:
\[
\begin{tikzcd}
	\varrho^*_{12}\cP_1 \otimes \varrho^*_{23}\cP_1 \otimes \mathfrak{z}^{\ast}_{(3)} Q \arrow[r, "\mu_1 \otimes \Id"] \arrow[d, "\Id \otimes \varrho^*_{23}\Psi"] & \varrho^*_{13}\cP_1 \otimes \mathfrak{z}^{\ast}_{(3)} Q \arrow[dd, "\varrho^*_{13}\Psi"] \\
	\varrho^*_{12}\cP_1 \otimes \mathfrak{z}^{\ast}_{2} Q \otimes \varrho^*_{23}\cP_2 \arrow[d, "\varrho^*_{12}\Psi \otimes \Id"] &  \\
	\mathfrak{z}^{\ast}_{1} Q \otimes \varrho^*_{12}\cP_2 \otimes \varrho^*_{23}\cP_2 \arrow[r, "\Id \otimes \mu_2"] & \mathfrak{z}^{\ast}_{1} Q \otimes \varrho^*_{13}\cP_2
\end{tikzcd}
\]

\noindent
over $Z^{[3]}$, where: 
\begin{equation*}
\varrho_{ij} \colon Z\times_M Z \times_M Z \to Z\times_M Z\, , \qquad i, j = 1,2,3
\end{equation*}

\noindent
denotes the map that projects onto the $i$-th and $j$-th factors.
\end{definition}

\noindent
The significance of the various maps and pull-backs occurring in the previous definition is explained in the following commutative diagram, to which we will refer as the \emph{master commutative diagram}:

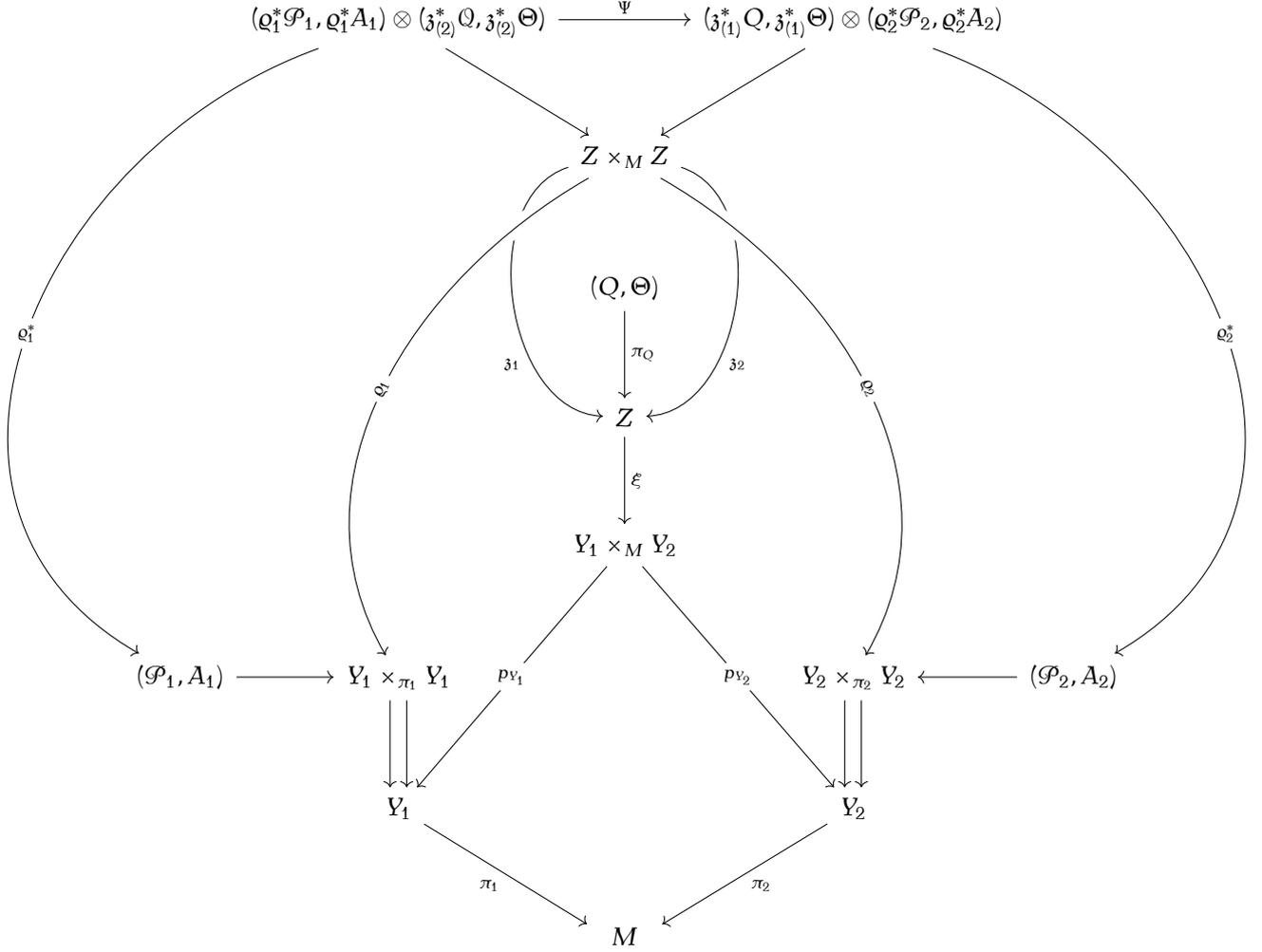
\begin{center}
\hspace*{-2.2cm}
\begin{tikzcd}[row sep=3.5em, column sep=0.1em]
	&  (\varrho^{\ast}_1\cP_1, \varrho^{\ast}_1A_1) \otimes (\mathfrak{z}^{\ast}_{2} \cQ , \mathfrak{z}^{\ast}_{2}\Theta ) \arrow[rr, "\Psi"] \arrow[dr] 
	    \arrow[to=6-1, "{\varrho_1^{\ast}}" description, out=-160, in=150, looseness=1.2]
	& & (\mathfrak{z}^{\ast}_{1} Q , \mathfrak{z}^{\ast}_{1}\Theta ) \otimes (\varrho^{\ast}_2\cP_2, \varrho^{\ast}_2A_2) \arrow[dl] 
	    \arrow[to=6-5, "{\varrho_2^{\ast}}"' description, out=-20, in=30, looseness=1.2] 
	& \\
	& & Z \times_M Z
	    \arrow[dd, bend left=80, "\mathfrak{z}_{2}" pos = 0.65]
	    \arrow[dd, bend right=80, "\mathfrak{z}_{1}"' pos = 0.65]
	    \arrow[to=6-2, "\varrho_1" description, out=-150, in=120, sloped, above, crossing over]
	    \arrow[to=6-4, "\varrho_2" description, out=-30, in=60, sloped, above, crossing over]
	& & \\
	& & (Q,\Theta) \arrow[d, "\pi_Q"] & & \\
	& & Z \arrow[d, "\xi"] & & \\
	& & Y_1 \times_M Y_2 \arrow[ddl, "p_{Y_1}"' description] \arrow[ddr, "p_{Y_2}" description] & & \\
	(\cP_1, A_1) \arrow[r] & Y_1 \times_{\pi_1} Y_1 \arrow[d, shift left=.7ex] \arrow[d, shift right=.7ex] & & Y_2 \times_{\pi_2} Y_2 \arrow[d, shift left=.7ex] \arrow[d, shift right=.7ex] & (\cP_2, A_2) \arrow[l] \\
	& Y_1 \arrow[dr, "\pi_1"', end anchor={[xshift=-1ex]}] & & Y_2 \arrow[dl, "\pi_2", end anchor={[xshift=1ex]}] & \\
	& & M & &
\end{tikzcd}
\end{center}

\noindent
Here we are denoting by: 
\begin{equation*}
\mathfrak{z}_{(i)} \colon Z\times_M Z \times_M Z \to Z \, , \qquad i = 1,2,3
\end{equation*} 

\noindent
the map that projects the \emph{i-th} entry to $Z$. Recall that a stable isomorphism may exist between given bundle gerbes if and only if they have the same Dixmier-Douady class in $H^3(M,\mathbb{Z})$. As explained in detail in \cite{Waldorf}, stable isomorphisms can be naturally composed and admit \emph{higher} isomorphisms, called \emph{2-isomorphisms}, that can in turn also be naturally composed, resulting in bundle gerbes over a fixed manifold forming a 2-category. More precisely, the collection of all bundle gerbes with connective structure on $M$, together with their stable isomorphisms and two-isomorphisms as introduced below, forms a \emph{bigroupoid}, which we denote by $\mathrm{Grb}(M)$. We refer the reader to \cite{MurrayStevenson,Waldorf0,Waldorf,Bunk21} for more details.

\begin{definition}
\label{def:2morphisms}
Let $(\cP_1,A_1,\pi_1,\mu_1)$ and $(\cP_2,A_2,\pi_2,\mu_2)$ be bundle gerbes over $M$. Let $(\cQ,\Theta,\xi,\Psi)$ and $(\cQ^{\prime}, \Theta^{\prime}, \xi^{\prime},\Psi^{\prime})$ be stable isomorphisms from $(\cP_1,A_1,\pi_1,\mu_1)$ to $(\cP_2,A_2,\pi_2,\mu_2)$. A \emph{representative} of a 2-isomorphism from $(\cQ,\Theta,\xi,\Psi)$ to $(\cQ^{\prime}, \Theta^{\prime}, \xi^{\prime},\Psi^{\prime})$ is a pair $(\frw, \beta)$ consisting of:
\begin{itemize}
\item A surjective submersion $\frw\colon W \to Z \times_{Y_1\times_M Y_2} Z^{\prime}$.
\item An isomorphism $\beta \colon (Q,\Theta) \to (Q^{\prime},\Theta^{\prime})$ of principal $\U(1)$-bundles with connection fitting into the following commutative diagram of morphisms of principal $\U(1)$-bundles with connection over $W^{[2]}$:
\begin{center}
\begin{tikzcd}
 (\rho_1\circ \frw^{[2]})^{\ast} (\cP_1 , A_1)\otimes  (\mathfrak{z}_{2} \circ \frw^{[2]})^{\ast} (Q,\Theta) \arrow[r, "\Psi"] \arrow[d, "\Id \otimes \beta"] &  (\mathfrak{z}_{1} \circ \frw^{[2]})^{\ast} (Q,\Theta) \otimes  (\rho_2 \circ \frw^{[2] \prime})^{\ast}(\cP_2 , A_2) \arrow[d, "\beta \otimes \Id"] \\
(\rho_1\circ \frw^{[2]})^{\ast} (\cP_1 , A_1) \otimes (\mathfrak{z}_{2}^{\prime} \circ \frw^{[2] \prime})^{\ast} (Q^{\prime} , \Theta^{\prime}) \arrow[r, "\Psi^{\prime}"] &  (\mathfrak{z}_{1}^{\prime} \circ \frw^{[2] \prime})^{\ast} (Q^{\prime} , \Theta^{\prime}) \otimes (\rho_2\circ \frw^{[2] \prime})^{\ast}(\cP_2 , A_2)
\end{tikzcd}
\end{center}

\noindent 
Here:
\begin{equation*}
Z\times_M Z \xleftarrow{\frw^{[2]}} W\times_{Y_1\times_M Y_2} W \xrightarrow{\frw^{[2]\prime}}  Z^{\prime}\times_M Z^{\prime}
\end{equation*}
are the maps naturally induced by $\frw\colon W \to Z \times_{Y_1\times_M Y_2} Z^{\prime}$ on the fibered product of $W$ with itself. Two such representatives of 2-isomorphism $(\frw, \beta)$ and $(\hat{\frw}, \hat{\beta})$ are said to be \emph{equivalent} if there exists a smooth manifold $X$ with surjective submersions to $W$ and $\widehat{W}$ such that the following diagram commutes:
\begin{center}
\begin{tikzcd}
& X \arrow[dl] \arrow[dr] & \\
W \arrow[dr, "\frw"] & & \widehat{W} \arrow[dl, "\widehat{\frw}"'] \\
& Z \times_{Y_1\times_M Y_2} Z^{\prime} &
\end{tikzcd}
\end{center}
and the pullbacks of the morphisms $\beta$ and $\beta^{\prime}$ to $X$ coincide.
\end{itemize}

\noindent
The \emph{2-isomorphisms} from $(\cQ,\Theta,\xi,\Psi)$ to $(\cQ^{\prime}, \Theta^{\prime}, \xi^{\prime},\Psi^{\prime})$ are the equivalence classes defined by the above equivalence relation on the representatives of 2-isomorphisms.
\end{definition}

\noindent
Let $(\cQ , \Theta , \xi ,\Psi)\colon (\cP_1,A_1,\pi_1,\mu_1) \to (\cP_2,A_2,\pi_2,\mu_2)$ be an isomorphism, and let $b_1$ and $b_2$ be curvings on $(\cP_1,A_1,\pi_1)$ and $(\cP_2,A_2,\pi_2)$, respectively. We say that $(\cQ , \Theta , \xi ,\Psi)$ \emph{maps} or \emph{relates} $b_1$ to $b_2$ if:
\begin{equation}
\label{eq:relationbs}
(p_{Y_2}\circ \xi)^{\ast} b_2 = (p_{Y_1}\circ \xi)^{\ast} b_1 + F_{\Theta} \in \Omega^2(Z)
\end{equation}

\noindent
where $F_{\Theta}\in \Omega^2(Z)$ is the curvature of $\Theta\in \Omega^1(Q)$ as a real 2-form on $Z$. From this equation, it follows that:
\begin{equation*}
(p_{Y_2}\circ \xi)^{\ast} \dd_{Y_2} b_2 = (p_{Y_1}\circ \xi)^{\ast} \dd_{Y_1} b_1
\end{equation*}

\noindent
where $\dd_{Y_1}$ and $\dd_{Y_2}$ denote the exterior differentials on $Y_1$ and $Y_2$, respectively. This, in turn, implies:
\begin{equation*}
(\pi_2 \circ p_{Y_2}\circ \xi)^{\ast} H_{b_2} = (\pi_1 \circ p_{Y_1}\circ \xi)^{\ast} H_{b_1}
\end{equation*}

\noindent
where $H_{b_1} , H_{b_2} \in \Omega^3(M)$ are the curvatures of $b_1$ and $b_2$, respectively. More explicitly, taking $z\in Z$ and set $(y_1 , y_2) := \xi(z)$, for every pair of vectors $V_1 , V_2 \in T_z Z$, Equation \eqref{eq:relationbs} implies:
\begin{equation*}
b_2\vert_{y_2} (\dd_z(p_{Y_2}\circ \xi)V_1 , \dd_z(p_{Y_2}\circ \xi)V_2) = b_1\vert_{y_1} (\dd_z(p_{Y_1}\circ \xi)V_1 , \dd_z(p_{Y_1}\circ \xi)V_2) + F_{\Theta}\vert_z (V_1 , V_2)
\end{equation*}

\noindent
This equation explicitly relates the evaluation of $b_2$ at $y_2 \in Y_2$ to the evaluation of $b_1$ at $y_1 \in Y_1$ via an object, the curvature $F_{\Theta}$, evaluated at $z\in Z$, where the map $Z \to M$ plays the role of a \emph{common refinement} of $Y_1 \to M$ and $Y_2 \to M$ (they belong to a generalized version of coverings of $M$). For every bundle gerbe $(\cP,A,\pi,\mu)$, we define its \emph{gauge 2-group} $\cG(\cP,A,\pi,\mu)$ as follows: first, a 2-group is a monoidal groupoid in which each object has an inverse with respect to the monoidal structure.
Then, the groupoid underlying the 2-group $\cG(\cP,A,\pi,\mu)$ is the automorphism groupoid of $(\cG,A,\pi)$ in $\mathrm{Grb}(M)$. This groupoid carries a canonical monoidal structure, which arises from the composition in the 2-groupoid $\mathrm{Grb}(M)$. Hence, it is naturally a 2-group. By construction, its monoidal structure turns out to be symmetric (i.e.~homotopy-coherently commutative).

\begin{example}
Let $(\cQ,\Theta,\xi,\Psi) \in \cG(\cP,A,\pi,\mu)$, and set $Z = Y\times_M Y$ with the submersion $\xi\colon Z\to Y\times_M Y$ given by the identity. I this case the master commutative diagram reduces to:
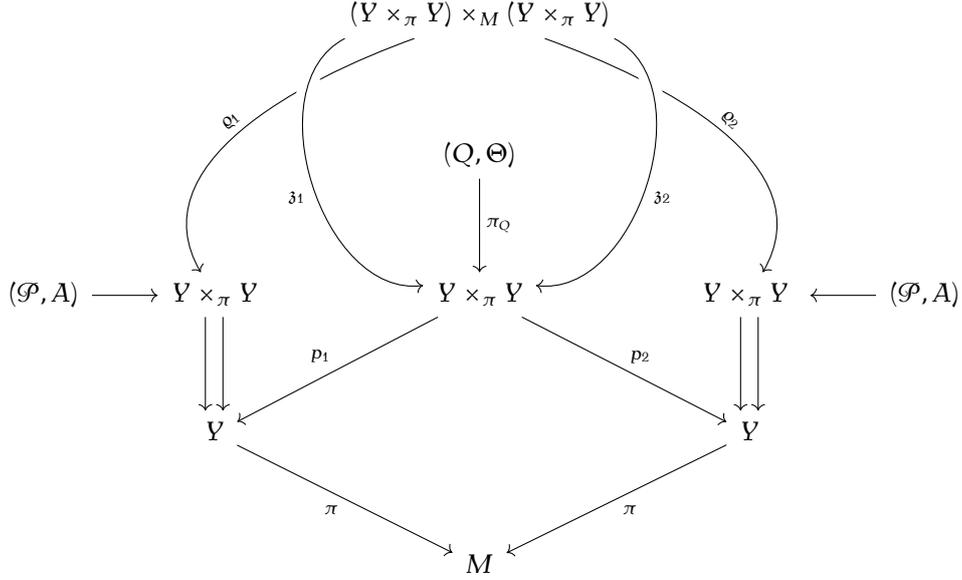
\begin{center}
\begin{tikzcd}[row sep=3.5em, column sep=2.5em]
	& & (Y \times_{\pi} Y) \times_M (Y \times_{\pi} Y)
	    \arrow[to=3-2, "\varrho_1", out=-160, in=120, sloped, above]
	    \arrow[to=3-4, "\varrho_2", out=-20, in=60, sloped]
	    \arrow[dd, "\mathfrak{z}_{2}", bend left=80, crossing over]
	    \arrow[dd, "\mathfrak{z}_{1}"', bend right=80, crossing over]
	& & \\
	& & (Q,\Theta) \arrow[d, "\pi_Q"]  & & \\
	(\cP, A) \arrow[r] & Y\times_{\pi} Y  \arrow[d, shift left=0.7ex] \arrow[d, shift right=0.7ex] & Y \times_{\pi} Y \arrow[dl, "p_1"'] \arrow[dr, "p_2"] & Y\times_{\pi} Y \ \arrow[d, shift left=0.7ex] \arrow[d, shift right=0.7ex] & (\cP, A)  \arrow[l] \\
	& Y \arrow[dr, "\pi"'] & & Y \arrow[dl, "\pi"] & \\
	& & M & &
\end{tikzcd}
\end{center}

\noindent
Similarly, Equation \eqref{eq:relationbs} reduces to:
\begin{equation}
\label{eq:brelationsexample}
p_{2}^{\ast} b_2 = p_{1}^{\ast} b_1 + F_{\Theta} \in \Omega^2(Y \times_M Y)
\end{equation} 

\noindent
Evaluating this expression at a point $(y_1 , y_2) \in Y\times_M Y$, we obtain:
\begin{eqnarray*}
 b_2\vert_{y_2}(\dd_{(y_1,y_2)} p_2 V_1, \dd_{(y_1,y_2)} p_2 V_2) = b_1\vert_{y_1}(\dd_{(y_1,y_2)} p_1 V_1, \dd_{(y_1,y_2)} p_1 V_2) + F_{\Theta}\vert_{(y_1,y_2)}(V_1,V_2)    
\end{eqnarray*}

\noindent
The points $y_1, y_2 \in Y$ belong to the same fiber of $\pi\colon Y\to M$ but may be different a priori. Hence, we cannot evaluate the same \emph{arguments} in the previous equation. Let $s\colon U\subset M \to Y$ be a local section of $\pi\colon Y\to M$, where $U$ is a contractible open subset of $M$. This defines a local section $s^{[2]}\colon U\subset M \to Y\times_M Y$ of the submersion $Y\times_M Y \to M$ which we use to pull Equation \eqref{eq:brelationsexample} back to $M$, obtaining:
\begin{eqnarray*}
s^{[2]\ast} p_{2}^{\ast} b_2 = s^{[2]\ast} p_{1}^{\ast} b_1 + s^{[2]\ast} F_{\Theta} \in \Omega^2(U)
\end{eqnarray*}

\noindent
which is equivalent to:
\begin{eqnarray*}
s^{\ast}  b_2 = s^{\ast} b_1 + s^{[2]\ast} F_{\Theta} \in \Omega^2(U)
\end{eqnarray*}

\noindent
Since $F_{\Theta} \in \Omega^2(Z)$ is closed and $U$ is contractible, there exists a local one-form $\theta\in \Omega^1(U)$ such that:
\begin{equation}
\label{eq:localrelationb2}
s^{\ast}  b_2 = s^{\ast} b_1 + \dd \theta \in \Omega^2(U)
\end{equation}

\noindent
Hence, the gauge transformations of $(\cP,A,\pi,\mu)$ relating $b_1$ and $b_2$ reproduce the celebrated local \emph{gauge transformations} of the $b$-field. 
\end{example}
 
\begin{remark}
Denote by $\mathrm{Bun}(M)$ the symmetric monoidal category of $\U(1)$ bundles with connection defined on $M$. Then, there is an equivalence of 2-groups between $\cG(\cP,Y,A)$ and $\mathrm{Bun}(M)$.    
\end{remark}

Following~\cite{BunkMüllerSzabo,BunkShahbazi}, we define automorphisms of bundle gerbes that cover possibly non-trivial diffeomorphisms of the base manifold:

\begin{definition}
An \emph{automorphism} of $(\cP,A,\pi,\mu)$ is a tuple of the form $(\cQ,\Theta, \xi, \Psi,f)$, where $f\colon M \to M$ is a diffeomorphism isotopic to the identity and $(\cQ,\Theta, \xi , \Psi)\colon (\cP,A,\pi,\mu) \to (\cP^f,A^f,\pi^f)$ is a stable isomorphism between $(\cP,A,\pi,\mu)$ and its pull-back $(\cP^f,A^f,\pi^f)$ by $f\colon M\to M$.  
\end{definition}

\begin{definition}
The \emph{automorphism 2-group} $\Aut(\cP,A,\pi,\mu)$ of $(\cP,A,\pi,\mu)$ is the groupoid whose objects are automorphisms of $(\cP,A,\pi,\mu)$, and whose arrows are defined as follows: given two objects $(\cQ,\Theta, \xi, \Psi,f)$ and $(\cQ',\Theta', \xi', \Psi',f')$, there exist no morphisms between them if $f \neq f'$. If $f = f'$, then the morphisms $(\cQ,\Theta, \xi, \Psi,f) \longrightarrow (\cQ',\Theta', \xi', \Psi',f)$ are exactly the 2-isomorphisms $(\cQ,\Theta, \xi, \Psi) \longrightarrow (\cQ',\Theta', \xi', \Psi')$ in $\mathrm{Grb}(M)$.
\end{definition}

\noindent
Analogously to~\cite{BunkMüllerSzabo} (which contains this construction just without connective structures, but working with smooth 2-groups), we have a short exact sequence of 2-groups:
\begin{equation*}
1 \Rightarrow \cG(\cP,A,\pi,\mu) \Rightarrow \Aut(\cP,A,\pi,\mu) \Rightarrow \Diff_o(M) \Rightarrow 1 
\end{equation*}
 
\noindent
where $\Diff_o(M)$ denotes the diffeomorphisms of $M$ isotopic to the identity, understood as a 2-group with only identity arrows.\medskip

A key ingredient in this paper is the action of a bundle gerbe automorphism on the set of curvings. We first prove the existence of a canonical action of this form in full generality, i.e.~without any simplifying assumptions on the acting 1- and 2-morphisms of bundle gerbes. Subsequently, we specialize to the case where the diffeomorphism $f$ covered by a bundle gerbe automorphism admits a lift to $Y$; this suffices for our purposes, and yields simplified formulas.\medskip

Let $(\cG, A, \pi)$ be a bundle gerbe (with connective structure), and let $(\cQ, \Theta, \xi, \Psi, f)$ be an automorphism of $(\cG, A, \pi)$ covering a diffeomorphism $f\colon M \to M$. The pullback surjective submersion $\pi^f \colon Y^f \to M$ fits into a commutative square:
\begin{equation}
    \begin{tikzcd}
        Y^f \ar[r, "\widehat{f}"] \ar[d, "\pi^f"']
        & Y \ar[d, "\pi"]
        \\
        M \ar[r, "f"']
        & M
    \end{tikzcd}
\end{equation}
Suppose that we are given a curving $b \in \Omega^2(Y)$ for the bundle gerbe $(\cG, A, \pi)$. It induces a curving $\widehat{f}^* b \in \Omega^2(Y^f)$ for the pullback bundle gerbe $(\cG^f, A^f, \pi^f)$, and we can further pull back this 2-form along the map $p_{Y^f} \circ \xi \colon Z \to Y^f$. We observe that if we had a curving $b' \in \Omega^2(Y)$ on the original bundle gerbe $(\cG, A, \pi,\mu)$ such that $(\cQ, \Theta, \xi, \Psi, f)$ was a morphism of bundle gerbes with connection, then $b'$ would have to satisfy the equation:
\begin{equation}
    (p_{Y^f} \circ \xi)^* \widehat{f}^* b
    = (p_Y \circ \xi)^* b' + F_\Theta
\end{equation}
of $2$-forms on $Z$.
We now show that this formula can indeed be used to induce a unique curving $b'$ on the bundle gerbe $(\cG, A, \pi, \mu)$ from the data of $b$ and $(\cQ, \Theta, \xi, \Psi, f)$, by proving that the form $(p_{Y^f} \circ \xi)^* \widehat{f}^* b - F_\Theta$ descends along the composed morphism $Z \to Y \times_M Y^f \to Y$:

\begin{prop}
\label{prop:general action on curvings}
    Consider the surjective submersion $\zeta = p_Y \circ \xi \colon Z \to Y$.
    Let $\delta_Y$ denote the \v{C}ech differential acting on differential forms on the \v{C}ech nerve of $\zeta$.
    The following statements hold true:
    \begin{enumerate}
        \item In the above set-up, we have that:
            \begin{equation}
                \delta_Y \big( (p_{Y^f} \circ \xi)^* \widehat{f}^* b - F_\Theta \big) = 0
            \end{equation}

        \item There exists a unique $2$-form $b' \in \Omega^2(Y)$ such that, on $Z$:
            \begin{equation}
                (p_{Y^f} \circ \xi)^* \widehat{f}^* b
                = (p_{Y^f} \circ \xi)^* b' + F_\Theta
            \end{equation}

        \item The $2$-form $b'$ is a curving for $(\cG, A, \pi)$, whose curvature is:
        \begin{equation}
            H_{b'} = f^* H_b
        \end{equation}
    \end{enumerate}
\end{prop}

\begin{proof}
For (1):
The differential form $\delta ((p_{Y^f} \circ \xi)^* \widehat{f}^* b - F_\Theta)$ is defined on the fiber product $Z \times_Y Z$.
There is a commutative square of smooth maps:
\begin{equation}
    \begin{tikzcd}[column sep = 2cm]
        Z \times_Y Z \ar[r, hookrightarrow] \ar[d]
        & Z \times_M Z \ar[d]
        \\
        (Y^f)^{[2]} \times_M Y \ar[r, "\mathrm{Id} \times \Delta_Y"] \ar[d]
        & (Y^f)^{[2]} \times_M Y^{[2]} \ar[d]
        \\
        Y \ar[r, "\Delta_Y"]
        & Y^{[2]}
    \end{tikzcd}
\end{equation}
Denoting the projection $Z \times_Y Z \to Z$ to the $i$-th factor by $q_i$, we compute:
\begin{align}
    \delta_Y \big( (p_{Y^f} \circ \xi)^* \widehat{f}^* b - F_\Theta \big)
    &= q_2^* \big( (p_{Y^f} \circ \xi)^* \widehat{f}^* b - F_\Theta \big)
    - q_1^* \big( (p_{Y^f} \circ \xi)^* \widehat{f}^* b - F_\Theta \big)
    \\
    &= \zeta^* \delta_M \widehat{f}^* b - (\delta_M F_\Theta)_{|Z \times_Y Z}
    \\
    &= \zeta^* \delta_M \widehat{f}^* b - (p_{Y^f}^{[2]*} F_{\cQ^f} - \underbrace{\Delta_Y^* p_Y^{[2]*} F_\cQ}_{=0})
    \\
    &= \zeta^* \delta_M \widehat{f}^* b - p_{Y^f}^{[2]*} F_{\cQ^f}
    \\
    &= 0
\end{align} 
For (2): This is now a direct consequence of part~(1) and the fact that for any surjective submersion $h \colon N \to N'$ between manifolds, the induced map $h^* \colon \Omega^k(N') \to \Omega^k(N)$ is injective.

For (3): It remains to show that the $2$-form $b'$ satisfies $\delta b' = F_A$ on $Y^{[2]}$.
To that end, we again consider the surjective submersion $\zeta \colon Z \to Y$.
This induces injective maps $\zeta^* \colon \Omega^2(Y) \to \Omega^2(Z)$ and $\zeta^{[2]*} \colon \Omega^2(Y^{[2]}) \to \Omega^2(Z^{[2]})$, and we have a commutative square:
\begin{equation}
    \begin{tikzcd}[column sep = 1.5cm, row sep = 1cm]
        Z^{[2]} \ar[r, shift left = 0.05cm, "p_1"] \ar[r, shift left = -0.05cm, "p_2"'] \ar[d, "\zeta^{[2]}"]
        & Z \ar[d, "\zeta"]
        \\
        Y^{[2]} \ar[r, shift left = 0.05cm, "p_1"] \ar[r, shift left = -0.05cm, "p_2"']
        & Y
    \end{tikzcd}
\end{equation}
Thus, it suffices to show that:
\begin{equation}
    \zeta^{[2]*} (p_2^*b' - p_1^*b') = \zeta^{[2]*} F_A
\end{equation}
on $Z^{[2]}$.
We compute:
\begin{align}
    \zeta^{[2]*} (p_2^*b' - p_1^*b')
    &= p_2^* \zeta^* b' - p_1^* \zeta^*b'
    \\
    &= p_2^* \big( (p_{Y^f} \circ \xi)^* \widehat{f}^* b - F_\Theta \big)
        - p_1^* \big( (p_{Y^f} \circ \xi)^* \widehat{f}^* b - F_\Theta \big)
    \\
    &= (p_{Y^f} \circ \xi)^{[2]*} F_{A^f} + (p_2^*F_\Theta - p_1^* F_\Theta)
    \\
    &= \zeta^{[2]*} F_A
\end{align}
Here we have used the commutativity of the above diagram in the first identity, claim~(1) in the second, in the third identity that $\widehat{f}^*b$ is a curving on the pullback bundle gerbe $(\cG^f, A^f, \pi^f)$, and in the last that $(\cQ, \Theta, \xi, \Psi)$ is a morphism of bundle gerbes with connective structure.
Finally, the claim regarding the curvatures follows because isomorphic bundle gerbes with connection have the same curvature, and the curvature of the pullback bundle gerbe $(\cG^f, A^f, \pi^f)$ with curving $\widehat{f}^*b$ is $f^*H_b$.
That completes the proof.
\end{proof}

We readily observe that if $(\cQ, \Theta, \xi, \Psi, f) \to (\cQ', \Theta', \xi', \Psi', f)$ is any isomorphism of bundle gerbe automorphisms (necessarily over the same diffeomorphisms), then, given a curving $b$ on $(\cG, A, \pi)$, these two automorphisms induce the same new curvings $b'$ on $(\cG, A, \pi)$. Indeed, this follows because only the field strength $F_\Theta$ enters in the form which descends to $Y$ to give the new curving $b'$, and isomorphic bundle gerbe automorphisms have identical field strength $2$-forms once pulled back to the total space $W$ of a representative of the isomorphism. By a similar argument, we also infer that the induced curving $b'$ is compatible with the composition of automorphisms.

\begin{thm}
\label{thm:action}
    There is a canonical \emph{right} action of $\Aut(\cP,A,\pi,\mu)$ on the affine space of curvings $\cC(\cP,A,\pi,\mu)$ on $(\cP,A,\pi,\mu)$ via Proposition~\ref{prop:general action on curvings}:
    \begin{equation*}
        \cC(\cP,A,\pi,\mu) \times \Aut(\cP,A,\pi,\mu) \to \cC(\cP,A,\pi,\mu)
    \end{equation*}
\end{thm}
This is an action by a $2$-group on a set, and so necessarily factors through an ordinary group action:
\begin{equation*}
\cC(\cP,A,\pi,\mu)\times \pi_0\Aut(\cP,A,\pi,\mu)\to \cC(\cP,A,\pi,\mu)
\end{equation*}
by the group of connected components (i.e.~isomorphism classes) of $\Aut(\cP,A,\pi,\mu)$. That is also consistent with the discussion above of the action of isomorphic automorphisms.\medskip

Having thus found the action of the automorphism 2-group on the set of curvings on a given bundle gerbe, we now derive simplified presentations of this action in cases where the diffeomorphism $f$ admits a lift to the total space $Y$ of the surjective submersion defining which is part of the data of the bundle gerbe. We will use this simplified case in the next subsection to propose a notion of \emph{smooth family} of automorphisms of a bundle gerbe. In particular, at least in open neighborhoods in the manifold, this allows us to reduce the general description of the action on curvings to formulas familiar from the supergravity literature. To that end, we observe that the pull-back $(\cP^f,A^f,\pi^f)$ of $(\cP,A,\pi,\mu)$ by a diffeomorphism $f\colon M\to M$ fits into the following commutative diagram.  
\begin{center}
\begin{tikzcd}[row sep=2.5em, column sep=3.5em]
((p_Y^{[2]}\circ i^{[2]})^{\ast} \cP, (p_Y^{[2]}\circ i^{[2]})^{\ast}A) \arrow[r] \arrow[d] & (p_Y^{[2]\ast}\cP, p_Y^{[2]\ast}A) \arrow[r] \arrow[d] & (\cP, A) \arrow[d] \\
Y \times_{f^{-1} \circ \pi } Y \arrow[r, "i^{[2]}"] \arrow[d, shift left=0.7ex] \arrow[d, shift right=0.7ex] & f^*Y \times_{p_M} f^*Y \arrow[r, "p_Y^{[2]}"] \arrow[d, shift left=0.7ex] \arrow[d, shift right=0.7ex] & Y \times_{\pi} Y \arrow[d, shift left=0.7ex] \arrow[d, shift right=0.7ex] \\
Y \arrow[r, "i"] \arrow[dr, "f^{-1} \circ \pi "'] & f^*Y \arrow[r, "p_Y"] \arrow[d, "p_M"'] & Y \arrow[d, "\pi"] \\
& M \arrow[r, "f"] & M
\end{tikzcd}
\end{center}

\noindent
In particular, this establishes the surjective submersion $f^{-1} \circ \pi \colon Y \to M$ as a representative of the pullback of $\pi \colon Y \to M$ along $f$.

\noindent
Hence $(\cP^f,A^f) := ((p_Y^{[2]}\circ i^{[2]})^{\ast}\cP , (p_Y^{[2]}\circ i^{[2]})^{\ast}A)$ by definition and, since $p_Y \circ i$ is the identity on $Y$, we also have a canonical isomorphism:
\begin{equation*}
((p_Y^{[2]}\circ i^{[2]})^{\ast} \cP, (p_Y^{[2]}\circ i^{[2]})^{\ast}A) = (p_Y^{[2]\ast}\cP, p_Y^{[2]\ast}A)
\end{equation*}

\noindent
Furthermore, for every element $ (y_1 ,y_2) \in Y \times_{f^{-1} \circ \pi } Y$ we have:
\begin{equation*}
((p_Y^{[2]}\circ i^{[2]})^{\ast} \cP, (p_Y^{[2]}\circ i^{[2]})^{\ast}A)\vert_{(y_1 , y_2)} =    (  \cP, A)\vert_{(p_Y^{[2]}\circ i^{[2]}) (y_1 , y_2)} =  (  \cP, A)\vert_{(y_1 , y_2)}
\end{equation*}

\noindent
Hence, by the previous diagram together with the fact that $Y\times_{f^{-1} \circ \pi } Y = Y\times_{\pi} Y$ as manifolds (note that their maps to $M$ differ), we can consider the pull-back bundle gerbe $(\cP^f,A^f , \pi^f)$ as a bundle gerbe of the form:
\begin{equation*}
(\cP, A) \to Y\times_{\pi} Y \rightrightarrows Y \xrightarrow{f^{-1} \circ \pi} M
\end{equation*}

\noindent
Using this model for the pull-back bundle gerbe, we have automorphisms $(\cQ,\Theta, \xi, \Psi,f)\in \Aut(\cP,A,\pi,\mu)$ fitting into the following commutative diagram:
 
\begin{center}
\begin{tikzcd}[row sep=3.5em, column sep=2.5em]
	& & (Y \times^{f}_{\pi} Y) \times_M (Y \times^{f}_{\pi} Y)
	    \arrow[to=3-2, "\varrho_1", out=-160, in=120, sloped, above]
	    \arrow[to=3-4, "\varrho_2", out=-20, in=60, sloped]
	    \arrow[dd, "t", bend left=80, crossing over]
	    \arrow[dd, "s"', bend right=80, crossing over]
	& & \\
	& & (Q,\Theta) \arrow[d] & & \\
	(\cP, A) \arrow[r] & Y\times_{\pi} Y  \arrow[d, shift left=0.7ex] \arrow[d, shift right=0.7ex] & Y \times^{f}_{\pi} Y \arrow[dl, "p_1"'] \arrow[dr, "p_2"] & Y\times_{f^{-1} \circ \pi } Y \ \arrow[d, shift left=0.7ex] \arrow[d, shift right=0.7ex] & (\cP, A)  \arrow[l] \\
	& Y \arrow[dr, "\pi"'] & & Y \arrow[dl, "f^{-1} \circ \pi "] & \\
	& & M & &
\end{tikzcd}
\end{center}

\noindent
Here, the surjective submersion $Y \times^{f}_{\pi} Y \to M$ is defined as the fibered product of the surjective submersions $\pi \colon Y \to M$ and $f^{-1} \circ \pi  \colon Y \to M$, that is:
\begin{equation*}
Y \times^{f}_{\pi} Y = \left\{ (y_1 , y_2) \in Y\times Y \,\, \vert\,\, f(\pi(y_1)) = \pi(y_2)\right\}
\end{equation*}

\noindent
Note that, for simplicity, here we have set $Z = Y \times^{f}_{\pi} Y$ and taken the surjective submersion $\xi\colon Z \to Y \times^{f}_{\pi} Y$ to be the identity map. Every automorphism of $(\cP,A,\pi,\mu)$ is of this form up to 2-isomorphism~\cite{Waldorf0,Waldorf}. By applying Equation \eqref{eq:relationbs} to this specific case, we conclude that two curvings $b_1$ and $b_2$ on $(\cP,A,\pi,\mu)$ are related by an isomorphism $(\cQ,\Theta, \Psi,f) \colon (\cP, A, \pi, \mu) \to (\cP, A, \pi, \mu)$ if and only if:
\begin{equation} 
\label{eq:brelations2}
p_{2}^{\ast} b_2 = p_{1}^{\ast} b_1 + F_{\Theta} \in \Omega^2(Y \times^{f}_{\pi} Y)
\end{equation}

\noindent
Let $s\colon U\subset M \to Y$ be a local section of $\pi \colon Y\to M$ defined on a contractible open subset $U\subset M$. This canonically induces a local section of the surjective submersion $f \circ \pi \colon Y \to M$, which is defined over the open subset $f^{-1}(U)$ of $M$.
Restricting to the case where the diffeomorphism $f$ preserves the open subset $U$, it follows that the two local sections and $f \circ s$ are now defined over the same open subset $U = f^{-1}(U)$, and we obtain a canonical section:
\begin{equation*}
s_f^{[2]}   \colon U \to Y\times^f_{\pi} Y\, , \qquad m\mapsto (s(m), f(s(m)))
\end{equation*}

\noindent
We would like to emphasize two observations at this point.
First, pulling-back Equation \eqref{eq:brelations2} by $s_f^{[2]}   \colon U \to Y\times^f_{\pi} Y$ we obtain:
\begin{equation*}
f^{\ast}(s^{\ast}b_2) = s^{\ast}b_1 + \dd \theta\in \Omega^2(U)    
\end{equation*}

\noindent
for a local one-form $\theta\in \Omega^1(U)$. Hence, we recover the well-known local transformation of a $b$-field under a combined local diffeomorphism and gauge transformation. In particular, if $b_1$ and $b_2$ are curvings on $(\cP,A,\pi,\mu)$ related by $(Q,\Theta,\Psi,f)$ then it follows that $f^{\ast} H_{b_2} = H_{b_1}$, as expected.
Second, this generally accepted formula holds because we have locally trivialized a global geometric structure which holds much more information than just the 2-form in this formula, and we have had to restrict to diffeomorphisms that preserve the support of this local trivialization, which in general singles out only a small subgroup of the diffeomorphisms of $M$.
The global perspective via (bundle) gerbes we employ and advocate here alleviates both of these restrictions simultaneously.

We now suppose that $f$ admits a lift , that is, assume that there exists a smooth map $\bar{f}\colon Y\to Y$ fitting into the following commutative diagram:
\begin{center}
\begin{tikzcd}[row sep=3.5em, column sep=2.5em]
Y \arrow[r, "\bar{f}"] \arrow[d, "\pi"'] & Y \arrow[d, "\pi"] \\
M \arrow[r, "f"] & M
\end{tikzcd}
\end{center}

\begin{remark}
    The existence of a lift $\bar{f}\colon Y\to Y$ as above is generally obstructed \cite{Meigniez2002}. However, every bundle gerbe is (stably) isomorphic to a bundle gerbe where $ \pi \colon Y \to M$ is a $\mathrm{PU}(H)$-principal bundle, for $\mathrm{PU}(H)$ the projective unitary group on an infinite-dimensional Hilbert space. Indeed, $B\mathrm{PU}(H)$ is a model for the Eilenberg-MacLane space $K(\mathbb{Z},3)$, and so for every isomorphism class of gerbes there is a corresponding isomorphism class of $\mathrm{PU(H)}$-bundles. To each $\mathrm{PU}(H)$-bundle there is an associated \textit{lifting bundle gerbe}, which presents the same class in $H^3(M, \mathbb{Z})$. Thus, if we allow the total space $Y$ of the surjective submersion to be a possibly $\infty$-dimensional manifold, then every gerbe on $M$ is isomorphic to one for which any diffeomorphism $f \in \mathrm{Diff}(M)$ connected to the identity admits a lift to $Y$. If the gerbe has torsion class in $H^3(M, \mathbb{Z})$, then we may even replace $\mathrm{PU}(H)$ by the finite-dimensional Lie group $\mathrm{PU}(n)$ for some $n \in \mathbb{N}$ (see, for instance,~\cite{Karoubi}).
\end{remark}

We will assume the existence of a lift $\bar{f}$ from this point on. This restricts our treatment to a subclass of (families of) bundle gerbe automorphisms; however, as we will see below, this already leads to a rich theory of higher geometric flows on gerbes. For a full treatment of smooth families of gerbe automorphisms (and substantially more generally morphisms of higher geometric structures on manifolds), we refer the reader to~\cite{BunkShahbazi}. On the other hand, there are cases, for instance, when $\pi \colon Y\to M$ is a universal cover, for which this assumption does not imply any loss of generality, since every diffeomorphism connected to the identity admits a lift in this case.\medskip 

Assuming the existence of a lift $\bar{f} \colon Y \to Y$ of the diffeomorphism $f \colon M \to M$, we propose a more convenient representative for every automorphism \emph{covering} $f$. Using the notation $(\cQ,\Theta, \xi, \Psi,f)\in \Aut(\cP,A,\pi,\mu)$ introduced above for the automorphisms of $(\cP,A,\pi,\mu)$, instead of taking $\xi$ to be the identity as in the previous diagram, we set:
\begin{eqnarray*}
\xi \colon Z = Y\times_{\pi} Y\to Y\times^{f}_{\pi} Y\, , \qquad (y_1 , y_2) \mapsto (y_1 , \bar{f}(y_2))
\end{eqnarray*}

\noindent
With this choice we obtain the following commutative diagram for $(\cQ,\Theta, \xi, \Psi,f)$:

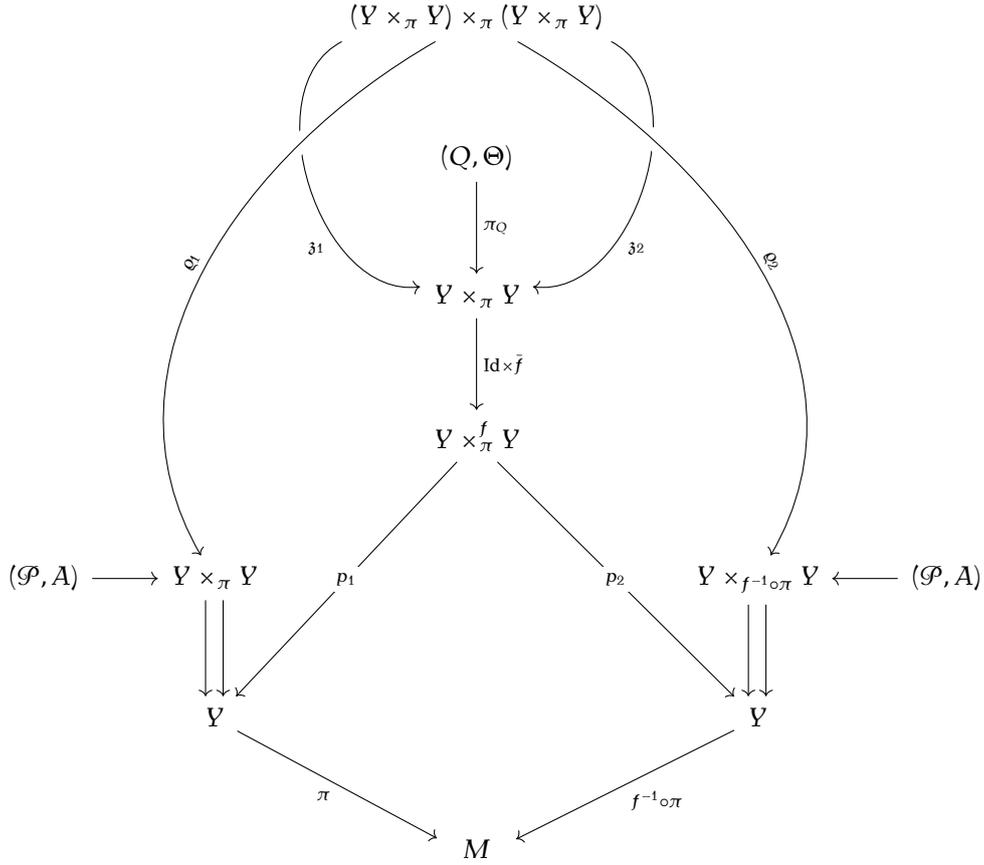
\begin{center}
\begin{tikzcd}[row sep=3.5em, column sep=2.5em]
	& & (Y\times_{\pi} Y) \times_{\pi} (Y\times_{\pi} Y)
	    \arrow[dd, bend left=80, "\mathfrak{z}_{2}" pos = 0.65]
	    \arrow[dd, bend right=80, "\mathfrak{z}_{1}"' pos = 0.65]
	    \arrow[to=5-2, "\varrho_1", out=-150, in=120, sloped, above, crossing over]
	    \arrow[to=5-4, "\varrho_2", out=-30, in=60, sloped, above, crossing over]
	& & \\
	& & (Q,\Theta) \arrow[d, "\pi_Q"] & & \\
	& & Y\times_{\pi} Y \arrow[d, "\mathrm{Id}\times \bar{f}"] & & \\
	& & Y \times_{\pi}^f Y \arrow[ddl, "p_{1}"' description] \arrow[ddr, "p_{2}" description] & & \\
	(\cP, A) \arrow[r] & Y \times_{\pi} Y \arrow[d, shift left=.7ex] \arrow[d, shift right=.7ex] & & Y \times_{f^{-1} \circ \pi } Y \arrow[d, shift left=.7ex] \arrow[d, shift right=.7ex] & (\cP, A) \arrow[l] \\
	& Y \arrow[dr, "\pi"', end anchor={[xshift=-1ex]}] & & Y \arrow[dl, "f^{-1} \circ \pi ", end anchor={[xshift=1ex]}] & \\
	& & M & &
\end{tikzcd}
\end{center}

\noindent
Note that for automorphisms $(\cQ,\Theta, \xi, \Psi,f)$ fitting into the previous diagram, the map $\xi$ is canonically determined by the lift $\bar{f}$. Hence, assuming that the lift $\bar{f}$ is given, we can omit $\xi$ from the notation and simply denote such an automorphism by the tuple $(\cQ,\Theta, \Psi,f)$. Equation \eqref{eq:relationbs} gives the following formula for curvings $b_1$ and $b_2$ that are related by such an automorphism:
\begin{equation*}
    (p_2\circ (\mathrm{Id}\times\bar{f}))^{\ast} b_2 = (p_1\circ (\mathrm{Id}\times\bar{f}))^{\ast} b_1 + F_{\Theta} \in \Omega^2(Y\times_{\pi} Y).
\end{equation*}

\noindent
Note that this is an equation for 2-forms defined on the standard fiber product $Y\times_{\pi} Y$, in contrast to Equation \eqref{eq:brelations2}, which is an equation for 2-forms on the restricted fiber product $Y\times_{\pi}^f Y$. Evaluating this equation at a diagonal point $(y,y) \in Y \times_{\pi} Y$, we obtain:
\begin{equation}
\label{eq:relationbY}
\bar{f}^{\ast} b_2  = b_1 + \Delta_{\pi}^{\ast} F_{\Theta} \in \Omega^2(Y),
\end{equation}

\noindent
as an equation for 2-forms on $Y$, where $\Delta_{\pi} \colon Y \to Y\times_{\pi} Y$ denotes the standard diagonal embedding. 

We now consider a local section $s\colon U\to Y$ defined on a contractible open set $U$ preserved by $f$. From this, we define a new local section:
\begin{equation*}
s_f := \bar{f}^{-1} \circ s\circ f \colon U \to Y.
\end{equation*}

\noindent
Pulling back Equation \eqref{eq:relationbY} by $s_f \colon U \to Y$, we obtain:
\begin{equation*}
    f^{\ast} s^{\ast}b_2 = s_f^{\ast}b_1 + s^{\ast}_f \Delta_{\pi}^{\ast} F_{\Theta} \in \Omega^2(U).
\end{equation*}

\noindent
On the other hand, for every $m\in U$, $s(m)$ and $s_f(m) = \bar{f}^{-1}s(f(m))$ lie in the same fiber of $\pi$. Therefore, the difference of the curving $b_1$ evaluated at these two points is governed by the curvature $F_A$ of the bundle gerbe connection:
\begin{equation*}
s_f^{\ast}b_1 \vert_m = b_1\vert_{\bar{f}^{-1}s(f(m))} = b_1\vert_{s(m)} + F_{A}\vert_{(s(m),\bar{f}^{-1}s(f(m)))}.
\end{equation*}

\noindent
Substituting this back into our previous expression, we find that $f^* s^* b_2$ and $s^* b_1$ differ only by pullbacks of the closed 2-forms $F_A$ and $F_\Theta$. Because the open set $U$ is contractible, these closed forms are exact on $U$, and hence, as expected, we conclude:
\begin{equation*}
    f^{\ast} s^{\ast}b_2 = s^{\ast}b_1 + \dd \theta
\end{equation*}
for a certain local 1-form $\theta \in \Omega^1(U)$.

\begin{remark}
We conclude this section by remarking that a diffeomorphism $f$ isotopic to the identity on $M$ admits a lift $\bar{f}$ along the map $\pi \colon Y \to M$ in many interesting cases, including when $Y$ is any principal bundle on $M$ (or, for instance, the diffeological based path fibration of $M$). This is the case, for example, whenever $Y \to M$ is the universal cover of $M$. If $\pi \colon Y \to M$ is a surjective submersion and $H^3(Y,\mathbb{Z}) = 0$, then any bundle gerbe on $M$ can be defined, up to 1-isomorphism, using the surjective submersion $\pi$: indeed, since bundle gerbes satisfy descent~\cite{Nikolaus}, any gerbe on $M$ can be written as bundle gerbe descent data for $\pi$, but if $H^3(Y,\mathbb{Z})$ is trivial, the gerbe itself can be trivialized over $Y$. In that case, the remaining components of a bundle gerbe descent datum over $\pi$ are the same as the constituents of a bundle gerbe on $M$ defined with respect to $\pi$. In particular, for any 3-manifold whose universal cover is not a 3-sphere, the universal cover admits a surjective submersion (even principal bundle) satisfying the conditions above. Analogous statements hold for $n$-manifolds whose universal cover is, for instance, $\mathbb{S}^n$, purely by cohomological reasons.
\end{remark}

Summarizing this entire categorical framework, the abstract 2-group $\Aut(\cP,A,\pi,\mu)$ and its canonical coupled action on the contractible set of Riemannian metrics and the affine space of curvings completely and rigorously encode the familiar classical symmetries of the generalized Ricci flow. If we evaluate this structure on a contractible local coordinate patch $U \subset M$ where the abstract curvings reduce to local $b$-field potentials $b \in \Omega^2(U)$, the categorical data translates directly into explicit tensor manipulation rules:
\begin{itemize}
\item An object in $\Aut(\cP,A,\pi,\mu)$ covering a diffeomorphism $f \colon M \to M$ locally yields a 1-form gauge parameter $\theta \in \Omega^1(U)$, derived from the connection $\Theta$ on the stable isomorphism. Its abstract action on a curving $b_2$ to produce a new curving $b_1$ exactly reproduces the classical supergravity formula for a diffeomorphism coupled with a $b$-field gauge transformation:
\begin{equation*}
b_1 = f^* b_2 - d\theta.
\end{equation*}

\noindent
for a one-form $\theta\in \Omega^1(M)$.

\item A 2-morphism in $\Aut(\cP,A,\pi,\mu)$ relating two such bundle gerbe automorphisms locally reduces to a smooth map $\Lambda \colon U \to \U(1)$. This map shifts the 1-form gauge parameter itself via the pullback of the Maurer-Cartan form:
    \begin{equation*}
        \theta' = \theta + d\log \Lambda.
    \end{equation*}
\end{itemize}
This explicit local translation illustrates how all the machinery introduced in this section, such fiber products, Čech nerves, bigroupoids, and descent data, exists precisely to provide the global foundation for the local transformation rules $b_1 = f^* b_2 - d\theta$ and $\theta' = \theta + d\log \Lambda$ that define the symmetry group of the $H$-flux and that were known in this local form in the supergravity literature since decades ago.


\subsection{The generalized Ricci flow}


Given a curving $b\in \Omega^2(Y)$ on $(\cP,A,\pi,\mu)$ we denote its curvature by $H_b\in \Omega^3(M)$. This is a closed three-form defined on $M$ with periods in $2\pi\mathbb{Z}$; it is defined as the unique $3$-form on $M$ such that $\pi^*H_b = \dd b$ on $Y$. The exact simplicial differential of the simplicial manifold generated as the \v{C}ech nerve of $Y\to M$ will be denoted by $\delta\colon \Omega^k(Y^{[j]})\to \Omega^k(Y^{[j+1]})$ for any $j,k\in\mathbb{N}$, where $Y^{[k]}$ is the $k$-fold fibered product of $Y\to M$ with itself.

\begin{definition}
\label{def:NSNSsystem}
The \emph{generalized Ricci flow} on $(\cP,A,\pi,\mu)$ is the following curvature evolution flow:
\begin{equation}
\label{eq:GeneralizedRicciflow}
\partial_t g_t = - 2 \Ric^{g_t} +    H_{b_t} \circ_{g_t} H_{b_t}\, , \qquad \partial_t b_t = - \nabla^{g_t\ast}H_{b_t}  
\end{equation}
	
\noindent
for pairs $\left\{ g_t,b_t\right\}$ consisting of a smooth family $g_t$ of Riemannian metrics on $M$ and a smooth family $b_t\in \Omega^2(Y)$ of curvings. Here we have set:
\begin{equation*}
(H_{b_t} \circ_{g_t} H_{b_t}) (v_1,v_2) = \langle H_{b_t}(v_1) , H_{b_t}(v_2)\rangle_{g_t} \, , \qquad \forall\,\, v_1 , v_2 \in \mathfrak{X}(M)
\end{equation*}

\noindent
in terms of the determinant metric $\langle - , - \rangle_{g_t}$ induced by $g_t$ on the exterior forms on $M$.
\end{definition}

\noindent
Hence, initial data for the generalized Ricci flow consist of pairs of the form $(g,b)$, where $g$ is a Riemannian metric on $M$ and $b$ is a curving on $(\cP,A,\pi,\mu)$.  

\begin{remark}
The set of curvings on $(\cP,A,\pi,\mu)$ is an affine space modelled on the two-forms $\Omega^2(M)$ on $M$, that is, if $b_1,b_2\in \Omega^2(Y)$ are curvings on $(\cP,A,\pi,\mu)$, then there exists a unique two-form $b_0\in \Omega^2(M)$ such that:
\begin{equation*}
b_1 - b_2 = \pi^{\ast}b_0 
\end{equation*}

\noindent
Therefore, the \emph{time} derivative $ \partial_t b_t$ of a smooth family $b_t$ of curvings, which is a priori a family of two-forms on $Y$ satisfying $F_A= \delta b_t \in \Omega^2(Y\times_M Y)$, is in fact canonically defined as a family of two-forms on $M$, and is understood as such in the second equation of \eqref{eq:GeneralizedRicciflow}. 
\end{remark}


\subsection{Self-similar flows}
\label{sec:selfsimilarflows}


Let $\cI(\cP,A,\pi,\mu)$ denote the set of initial data on $(\cP,A,\pi,\mu)$, that is, the set of pairs $(g,b)$ consisting of a Riemannian metric on $M$ and a curving on $(\cP,A,\pi,\mu)$. From Definition~\ref{def:NSNSsystem} we readily see that such a pair $(g,b) \in \cI(\cP,A,\pi,\mu)$ is a generalized Ricci flow if and only if:
\begin{equation*} 
\Ric^{g}  = \frac{1}{2} H_{b} \circ_{g} H_{b}\, , \qquad  \nabla^{g\ast}H_{b}  = 0
\end{equation*}

\noindent
Solutions to these equations are \emph{fixed points} of the generalized Ricci flow and can be considered as having \emph{trivial} evolution. There is a more general class of generalized Ricci flows that can be considered as evolving trivially, namely those that evolve via symmetries of the system, which in our framework correspond to families of automorphisms in $\Aut(\cP,A,\pi,\mu)$. We define a \emph{self-similar} generalized Ricci flow as a generalized Ricci flow of the form:
\begin{equation*}
(g_t, b_t)  = (g,b) \cdot (\cQ_t,\Theta_t, \xi_t, \Psi_t,f_t)\, , \qquad (g,b) \in \cI(\cP,A,\pi,\mu)
\end{equation*}

\noindent
for a \emph{smooth} family of automorphisms $(\cQ_t,\Theta_t, \xi_t, \Psi_t,f_t)$ parametrized by an open interval with Cartesian coordinate $t$. Recalling our assumption that the base diffeomorphisms $f_t \colon M \to M$ admit lifts $\bar{f}_t \colon Y \to Y$, this action evaluates explicitly to: 
\begin{equation}
\label{eq:evolutionauto}
(g_t, b_t)  = (g,b) \cdot (\cQ_t,\Theta_t, \xi_t, \Psi_t,f_t) = (f_t^{\ast} g , \bar{f}^{\ast}_t b  - \Delta_{\pi}^{\ast} F_{\Theta_t} )
\end{equation}

\begin{remark}
Conceptually, a self-similar flow may be viewed as a \emph{homotopy stationary point} (or homotopy fixed point) of the generalized Ricci flow. The space of initial data $\cI(\cP,A,\pi,\mu)$ carries a canonical right action of the full automorphism 2-group $\Aut(\cP,A,\pi,\mu)$ (or rather, its group of isomorphism classes). This action gives rise to a notion of morphisms between geometric states, inducing an action groupoid whose objects are the metrics and curvings. In this higher categorical context, a self-similar flow consists of specifying, for each time $t$, a state $(g_t, b_t)$ of the flow, which may be strictly different from $(g_0, b_0)$, while simultaneously specifying a smooth family of 1-morphisms $(\cQ_t,\Theta_t, \xi_t, \Psi_t,f_t)$ in the action groupoid. These 1-morphisms provide the explicit geometric data witnessing that the evolving point $(g_t, b_t)$ remains continuously isomorphic---though, in general, no longer strictly equal---to the initial point $(g_0, b_0)$.
\end{remark}

\noindent
There is a smooth theory for the automorphism 2-groups $\Aut(\cP,A,\pi,\mu)$ in~\cite{BunkShahbazi}, developed in a smooth higher-geometric setting (using the language of simplicial and $\infty$-presheaves on cartesian spaces), but here we employ a restricted notion of a smooth family of automorphisms $(\cQ_t,\Theta_t, \xi_t, \Psi_t,f_t)$, which will be sufficient for our purposes. First, we consider $f_t$ to be a smooth isotopy of diffeomorphisms such that $f_0 = \Id$ is the identity on $M$. Secondly, we assume that this isotopy $f_t$ admits a smooth lift $\bar{f}_t \colon Y\to Y$ to an isotopy of $Y$ with $\bar{f}_t = \Id_Y$, which we consider as given. Existence of such a smooth lift is, in general, obstructed, but it always exists in important cases, for instance when $\pi\colon Y \to M$ is the universal cover of $M$. The lift $\bar{f}_t \colon Y\to Y$ fixes, in turn, the smooth map $\xi_t$ as stated in the previous diagram:
\begin{eqnarray*}
\xi_t = \Id \times \bar{f}_t \colon Y \times_{\pi} Y \to Y \times^{f_t}_{\pi} Y
\end{eqnarray*}

\noindent
which can hence be omitted for ease of notation. Since $f_t$ is an isotopy to the identity at $t=0$, it follows that we can choose $\bar{f}_t$ to be an isotopy to the identity at $t=0$, and therefore we may take $Q$ to be a \emph{fixed} trivial $\U(1)$ bundle on $Y \times_{\pi} Y$. Finally, we choose a family of connections $\Theta_t$ and a family of automorphisms:
\begin{equation}
\label{eq:familyPsi}
\Psi_t\colon  (\varrho^{\ast}_1\cP, \varrho^{\ast}_1 A) \otimes (\mathfrak{z}^{\ast}_{2} \cQ , \mathfrak{z}^{\ast}_{2}\Theta_t ) \to (\mathfrak{z}^{\ast}_{1} Q , \mathfrak{z}^{\ast}_{1}\Theta_t ) \otimes (\varrho^{\ast}_2\cP, \varrho^{\ast}_2 A)
\end{equation}

\noindent
such that $(\cQ,\Theta_t, \xi_t, \Psi_t,f_t)$ fits into the master diagram for every $t$. This yields a well-defined notion of family of smooth automorphisms $(\cQ,\Theta_t, \xi_t, \Psi_t,f_t)$ of $(\cP,A,\pi,\mu)$, which we require to be the \emph{identity} at $t= 0$, under the assumption that isotopies of $M$ can be lifted to isotopies of $Y$ along $\pi\colon Y \to M$. By this discussion, we also obtain the following result.

\begin{lemma}
\label{lemma:existencePsiTheta}
Let $f_t \colon M\to M$ be an isotopy to the identity at $t=0$ and assume there exists a smooth lift $\bar{f}_t\colon Y\to Y$ of $f_t$ that is an isotopy to the identity. Then, there exists family $\Theta_t$ of connections on $Q$ and a family of automorphisms $\Psi_t$ as in \eqref{eq:familyPsi}, such that $(\cQ,\Theta_t, \xi_t, \Psi_t,f_t)$ is a family of automorphisms of $(\cP,A,\pi,\mu)$. 
\end{lemma}

\noindent
We characterize now the self-similar solutions of the generalized Ricci flow on an abelian bundle gerbe $(\cP,A,\pi,\mu)$, similar to the characterization obtained in \cite{GFStreetsBook} within the framework of generalized metrics on exact Courant algebroids. 

\begin{prop}
\label{prop:pregeneralizedsolitons}
Let $(g,b)\in \cI(\cP,A,\pi,\mu)$. There exists a family $(\cQ,\Theta_t, \Psi_t,f_t) \in \Aut(\cP,A,\pi,\mu)$ such that the pair:
\begin{equation}
\label{eq:generatedflow}
(g_t, b_t) = (\cQ,\Theta_t, \Psi_t,f_t)\cdot (g,b) = (f_t^{\ast} g , \bar{f}^{\ast}_t b  - \Delta_{\pi}^{\ast} F_{\Theta_t})    
\end{equation}

\noindent
is a generalized Ricci flow only if:
\begin{equation}
\label{eq:pregeneralizedsolitons}
\Ric^{g} + \frac{1}{2} \cL_{v} g - \frac{1}{2} H_{b} \circ_{g} H_{b} = 0\, , \qquad \nabla^{g\ast}H_{b} +   H_b (v) + \alpha  = 0  
\end{equation} 

\noindent
for a vector field $v \in \mathfrak{X}(M)$ and a closed two-form $\alpha \in \Omega^2(M)$ on $M$. Conversely, if $(g,b)\in \cI(\cP,A,\pi,\mu)$ satisfies \eqref{eq:pregeneralizedsolitons} for a vector field $v\in \mathfrak{X}(M)$ and a closed two-form $\alpha\in \Omega^2(M)$ such that $\pi^{\ast}\alpha \in \Omega^2(Y)$ is exact, then there exists a family of automorphisms $(\cQ,\Theta_t, \Psi_t,f_t) \in \Aut(\cP,A,\pi,\mu)$ for which \eqref{eq:generatedflow} is a generalized Ricci flow on $(\cP,A,\pi,\mu)$.
\end{prop}

\begin{proof}
We proceed by plugging $(g_t, b_t)   = (f_t^{\ast} g , \bar{f}^{\ast}_t b  - \Delta_{\pi}^{\ast} F_{\Theta_t})$ into the equations of the generalized Ricci flow, given in \eqref{eq:GeneralizedRicciflow}. We first compute: 
\begin{eqnarray*}
& \partial_t g_t = \partial_t f^{\ast}_t g = f^{\ast}_t \cL_{v_t} g\\
& \partial_t b_t = \partial_t (\bar{f}^{\ast}_t b  - \Delta_{\pi}^{\ast} F_{\Theta_t}) = \bar{f}^{\ast}_t \cL_{\bar{v}_t}b - \Delta_{\pi}^{\ast} F_{\partial_t\Theta_t} = \bar{f}^{\ast}_t ((\pi^{\ast}H_b)(\bar{v}_t) +  \dd_Y(b(\bar{v}_t))) - \Delta_{\pi}^{\ast} F_{\partial_t\Theta_t}
\end{eqnarray*}

\noindent
where $v_t$ is the family of vector fields on $M$  generated by $f_t$ via the relation:
\begin{equation*}
v_t \circ f_t = \partial_t f_t
\end{equation*}

\noindent
and the symbol $\cL$ denotes the Lie derivative. Similarly, $\bar{v}_t$ is the family of vector fields on $Y$ generated by $\bar{f}_t$:
\begin{equation*}
\bar{v}_t \circ \bar{f}_t = \partial_t \bar{f}_t
\end{equation*}

\noindent
Note that $\dd \pi (\bar{v}_t) = v_t$. This immediately implies that $(g_t, b_t)$ satisfies the Einstein evolution equation in \eqref{eq:GeneralizedRicciflow} if and only if:
\begin{equation*}
f^{\ast}_t \cL_{v_t} g = - 2 f^{\ast}_t\Ric^{g} +    f^{\ast}_t(H_{b} \circ_{g} H_{b}) \,\, \Leftrightarrow \,\,     \cL_{v_t} g = - 2  \Ric^{g} + H_{b} \circ_{g} H_{b}
\end{equation*} 

\noindent
Hence, evaluating this equation at $t=0$ yields the first equation in the statement. Regarding the Maxwell evolution equation in \eqref{eq:GeneralizedRicciflow}, that is, the evolution equation  for $b_t$, we observe that:
\begin{equation*}
\delta(\dd_Y(\bar{f}^{\ast}_t  b(\bar{v}_t)) - \Delta_{\pi}^{\ast} F_{\partial_t\Theta_t}) = 0
\end{equation*}

\noindent
where $\delta$ is the simplicial differential of the \v{C}ech nerve of $\pi\colon Y\to M$. Therefore, there exists a smooth family $\alpha_t$ of closed two-forms such that: 
\begin{equation*}
\pi^{\ast}\alpha_t = \dd_Y(\bar{f}^{\ast}_t  b(\bar{v}_t)) - \Delta_{\pi}^{\ast} F_{\partial_t\Theta_t} 
\end{equation*} 

\noindent
and consequently:
\begin{equation*}
\nabla^{g\ast}H_{b} +   H_b (v_t) + \alpha_t  = 0  
\end{equation*} 

\noindent
Evaluating this equation at $t=0$ we obtain the second equation in the statement. For the converse, assume $(g,b) \in \cI(\cP,A,\pi,\mu)$ satisfies equations \eqref{eq:pregeneralizedsolitons} for a certain $v\in \mathfrak{X}(M)$ and a closed two-form $\alpha  \in \Omega^2(M)$ such that $\pi^{\ast}\alpha = \dd \theta \in \Omega^2(Y)$ is exact. Consider the family $f_t$ of diffeomorphisms connected to the identity generated by $v$ and its lift $\bar{f}_t\colon Y\to Y$, with corresponding vector field $\bar{v} \in \mathfrak{X}(Y)$. By Lemma \ref{lemma:existencePsiTheta}, there exists a family of connections $\Theta_t$ and a family of isomorphisms $\Psi_t$ as in \eqref{eq:familyPsi} such that $(\cQ,\Theta_t, \xi_t, \Psi_t,f_t)$ is a family of automorphisms of $(\cP,A,\pi,\mu)$. Then, we define:
\begin{equation*}
g_t = f^{\ast}_t g\, , \qquad b_t = b\cdot (\cQ,\Theta_t, \xi_t, \Psi_t,f_t)  = \bar{f}^{\ast}_t b  - \Delta_{\pi}^{\ast} F_{\Theta_t}
\end{equation*}

\noindent
where $f_t$ is the family of diffeomorphisms generated by $v$ and $\bar{f}_t \colon Y \to Y$ is a lift of $f_t$, with corresponding vector field $\bar{v} \in \mathfrak{X}(Y)$. We compute:
\begin{equation*}
\partial_t g_t = \partial_t f^{\ast}_t g = f^{\ast}_t \cL_v g = f^{\ast}_t (-  2 \Ric^{g} + H_b\circ_g H_b) = -  2 \Ric^{g_t} + (f^{\ast}_tH_b)\circ_{g_t} (f^{\ast}_t H_b) = -  2 \Ric^{g_t} + H_{b_t}\circ_{g_t} H_{b_t} 
\end{equation*}

\noindent
and hence the evolution equation for $g_t$ in Equation \eqref{eq:GeneralizedRicciflow} is satisfied. Regarding $b_t$, we have:
\begin{eqnarray*}
& \partial_t b_t  = \partial_t (\bar{f}^{\ast}_t b  - \Delta_{\pi}^{\ast} F_{\Theta_t}) = \bar{f}^{\ast}_t \cL_{\bar{v}} b - \Delta_{\pi}^{\ast} F_{\partial_t\Theta_t} = \bar{f}^{\ast}_t ( \pi^{\ast}H_b(v) + \dd_Y (b (\bar{v}))) - \Delta_{\pi}^{\ast} F_{\partial_t\Theta_t}  \\
& = \bar{f}^{\ast}_t ( - \pi^{\ast}( \alpha + \nabla^{g\ast}H_b) + \dd_Y (b (\bar{v}))) - \Delta_{\pi}^{\ast} F_{\partial_t\Theta_t} 
=   -   \nabla^{g_t\ast}H_{b_t} - \dd_Y (\bar{f}^{\ast}_t \theta) + \dd_Y (\bar{f}^{\ast}_t b (\bar{v})) - \Delta_{\pi}^{\ast} F_{\partial_t\Theta_t} 
\end{eqnarray*}

\noindent
Therefore, by choosing the family of connections $\Theta_t$ on the fixed trivial bundle $\cQ \to Y \times_\pi Y$ such that its diagonal pullback satisfies:
\begin{equation*}
 \Delta_\pi^* \Theta_t =  \int_0^t \bar{f}^{\ast}_s (i_{\bar{v}} b - \theta)\, \dd s \in \Omega^1(Y)
\end{equation*}

\noindent
modulo a closed one-form on $Y$, the exact terms perfectly cancel and we obtain $\partial_t b_t = - \pi^* (\nabla^{g_t\ast}H_{b_t})$ and thus $(g_t,b_t)$ is indeed a solution to the generalized Ricci flow.
\end{proof}

\noindent
This result motivates the following definition.
\begin{definition}
A \emph{steady generalized Ricci soliton} on $(\cP,A,\pi,\mu)$ is a triple $(g,b,v)$ consisting of a Riemannian metric $g$ on $M$, a curving $b$ on $(\cP,A,\pi,\mu)$ and a vector field $v\in \mathfrak{X}(M)$ satisfying the following differential system:
\begin{equation}
\label{eq:generalizedriccisystem}
\Ric^{g} + \frac{1}{2} \cL_{v} g - \frac{1}{2} H_{b} \circ_{g} H_{b} = 0\, , \qquad \nabla^{g\ast}H_{b} +   H_b (v)   = 0  
\end{equation} 

\noindent
A steady generalized Ricci soliton $(g,b,v)$ is \emph{gradient} if $v^{\flat_g} = \nabla^g\phi$ for a smooth function $\phi\in C^{\infty}(M)$. 
\end{definition}

\begin{remark}
It is also possible to define \emph{shrinking} and \emph{expanding} generalized Ricci solitons, see for example \cite{GFStreetsBook,PodestaRaffero}. Since these do not seem to have a direct supergravity interpretation, we restrict in this note to \emph{steady} generalized Ricci solitons.
\end{remark}

\noindent
Since we will exclusively consider steady generalized Ricci solitons, in the following we will omit the term \emph{steady} for simplicity in the exposition. We will denote gradient Ricci solitons as triples of the form $(g,b,\phi)$, and we will refer to the function $\phi$ as the \emph{dilaton} of the generalized Ricci soliton. We will refer to the first equation in \eqref{eq:generalizedriccisystem} as the Einstein equation, whereas we will refer to the second equation in \eqref{eq:generalizedriccisystem} as the \emph{Maxwell equation}. This terminology is rooted in the theoretical physics interpretation of gradient generalized Ricci solitons as solutions to a certain supergravity theory. For gradient generalized Ricci solitons, the Einstein equation simplifies to:
\begin{equation*}
 \Ric^{g} + \nabla^g \dd \phi - \frac{1}{2} H_{b} \circ_{g} H_{b} = 0   
\end{equation*}

\noindent
The following result is well-known; an explicit proof can be found, for instance, in \cite[Proposition 4.33]{GFStreetsBook} or \cite[Proposition 3.7]{Moroianu:2023jof}.
\begin{lemma}
\label{lemma:gradientcondition}
Let $(g,b,\phi)$ be a generalized gradient Ricci soliton. The following formula holds:
\begin{equation}
\label{eq:strongcondition}
\nabla^{g\ast}\dd\phi + \vert\dd \phi\vert_g^2 - \vert H_b \vert^2_g = \lambda(g,b,\phi)
\end{equation}

\noindent
for a real constant $\lambda(g,b,\phi) \in \mathbb{R}$.
\end{lemma}
 
\begin{remark}
The factor 6 discrepancy with respect to \cite[Proposition 4.33]{GFStreetsBook} is due to the fact that we are using the determinant norm rather than the tensor norm on forms. 
\end{remark}

\noindent
A gradient generalized Ricci soliton $(g,b,\phi)$ satisfies the equations of motion of bosonic NS-NS supergravity if and only if $\lambda(g,b,\phi) = 0$. When this is the case, Equation \eqref{eq:strongcondition} corresponds to the \emph{dilaton equation} of bosonic NS-NS supergravity. This motivates the following definition. 

\begin{definition}
A gradient generalized Ricci soliton $(g,b,\phi)$ is \emph{NS} if $\lambda(g,b,\phi) = 0$.
\end{definition} 

 \noindent
We will denote by $\Conf(\cP,A,\pi,\mu)$ the configuration space of triples of the form $(g,b,\phi)$, and by $\Sol(\cP,A,\pi,\mu)\subset \Conf(\cP,A,\pi,\mu)$ the set of gradient generalized Ricci solitons on $(\cP,A,\pi,\mu)$.  

\begin{remark}
The \emph{infinitesimal} results we have obtained above by differentiating families of automorphisms can be understood within the general and comprehensive framework for infinitesimal symmetries of bundle gerbes developed in \cite{DjounvounaVaughan,KrepskiVaughan} in terms of multiplicative vector fields on the Lie groupoid underlying a bundle gerbe. This is, however, beyond the scope of this note.
\end{remark}


\subsection{The generalized Ricci soliton system in the Einstein Frame}
\label{sec:GenRicciSolEinstein}

 
The gradient generalized Ricci soliton system admits a variational interpretation through the following action functional, which is inherited from the NS-NS supergravity bosonic action:
\begin{eqnarray}
\label{eq:stringactionfunctional}
S[g,b] = \int_M \nu_g e^{-\phi} (s^g + \vert \dd\phi \vert^2_g -\frac{1}{2}\vert H_b \vert_g^2) 
\end{eqnarray}

\noindent
where $s^g \in C^{\infty}(M)$ denotes the scalar curvature of $g$ and $\phi \in C^{\infty}(M)$ is considered as given. The fact that in this functional the scalar curvature occurs with a multiplicative factor of the dilaton is related to the fact that the Einstein equation of the generalized Ricci soliton system includes a term of the form $\nabla^g\dd \phi \in \Gamma(T^{\ast}M\odot T^{\ast}M)$, which turns out to be problematic regarding the analytic properties of the system. This can be remedied by a conformal transformation of the metric that eliminates this factor and brings the Einstein equation to its \emph{standard form}, usually referred to as the \emph{Einstein frame} in the theoretical physics literature. To do this, we consider the bijection:
\begin{eqnarray*}
( - )_E \colon \Conf(\cP,A,\pi,\mu) \to \Conf(\cP,A,\pi,\mu)\, , \qquad (g,b,\phi) \mapsto (g_E = e^{-\frac{2\phi}{n-1}}\,g , b , \phi)
\end{eqnarray*}

\noindent
where $\dim(M) = n + 1$. 
\begin{lemma}
A triple $(g,b,\phi) \in \Conf(\cP,A,\pi,\mu)$ is a gradient generalized Ricci soliton if and only if $(g_E,b,\phi)$ satisfies the 
following differential system:
\begin{eqnarray}
& \Ric^{g_E}   - \frac{1}{n-1}\,\dd\phi\otimes \dd\phi + \frac{1}{\,n-1\,} (e^{-\tfrac{4\phi}{n-1}} \vert H_b\vert_{g_E}^2 +   \lambda \, e^{ \frac{2\phi}{n-1}})\, g_E - \frac{1}{2} e^{-\tfrac{4\phi}{n-1}}\,H_b\circ_{g_E} H_b  = 0 \label{eq:NSNSEinsteinframeEinstein}\\
& \delta^{g_E} H_b + \frac{4}{n-1}  H_b(\dd\phi^{\sharp_{g_E}}) = 0\, , \qquad \delta^{g_E} \dd\phi = e^{-\tfrac{4\phi}{n-1}} \vert H_b\vert_{g_E}^2 +  \lambda\, e^{ \frac{2\phi}{n-1}} \label{eq:NSNSEinsteinframeII}
\end{eqnarray}
 
\noindent
for a real constant $\lambda \in \mathbb{R}$.
\end{lemma}

\begin{remark}
By Lemma \ref{lemma:gradientcondition}, the second equation in \eqref{eq:NSNSEinsteinframeII}, to which we still refer as the \emph{dilaton equation} of the system, is redundant. However, it is convenient to consider it explicitly as part of the gradient generalized Ricci soliton system to study its Cauchy problem.
\end{remark}

\begin{proof}
Setting $g = e^{\frac{2\phi}{n-1}} g_E$, we begin with the classical identity:
\begin{eqnarray*}
\nabla^{g}_{v_1} v_2 = \nabla^{g_E}_{v_1} v_2  + \frac{1}{n-1} v_1(\phi) v_2 + \frac{1}{n-1} v_2 (\phi) v_1 - \frac{1}{n-1} g_E(v_1 , v_2) (\dd\phi)^{\sharp_{g_E}}
\end{eqnarray*}

\noindent
From this, we obtain the standard identities:
\begin{eqnarray*}
\label{eq:conformalidentitiescodifferentialricci}
& \Ric^g = \Ric^{g_E} - (\nabla^{g_E} \dd\phi - \frac{1}{n-1} \dd\phi \otimes \dd\phi) + \frac{1}{n-1} (\delta^{g_E}\dd\phi - \lvert \dd\phi\rvert^2_{g_E} )\,g_E \nonumber \\
&\delta^g \alpha = e^{-\frac{2\phi}{n-1}} (\delta^{g_E}\alpha - \frac{n + 1 -2k}{n-1} \alpha(\dd\phi)^{\sharp_{g_E}}) \\
& \nabla^g \dd \phi = \nabla^{g_E} \dd \phi - \frac{2}{n-1} \dd \phi \otimes \dd \phi + \frac{1}{n-1} \vert \dd \phi\vert_{g_E}^2 g_E
\end{eqnarray*}

\noindent
Using the previous relations, we compute:
\begin{eqnarray*}
& \Ric^{g} + \nabla^g \dd \phi - \frac{1}{2} H_{b} \circ_{g} H_{b}  = \Ric^{g_E} + \frac{1}{\,n-1\,} \delta^{g_E} \dd\phi\, g_E - \frac{1}{2} e^{-\tfrac{4\phi}{n-1}}\,H_b\circ_{g_E} H_b - \frac{1}{n-1}\,\dd\phi\otimes \dd\phi \\
& \delta^g H_{b} +   H_b (\dd \phi)^{\sharp_g}   = e^{-\frac{2\phi}{n-1}}( \delta^{g_E} H_{b} +  \frac{4}{n-1} H_b (\dd \phi)^{\sharp_{g_E}} ) \\
& \nabla^{g\ast}\dd\phi + \vert\dd \phi\vert_g^2 - \vert H_b \vert^2_g = e^{- \frac{2\phi}{n-1}}\delta^{g_E}\dd\phi - e^{- \frac{6\phi}{n-1}}\ \vert H_b\vert_{g_E}^2 
\end{eqnarray*}

\noindent
Combining these equations, we conclude.
\end{proof}

\begin{remark}
Recall the identity $\int_M e^{-\phi} \nu_g s^g = \int_M \nu_{g_E} s^{g_E}$, which shows explicitly that the scalar curvature term in \eqref{eq:stringactionfunctional} becomes the Einstein-Hilbert term in the Einstein frame.
\end{remark}

\noindent
In the following, we will exclusively consider the gradient generalized Ricci soliton system in the Einstein frame, that is, we will be working with equations \eqref{eq:NSNSEinsteinframeEinstein} and \eqref{eq:NSNSEinsteinframeII}, omitting the subscript $E$ in $g_E$ for ease of notation. There is no loss of generality in this choice since $( - )_E \colon \Conf(\cP,A,\pi,\mu) \to \Conf(\cP,A,\pi,\mu)$ is, by construction, a bijection both at the level of configuration and solution spaces.


\section{The Cauchy problem}


In this section, we study the Cauchy problem for gradient generalized Ricci solitons on a bundle gerbe. In order to set it up, we fix a smooth hypersurface $\Sigma\subset M$ and consider the generalized Ricci soliton system on a tubular neighborhood of $\Sigma$ in $M$, which we identify with $\cI\times \Sigma$ for an interval $\cI$ with Cartesian coordinate $\tau$. Hence, we consider the gradient generalized Ricci soliton system on the product manifold $\cI\times \Sigma$, and we proceed by reinterpreting the configuration space of the system in terms of families of objects defined over $\Sigma$ satisfying a system of evolution equations on $\Sigma$. We begin with the reduction of a curving on a bundle gerbe defined on $M = \cI \times \Sigma$, which is one of the main ingredients of the generalized Ricci soliton system. We proceed along the lines of \cite[Chapter 6]{Shahbazi}.


\subsection{Time-dependent reduction of bundle gerbes with connection}
\label{subsec:bundlereduction}


Let $\bar{\cC} := (\bar{\cP} , \bar{A}, \bar{Y})$ be a bundle gerbe on $M = \cI\times \Sigma$. Since the goal is to reduce the generalized Ricci soliton system on $\bar{\cC}$, we will assume that the \emph{topological data} contained in $\cC$, namely $(\bar{\cP},\bar{Y})$, is given by the pull-back of a bundle gerbe $(\cP,Y)$ defined over $\Sigma$, while we will allow the connective structure $\bar{A}$ to be genuinely \emph{time-dependent}, that is, not necessarily given by the pull-back of a fixed connective structure on $(\cP,Y)$. We will refer to such bundle gerbes as \emph{reducible} bundle gerbes on $\cI\times \Sigma$. This setup captures the dynamics of the Cauchy problem for gradient generalized Ricci solitons on a bundle gerbe in its full generality. Since the submersion $\bar{Y} \to M$ of a reducible bundle gerbe on $\cI\times\Sigma$ is, by assumption, a pull-back submersion, it follows that it is isomorphic to:
\begin{equation*}
\bar{Y} = \cI \times Y \to \cI \times \Sigma	
\end{equation*}

\noindent
where the projection on the first factor is the identity map and $Y\to \Sigma$ is the submersion underlying the bundle gerbe $(\cP , Y)$. Similarly:
\begin{equation*}
\bar{Y}\times_M  \bar{Y} = \cI\times Y \times_\Sigma Y \to \cI \times \Sigma
\end{equation*}

\noindent
with arrow restricting to the identity map $\cI\to \cI$ on the first factor. Since $\bar{\cP}$ is the pull-back of $\cP\to Y\times_{\Sigma} Y$ to $\cI \times Y\times_{\Sigma} Y$, we have:
\begin{eqnarray*}
\bar{\cP}= \cI \times \cP \to \cI \times Y\times_{\Sigma} Y
\end{eqnarray*}

\noindent
again with the arrow restricting to the identity map $\cI\to \cI$ on the first factor. The principal action of the group $\U(1)$ on $\bar{\cP} = \cI \times \cP$ is the one induced by the principal $\U(1)$ action carried by $\cP$ and the trivial action on the $\cI$ factor. Let $\bar{A} \in \Omega^1(\cI\times \cP , \mathbb{R})$ be a connective structure on $\bar{\cP}$. Then we can write:
\begin{equation*}
\bar{A} = \Psi_{\tau} \, \dd \tau + A_{\tau}
\end{equation*}

\noindent
where $\Psi_{\tau}$ is a family of functions on $\cP$ and $A_{\tau}$ is a family of connections on $\cP$. Recall that $\tau$ denotes the Cartesian coordinate on $\cI$. Since $\bar{A}$ is a connection on $\bar{\cP}$, it is in particular invariant under the $\U(1)$ action of the principal bundle $\bar{\cP}$, and since this action is ineffective on $\cI$ we conclude that $\Psi_{\tau}$ defines in fact a family of invariant functions on $\cP$ whence it descends to a family of invariant functions on $Y\times_{\Sigma} Y$ that we denote by the same symbol for ease of notation. Since $(\bar{\cP} , \bar{A}, \bar{Y})$ is a bundle gerbe, it comes equipped with an isomorphism:
\begin{equation*}
 (\bar{\pi}^{\ast}_{12}\bar{\cP}\otimes\bar{\pi}^{\ast}_{23}\bar{\cP} , \bar{\pi}^{\ast}_{12}\bar{A}\otimes \bar{\pi}^{\ast}_{23}\bar{A})  \xrightarrow{\simeq} (\bar{\pi}^{\ast}_{13} \bar{\cP} , \bar{\pi}^{\ast}_{13}\bar{A})
\end{equation*}

\noindent
where $\bar{\pi}\colon \bar{Y} \to M$ and the notation $\bar{\pi}_{ij} \colon \bar{Y}\times_M \bar{Y}\times_M \bar{Y} \to \bar{Y}\times_M \bar{Y}$ forgets the entry not labelled neither by $i$ nor $j$. This implies $\delta \Psi_{\tau} = 0$, where $\delta$ is the simplicial differential of the \v{C}ech nerve of $\pi\colon Y \to \Sigma$. Since this differential is exact, it follows that $\Psi_{\tau}$ descends through $\delta$ to a family of functions on $Y$ that we denote again by $\Psi_{\tau}$. The preceding discussion thus leads to the following correspondence. 

\begin{prop}
There is a natural one-to-one correspondence between reducible bundle gerbes $(\bar{\cP},\bar{Y},\bar{A})$ on $M = \cI\times \Sigma$ and tuples $(\cP,Y,A_{\tau} , \Psi_{\tau})$, consisting of a bundle gerbe $(\cP,Y)$ on $\Sigma$, a family of functions $\Psi_{\tau}$  on $Y$, and a family of connective structures $A_{\tau}$ on $(\cP,Y)$.
\end{prop}

\noindent
We will refer to $(\cP,Y, A_{\tau} , \Psi_{\tau} )$ as the \emph{reduction} of $(\bar{\cP},\bar{Y},\bar{A})$. A direct computation gives the following result.

\begin{lemma}
\label{lemma:gerbecurvaturedecomposition}
The curvature $\cF_{\bar{A}} \in \Omega^2(\cI \times Y\times_{\Sigma} Y)$ of $\bar{A} = \Psi_{\tau} \dd \tau + A_{\tau}$ satisfies the following equation:
\begin{eqnarray*}
\cF_{\bar{A}} = \dd_{Y^{[2]}} \Psi_{\tau} \wedge \dd \tau + \dd \tau \wedge \partial_{\tau} A_{\tau} + \cF_{A_{\tau}}
\end{eqnarray*}

\noindent
where $\dd_{Y^{[2]}}\colon \Omega^{\bullet}(Y\times_{\Sigma} Y) \to \Omega^{\bullet}(Y\times_{\Sigma} Y) $ is the exterior derivative on $\Omega^{\bullet}(Y\times_{\Sigma} Y)$ and $\cF_{A_{\tau}}$ is the family of curvatures of $A_{\tau}$.
\end{lemma}

\noindent
A curving on $(\bar{\cP},\bar{Y},\bar{A})$ is by definition a two-form $\bar{b}\in \Omega^2(\bar{Y}\times_M \bar{Y})$ satisfying:
\begin{equation*}
\bar{\delta} \bar{b} = F_{\bar{A}}
\end{equation*}

\noindent 
Since $\bar{Y}\times_M \bar{Y} = \cI \times Y \times_{\Sigma} Y$, we can canonically write every curving $\bar{b}$ on $(\bar{\cP},\bar{Y},\bar{A})$ as follows:
\begin{equation}
\label{eq:decompositionbarb}
\bar{b} = \dd \tau \wedge a_{\tau} + b_{\tau}
\end{equation}

\noindent
for uniquely determined families of one forms $a_{\tau}$ and two-forms $b_{\tau}$ on $Y$.
\begin{lemma}
\label{lemma:reduciblegerbe}
Let $(\bar{\cP},\bar{Y},\bar{A})$ be a reducible bundle gerbe with connective structure on $M = \cI \times \Sigma$ and let  $(\cP,Y, A_{\tau}, \Psi_{\tau})$ be its reduction on $\Sigma$. A two-form $\bar{b}\in \Omega^2(\bar{Y})$ is a curving on $(\bar{\cP},\bar{Y},\bar{A})$ if and only if, in the decomposition given in \eqref{eq:decompositionbarb}, $b_{\tau}$ is a family of curvings on $(\cP,Y,A_{\tau})$ and $a_{\tau}$ is a family of one-forms on $Y$ satisfying:
\begin{equation}
\label{eq:conditionat}
\delta a_{\tau} = \partial_{\tau} A_{\tau} - \dd_{Y^{[2]}} \Psi_{\tau}
\end{equation}

\noindent
for every $\tau\in \cI$, where $\delta$ is the simplicial differential of the \v{C}ech nerve of $\pi \colon Y \to \Sigma$.
\end{lemma}

\begin{proof}
Plugging equation \eqref{eq:decompositionbarb} in $\bar{\delta} \bar{b} = F_{\bar{A}}$ we have:
\begin{equation*}
\bar{\delta}\bar{b} = \bar{\delta}(\dd \tau \wedge a_{\tau} + b_{\tau}) = \dd \tau \wedge \delta a_{\tau} + \delta b_{\tau} = F_{\bar{A}} = \dd_{Y^{[2]}} \Psi_{\tau} \wedge \dd \tau + \dd \tau \wedge \partial_{\tau} A_{\tau} + F_{A_{\tau}}
\end{equation*}

\noindent
where we have used Lemma \ref{lemma:gerbecurvaturedecomposition}. Isolating by tensor type, this gives:
\begin{equation*}
\delta b_{\tau} = F_{A_{\tau}}\, , \qquad \delta a_{\tau} = \partial_{\tau} A_{\tau} - \dd_{Y^{[2]}} \Psi_{\tau}
\end{equation*}

\noindent
and thus we conclude. 
\end{proof}

\begin{cor}
\label{cor:curvingreduction}
There is a natural one-to-one correspondence between curvings on a reducible bundle gerbe $(\bar{\cP},\bar{Y},\bar{A})$ with reduction $(\cP,Y,A_{\tau},\Psi_{\tau})$ and pairs $(a_{\tau},b_{\tau})$ consisting of a family $b_{\tau}$ of curvings on $(\cP,Y,A_{\tau})$ and a family of one-forms $a_{\tau}$ on $Y$ satisfying Equation \eqref{eq:conditionat}.
\end{cor}

\noindent
Let $\bar{b} = \dd \tau \wedge a_{\tau} + b_{\tau}$ be a curving on a reducible bundle gerbe over $\cI\times \Sigma$. Then:
\begin{equation*}
\dd_{\bar{Y}}\bar{b} = \dd \tau \wedge (\partial_{\tau} b_{\tau} - \dd_Y a_{\tau}) + \dd_Y b_{\tau}
\end{equation*}

\noindent
where $\dd_{\bar{Y}}$ is the exterior derivative on $\bar{Y}$ and $\dd_Y$ is the exterior derivative on $Y$. 

\begin{lemma}
The family of two-forms $(\partial_{\tau} b_{\tau} - \dd_Y a_{\tau}) \in \Omega^2(Y)$ satisfies $\delta (\partial_{\tau} b_{\tau} - \dd_Y a_{\tau}) = 0$ for every $\tau \in \cI$, where $\delta$ denotes the simplicial differential of the \v{C}ech nerve of $\pi \colon Y\to M$. 
\end{lemma}

\begin{proof}
We compute:
\begin{equation*}
\delta (\partial_{\tau} b_{\tau} - \dd_Y a_{\tau}) = \partial_{\tau} \delta b_{\tau} - \dd_{Y^{[2]}} \delta a_{\tau} =  \partial_{\tau} \cF_{A_{\tau}}- \dd_{Y^{[2]}} (\partial_{\tau} A_{\tau} - \dd_{Y^{[2]}} \Psi_{\tau}) = 0
\end{equation*}

\noindent
where we have used Lemma \ref{lemma:reduciblegerbe}.
\end{proof}

\noindent
Since $\delta$ is exact, by the previous lemma, it follows that there exists a unique family of two-forms $\psi_{\tau} \in \Omega^2(\Sigma)$ on $\Sigma$ such that:
\begin{equation*}
\pi^{\ast} \psi_{\tau} = \partial_{\tau} b_{\tau} - \dd_Y a_{\tau}
\end{equation*}

\noindent
Since $\dd_{\bar{Y}} \bar{b} = \bar{\pi}^{\ast} H_{\bar{b}}$ for a uniquely determined three-form $\bar{H}_{\bar{b}}\in \Omega^3(M)$, it follows that:
\begin{equation}
\label{eq:decompositionH}
H_{\bar{b}} = \dd \tau \wedge \psi_{\tau} + H_{b_{\tau}}
\end{equation}

\noindent
where $H_{b_{\tau}}$ is the curvature of the family of curvings $b_{\tau}$. In particular:
\begin{eqnarray*}
\langle H_{\bar{b}} , H_{\bar{b}} \rangle_{\bar{g}} = \vert \psi_{\tau} \vert^2_{h_{\tau}}  + \vert H_{b_{\tau}} \vert^2_{h_{\tau}}  
\end{eqnarray*}

\noindent
Therefore, every curving $\bar{b}$ on a reducible bundle gerbe is equivalent to a pair $(a_{\tau} , b_{\tau})$ as introduced in Lemma \ref{lemma:reduciblegerbe}, and this pair defines in turn a family of two-forms $\psi_{\tau}$ on $\Sigma$ as described above. 
\begin{definition}
The family of two-forms $\psi_{\tau}$ is the \emph{derived} family of two-forms of the curving $\bar{b}$ on the reducible bundle gerbe $(\bar{\cP},\bar{A},\bar{Y})$.
\end{definition}

\noindent
The fact that $H_{\bar{b}}\in \Omega^3(\cI\times \Sigma)$ is closed amounts to conditions $\partial_{\tau} H_{b_{\tau}} = \dd_{\Sigma} \psi_{\tau}$ and $\dd_{\Sigma} H_{b_{\tau}} = 0$.


\subsection{The Cauchy evolution flow}


In this subsection, we characterize the gradient generalized Ricci solitons on a reducible bundle gerbe $(\bar{\cP} , \bar{A}, \bar{Y})$ on $\cI\times \Sigma$ in terms of solutions to a prescribed system of flow equations on the reduction $(\cP, Y , A_{\tau}  , \Psi_{\tau})$ of $(\bar{\cP} , \bar{A}, \bar{Y})$ to $\Sigma$. 

\begin{lemma}
Elements $(g , \bar{b}, \bar{\phi}) \in \Conf(\bar{\cP} , \bar{A}, \bar{Y})$ are in one-to-one correspondence with tuples $( h_{\tau} , a_{\tau} , b_{\tau} , \phi_{\tau})$ as introduced above.  
\end{lemma}

\noindent
We will refer $( h_{\tau} , a_{\tau} , b_{\tau} , \phi_{\tau})$ as the \emph{reduction} of $(g , \bar{b}, \bar{\phi})$. We set:
\begin{equation*}
\Sigma_{\tau} := \left\{ \tau\right\}\times \Sigma \hookrightarrow M = \cI \times \Sigma\, , \qquad \Sigma := \left\{ 0\right\}\times \Sigma \hookrightarrow M 
\end{equation*}

\noindent
Tuples of the form $( h_{\tau} , a_{\tau} , b_{\tau} , \phi_{\tau})$ constitute the variables of the system of flow equations defined by gradient generalized Ricci solitons on $(\cP, Y , A_{\tau}  , \Psi_{\tau})$, that we present in Proposition \ref{prop:CylindricalGeneralizedSolitonSystem}. We denote by $\nabla^{h_{\tau}}$ the Levi-Civita connection on $(\Sigma_{\tau},h_{\tau})$. A direct computation gives the following formulae. 
\begin{lemma}
\label{lemma:hyperbolicderivatives}
Let $g =  \dd \tau \otimes \dd \tau + h_{\tau}$ be a Riemannian metric on $M = \cI\times \Sigma$. Then, for every family of one-forms $\beta_{\tau} \in \Omega^1(\Sigma)$ we have:
\begin{eqnarray*}
& \nabla^g \dd \tau =  \frac{1}{2 } \partial_{\tau} h_{\tau} \in \Gamma(T^{\ast}\Sigma_\tau\odot T^{\ast}\Sigma_\tau)\, , \quad  \nabla^g\beta_{\tau} =  \dd \tau \otimes \partial_{\tau} \beta_{\tau} - \frac{1}{2} \dd \tau \odot (\partial_{\tau} h_{\tau}) (\beta^{\sharp_{h_{\tau}}}_{\tau}) + \nabla^{h_{\tau}}\beta_{\tau} 
\end{eqnarray*}

\noindent
where the superscript $\sharp_{h_{\tau}}$ denotes the musical isomorphism with respect to $h_{\tau}$.
\end{lemma}

\begin{remark}
Note that $-2^{-1} \partial_{\tau} h_{\tau} $ is the shape operator or scalar second fundamental form of the embedded manifold $\Sigma_{\tau} \hookrightarrow M = \cI \times \Sigma$.
\end{remark}

\begin{lemma}
\label{lemma:Einsteingh}
A tuple $(g , \bar{b}, \bar{\phi}) \in \Conf(\bar{\cP} , \bar{Y}, \bar{A})$ satisfies the Einstein equation \eqref{eq:NSNSEinsteinframeEinstein} if and only if its reduction $( h_{\tau} , a_{\tau} , b_{\tau} , \phi_{\tau})$ satisfies the following differential system:
\begin{align*}
& \partial_{\tau}^2 h_{\tau} = 2 \Ric^{h_{\tau}} +  \partial_{\tau} h_{\tau} \circ_{h_{\tau}} \partial_{\tau} h_{\tau} - \frac{1}{2} \Tr_{h_{\tau}}(\partial_{\tau} h_{\tau}) \partial_{\tau} h_{\tau}   - \frac{2}{n-1}\,\dd_{\Sigma}\phi_{\tau} \otimes \dd_{\Sigma}\phi_{\tau}\\
& + \frac{2}{\,n-1\,} (e^{-\tfrac{4\phi}{n-1}} (\vert \psi_{\tau}\vert_{h_{\tau}}^2 + \vert H_{b_{\tau}}\vert_{h_{\tau}}^2) +   \lambda \, e^{ \frac{2\phi_{\tau}}{n-1}})\, h_{\tau} -  e^{-\tfrac{4\phi_{\tau}}{n-1}}\,(\psi_{\tau} \circ_{h_{\tau}} \psi_{\tau} + H_{b_{\tau}} \circ_{h_{\tau}} H_{b_{\tau}})  \\
&   \delta^{h_{\tau}}\partial_{\tau} h_{\tau} + \dd_\Sigma \Tr_{h_{\tau}}(\partial_{\tau} h_{\tau}) + \frac{2}{n-1}\, \partial_{\tau} \phi_{\tau}\, \dd_{\Sigma} \phi_{\tau} + e^{-\tfrac{4\phi_{\tau}}{n-1}}\,\psi_{\tau} \lrcorner_{h_{\tau}} H_{b_{\tau}} = 0\\
& s^{h_{\tau}} +  \frac{1}{4} \vert \partial_{\tau} h_{\tau} \vert_{h_{\tau}}^2 - \frac{1}{4} \Tr_{h_{\tau}}( \partial_{\tau} h_{\tau})^2  + \lambda   \, e^{ \frac{2\phi_{\tau}}{n-1}} + \frac{1}{n-1} ( (\partial_{\tau} \phi_{\tau})^2 - \vert \dd_{\Sigma}\phi_{\tau} \vert^2_{h_{\tau}})  +  \frac{1}{2} ( \vert \psi_{\tau} \vert^2_{h_{\tau}}   -  \vert H_{b_{\tau}} \vert^2_{h_{\tau}} ) e^{-\tfrac{4\phi_{\tau}}{n-1}}   = 0
\end{align*}

\noindent
where $\psi_{\tau}$ is the derived family of two-forms associated to $\bar{b} \in \Omega^2(\bar{Y})$. 
\end{lemma}

\begin{proof}
Let $(g,\bar{b},\bar{\phi})$ be a reducible gradient generalized Ricci soliton with reduction $( h_{\tau} , a_{\tau} , b_{\tau} , \phi_{\tau})$. A standard computation shows that the Ricci curvature tensor of $g = \dd\tau\otimes \dd\tau + h_{\tau}$ on $\cI\times \Sigma$ decomposes as follows:
\begin{eqnarray}
\label{eq:riccicylindrical}
& 2 \Ric^g   =  ( \frac{1}{2} \Tr_{h_{\tau}} (\partial_{\tau} h_{\tau} \circ_{h_{\tau}} \partial_{\tau} h_{\tau} ) -  \Tr_{h_{\tau}}( \partial_{\tau}^2 h_{\tau})) \, \dd \tau \otimes \dd\tau \\
& -   \dd\tau \odot  ( \delta^{h_{\tau}}\partial_{\tau} h_{\tau} + \dd_\Sigma \Tr_{h_{\tau}}(\partial_{\tau} h_{\tau}) ) + 2 \Ric^{h_{\tau}} +  \partial_{\tau} h_{\tau} \circ_{h_{\tau}} \partial_{\tau} h_{\tau} - \frac{1}{2} \Tr_{h_{\tau}}(\partial_{\tau} h_{\tau}) \partial_{\tau} h_{\tau} -  \partial_{\tau}^2 h_{\tau} \nonumber 
\end{eqnarray}

\noindent
Taking the trace, one obtains the following expression for the scalar curvature of $g$: 
\begin{equation}
\label{eq:scalarcylindrical}
s^g = s^{h_{\tau}} + \frac{3}{4}  \vert \partial_{\tau} h_{\tau} \vert_{h_{\tau}}^2 - \frac{1}{4} \Tr_{h_{\tau}} ( \partial_{\tau} h_{\tau} )^2 - \Tr_{h_{\tau}} ( \partial_{\tau}^2 h_{\tau})  
\end{equation}

\noindent
Furthermore:
\begin{eqnarray*}
s^g-2\Ric^g(\partial_{\tau},\partial_{\tau})  = s^{h_{\tau}}+\frac{1}{4}\vert \partial_{\tau} h_{\tau} \vert_{h_{\tau}}^2 - \frac{1}{4}\text{Tr}_{h_{\tau}}(\partial_{\tau}{h}_{\tau})^2
\end{eqnarray*}
We decompose the Einstein equation \eqref{eq:NSNSEinsteinframeEinstein} on $\cI\times \Sigma$ accordingly, that is:
\begin{eqnarray*}
& \Ric^{g}(\partial_{\tau} , \partial_{\tau}) + \frac{1}{\,n-1\,} (e^{-\tfrac{4\phi}{n-1}} \vert H_{\bar{b}} \vert_{g_E}^2 +   \lambda \, e^{ \frac{2\phi}{n-1}} - (\partial_{\tau}\phi)^2)  - \frac{1}{2} e^{-\tfrac{4\phi}{n-1}}\langle H_{\bar{b}}(\partial_{\tau}) ,  H_{\bar{b}} (\partial_{\tau})\rangle_g =   \\
& = \frac{1}{2} ( \frac{1}{2} \vert \partial_{\tau} h_{\tau} \vert_{h_{\tau}}^2 -  \Tr_{h_{\tau}}( \partial_{\tau}^2 h_{\tau})) + \frac{1}{\,n-1\,} ( \frac{3-n}{2} \vert \psi_{\tau} \vert^2_{h_{\tau}} e^{-\tfrac{4\phi_{\tau}}{n-1}} + \vert H_{b_{\tau}} \vert^2_{h_{\tau}} e^{-\tfrac{4\phi_{\tau}}{n-1}} +   \lambda \, e^{ \frac{2\phi_{\tau}}{n-1}} - (\partial_{\tau}\phi_{\tau})^2)   = 0\\
& \Ric^{g}(\partial_{\tau} )\vert_{T\Sigma_{\tau}}   - \frac{1}{n-1}\, \partial_{\tau} \phi_{\tau}\, \dd_{\Sigma} \phi_{\tau} - \frac{1}{2} e^{-\tfrac{4\phi}{n-1}}\,(H_{\bar{b}} \circ_{g} H_{\bar{b}}) (\partial_{\tau} )\vert_{T\Sigma_{\tau}} = \\
& = - \frac{1}{2} ( \delta^{h_{\tau}}\partial_{\tau} h_{\tau} + \dd_\Sigma \Tr_{h_{\tau}}(\partial_{\tau} h_{\tau}) )  - \frac{1}{n-1}\, \partial_{\tau} \phi_{\tau}\, \dd_{\Sigma} \phi_{\tau} - \frac{1}{2} e^{-\tfrac{4\phi_{\tau}}{n-1}}\,\psi_{\tau} \lrcorner_{h_{\tau}} H_{b_{\tau}} = 0\\
&  \Ric^{g}\vert_{T\Sigma_{\tau}\times T\Sigma_{\tau}}   - \frac{1}{n-1}\,\dd_{\Sigma}\phi\otimes \dd_{\Sigma}\phi + \frac{1}{\,n-1\,} (e^{-\tfrac{4\phi}{n-1}} \vert H_{\bar{b}}\vert_{g}^2 +   \lambda \, e^{ \frac{2\phi}{n-1}})\, h_{\tau} - \frac{1}{2} e^{-\tfrac{4\phi}{n-1}}\,(H_{\bar{b}} \circ_{g} H_{\bar{b}}) \vert_{T\Sigma_{\tau}\times T\Sigma_{\tau}} = \\
& =   \Ric^{h_{\tau}} +  \frac{1}{2} \partial_{\tau} h_{\tau} \circ_{h_{\tau}} \partial_{\tau} h_{\tau} - \frac{1}{4} \Tr_{h_{\tau}}(\partial_{\tau} h_{\tau}) \partial_{\tau} h_{\tau} - \frac{1}{2} \partial_{\tau}^2 h_{\tau} - \frac{1}{n-1}\,\dd_{\Sigma}\phi_{\tau} \otimes \dd_{\Sigma}\phi_{\tau}\\
& + \frac{1}{\,n-1\,} (e^{-\tfrac{4\phi_{\tau}}{n-1}} (\vert \psi_{\tau}\vert_{h_{\tau}}^2 + \vert H_{b_{\tau}}\vert_{g}^2) +   \lambda \, e^{ \frac{2\phi_{\tau}}{n-1}})\, h_{\tau} - \frac{1}{2} e^{-\tfrac{4\phi_{\tau}}{n-1}}\,(\psi_{\tau} \circ_{h_{\tau}} \psi_{\tau} + H_{b_{\tau}} \circ_{h_{\tau}} H_{b_{\tau}}) = 0
\end{eqnarray*}

\noindent
where we have repeatedly used Equation \eqref{eq:decompositionH}. Taking the trace of the $T\Sigma_{\tau}\times T\Sigma_{\tau}$ component of the Einstein equation \eqref{eq:NSNSEinsteinframeEinstein} we obtain: 
\begin{eqnarray*}
& s^{h_{\tau}} +  \frac{1}{2} \vert \partial_{\tau} h_{\tau} \vert_{h_{\tau}}^2 - \frac{1}{4} \Tr_{h_{\tau}}( \partial_{\tau} h_{\tau})^2 - \frac{1}{2} \Tr_{h_{\tau}}( \partial_{\tau}^2 h_{\tau}) - \frac{1}{\,n-1\,} \vert \dd_{\Sigma}\phi_{\tau} \vert^2_{h_{\tau}}\\ 
& + \frac{1}{\,n-1\,} ( \vert \psi_{\tau} \vert^2_{h_{\tau}} e^{-\tfrac{4\phi_{\tau}}{n-1}}  + \frac{3-n}{2} \vert H_{b_{\tau}} \vert^2_{h_{\tau}} e^{-\tfrac{4\phi_{\tau}}{n-1}} +   n \lambda \, e^{ \frac{2\phi_{\tau}}{n-1}} )  = 0
\end{eqnarray*}

\noindent
Combining this equation with the $\tau\tau$ component of the Einstein equation \eqref{eq:NSNSEinsteinframeEinstein} we obtain:
\begin{equation*}
s^{h_{\tau}} +  \frac{1}{4} \vert \partial_{\tau} h_{\tau} \vert_{h_{\tau}}^2 - \frac{1}{4} \Tr_{h_{\tau}}( \partial_{\tau} h_{\tau})^2  + \lambda   \, e^{ \frac{2\phi_{\tau}}{n-1}} + \frac{(\partial_{\tau} \phi_{\tau})^2 - \vert \dd_{\Sigma}\phi_{\tau} \vert^2_{h_{\tau}}}{n-1} +  \frac{1}{2} ( \vert \psi_{\tau} \vert^2_{h_{\tau}}   -  \vert H_{b_{\tau}} \vert^2_{h_{\tau}} ) e^{-\tfrac{4\phi_{\tau}}{n-1}}   = 0
\end{equation*}

\noindent
and thus we conclude. 
\end{proof}

\begin{lemma}
\label{lemma:Maxwellgh}
A tuple $(g , \bar{b}, \bar{\phi}) \in \Conf(\bar{\cP} , \bar{A}, \bar{Y})$ satisfies the Maxwell equation in \eqref{eq:NSNSEinsteinframeII} if and only if its reduction $( h_{\tau} , a_{\tau} , b_{\tau} , \phi_{\tau})$ satisfies the following differential system:
\begin{align*}
& \partial_{\tau} \psi_{\tau} = \delta^{h_{\tau}} H_{b_{\tau}} + \frac{4}{n-1} H_{b_{\tau}}(\dd_{\Sigma} \phi^{\sharp_{h_{\tau}}}) + \partial_{\tau}h_{\tau} \Delta_1 \psi_{\tau} + (\frac{4}{n-1} \partial_{\tau} \phi_{\tau} - \frac{1}{2} \Tr_{h_{\tau}}(\partial_{\tau} h_{\tau}) )\, \psi_{\tau}  \\
& \delta^{h_{\tau}} \psi_{\tau} + \frac{4}{n-1} \psi_{\tau}(\dd_{\Sigma} \phi^{\sharp_{h_{\tau}}}) = 0
\end{align*}
 
\noindent
where $\psi_{\tau}$ is the derived family of two-forms associated to $\bar{b} \in \Omega^2(\bar{Y})$.  
\end{lemma}

\begin{proof}
Let $(g,\bar{b},\bar{\phi})$ be a reducible gradient generalized Ricci soliton with reduction $( h_{\tau} , a_{\tau} , b_{\tau} , \phi_{\tau})$ satisfying the Maxwell equation in \eqref{eq:NSNSEinsteinframeII}. Using Equation \eqref{eq:decompositionH} together with Lemma \ref{lemma:hyperbolicderivatives}, we compute:
\begin{eqnarray*}
& \nabla^g_{\partial_{\tau}} H_{\bar{b}} = \nabla^g_{\partial_{\tau}} (\dd\tau \wedge \psi_{\tau} + H_{b_{\tau}}) = \dd\tau \wedge \nabla^g_{\partial_{\tau}}\psi_{\tau} + \nabla^g_{\partial_{\tau}} H_{b_{\tau}} \\
& = \dd\tau \wedge ( \partial_{\tau} \psi_{\tau} - \frac{1}{2}\partial_{\tau}h_{\tau} \Delta_1 \psi_{\tau}) + \partial_{\tau} H_{b_{\tau}} - \frac{1}{2} \partial_{\tau}h_{\tau} \Delta_1 H_{b_{\tau}} 
\end{eqnarray*}

\noindent
where, given any form $\alpha \in \Omega^{\bullet}(\Sigma)$, we have defined:
\begin{equation*}
 \partial_{\tau}h_{\tau} \Delta_1 \alpha = \sum_i (\partial_{\tau}h_{\tau})(e_i) \wedge \alpha(e_i)
\end{equation*}

\noindent
in terms of any orthonormal frame $(e_1 , \hdots , e_n)$. Similarly, for every $v\in T\Sigma$, we compute:
\begin{eqnarray*}
& \nabla^g_{v} H_{\bar{b}} = \nabla^g_{v} (\dd\tau \wedge \psi_{\tau} + H_{b_{\tau}}) = \nabla^g_{v} \dd\tau \wedge \psi_{\tau} + \dd\tau \wedge \nabla^g_{v}\psi_{\tau} + \nabla^g_{v} H_{b_{\tau}}\\
& = \frac{1}{2} (\partial_{\tau} h_{\tau})(v) \wedge \psi_{\tau} + \dd\tau \wedge \nabla^{h_{\tau}}_{v}\psi_{\tau} - \frac{1}{2} \dd \tau \wedge H_{b_{\tau}}(\partial_{\tau} h_{\tau} (v)) + \nabla^{h_{\tau}}_{v} H_{b_{\tau}}
\end{eqnarray*}

\noindent
From this, we obtain the divergence of $H_{\bar{g}}$:
\begin{eqnarray*}
& \delta^g H_{\bar{b}} = - (\nabla^g_{\partial_{\tau}} H_{\bar{b}})(\partial_{\tau}) - \sum_i (\nabla^g_{e_i} H_{\bar{b}})(e_i) \\
& =  \partial_{\tau}h_{\tau} \Delta_1 \psi_{\tau} - \partial_{\tau} \psi_{\tau} - \frac{1}{2} \Tr_{h_{\tau}}(\partial_{\tau} h_{\tau}) \psi_{\tau} - \dd \tau \wedge \delta^{h_{\tau}} \psi_{\tau} + \delta^{h_{\tau}} H_{b_{\tau}}
\end{eqnarray*}

\noindent
On the other hand, we have:
\begin{equation*}
H_{\bar{b}}(\dd\phi^{\sharp_{g}}) = H_{\bar{b}}(\partial_{\tau} \phi_{\tau} \partial_{\tau} + \dd_{\Sigma} \phi^{\sharp_{h_{\tau}}}) = \partial_{\tau} \phi_{\tau}\, \psi_{\tau} - \dd \tau \wedge \psi_{\tau}(\dd_{\Sigma} \phi^{\sharp_{h_{\tau}}}) + H_{b_{\tau}}(\dd_{\Sigma} \phi^{\sharp_{h_{\tau}}})
\end{equation*}

\noindent
Hence:
\begin{eqnarray*}
& \delta^{g} H_{\bar{b}} + \frac{4}{n-1}  H_{\bar{b}}(\dd\phi^{\sharp_{g}}) = \delta^{h_{\tau}} H_{b_{\tau}} + \frac{4}{n-1} H_{b_{\tau}}(\dd_{\Sigma} \phi^{\sharp_{h_{\tau}}}) + \partial_{\tau}h_{\tau} \Delta_1 \psi_{\tau} - \partial_{\tau} \psi_{\tau} + (\frac{4}{n-1} \partial_{\tau} \phi_{\tau} - \frac{1}{2} \Tr_{h_{\tau}}(\partial_{\tau} h_{\tau}) )\, \psi_{\tau}\\
& - \dd \tau \wedge (\delta^{h_{\tau}} \psi_{\tau} + \frac{4}{n-1} \psi_{\tau}(\dd_{\Sigma} \phi^{\sharp_{h_{\tau}}}))   
\end{eqnarray*}

\noindent
and we conclude.
\end{proof}

\begin{lemma}
\label{lemma:Dilatongh}
A tuple $(g , \bar{b}, \bar{\phi}) \in \Conf(\bar{\cP} , \bar{Y}, \bar{A})$ satisfies the dilaton equation in \eqref{eq:NSNSEinsteinframeII} if and only if its reduction $( h_\tau , a_\tau , b_\tau , \phi_\tau)$ satisfies the following differential system:
\begin{equation*}
\partial_{\tau}^2 \phi_{\tau} = \delta^{h_{\tau}} \dd_{\Sigma}\phi_{\tau} - \frac{1}{2} \partial_{\tau} \phi_{\tau} \Tr_{h_{\tau}}(\partial_{\tau} h_{\tau}) - e^{-\frac{4\phi_{\tau}}{n-1}} (\vert \psi_{\tau}\vert^2_{h_{\tau}} + \vert H_{b_\tau}\vert^2_{h_{\tau}}) - \lambda e^{\frac{2\phi_{\tau}}{n-1}}
\end{equation*}

\noindent
where $\psi_{\tau}$ is the derived family of two-forms associated to $\bar{b} \in \Omega^2(\bar{Y})$. 
\end{lemma}

\begin{proof}
Let $(g,\bar{b},\bar{\phi})$ be a reducible gradient generalized Ricci soliton with reduction $( h_{\tau} , a_{\tau} , b_{\tau} , \phi_{\tau})$ satisfying the dilaton equation in \eqref{eq:NSNSEinsteinframeII}. Using Equation \eqref{eq:decompositionH} together with Lemma \ref{lemma:hyperbolicderivatives}, we compute:
\begin{eqnarray*}
& \nabla^g_{\partial_{\tau}} \dd \bar{\phi} = \nabla^g_{\partial_{\tau}} (\partial_{\tau} \phi_{\tau}\, \dd \tau + \dd_{\Sigma} \phi_{\tau}) = \partial_{\tau}^2 \phi_{\tau}\, \dd \tau +   \partial_{\tau} \dd_{\Sigma} \phi_{\tau} - \frac{1}{2}  (\partial_{\tau} h_{\tau}) (\dd_{\Sigma} \phi_{\tau}^{\sharp_{h_{\tau}}})  \\
& \nabla^g_{v} \dd \bar{\phi} = \nabla^g_{v} (\partial_{\tau} \phi_{\tau}\, \dd \tau + \dd_{\Sigma} \phi_{\tau}) = \nabla^{h_{\tau}}_v \dd_{\Sigma}\phi_{\tau} - \frac{1}{2} (\partial_{\tau} h_{\tau})(\dd_{\Sigma}\phi_{\tau}^{\sharp_{h_{\tau}}} , v) \dd\tau + v(\partial_{\tau} \phi_{\tau}) \dd \tau + \frac{1}{2}  \partial_{\tau} \phi_{\tau} (\partial_{\tau} h_{\tau})(v)  
\end{eqnarray*}

\noindent
Hence:
\begin{eqnarray*}
\delta^g \dd \bar{\phi} = - (\nabla^g_{\partial_{\tau}} \dd \bar{\phi})(\partial_{\tau}) - \sum_i (\nabla^g_{e_i} \dd \bar{\phi})(e_i)  = \delta^{h_{\tau}} \dd_{\Sigma}\phi_{\tau} - \partial_{\tau}^2 \phi_{\tau} - \frac{1}{2} \partial_{\tau} \phi_{\tau} \Tr_{h_{\tau}}(\partial_{\tau} h_{\tau})
\end{eqnarray*}

\noindent
and thus we conclude.
\end{proof}

\noindent
The previous lemmata allow us to characterize reducible gradient generalized Ricci solitons $(g,\bar{b},\bar{\phi})$ on $(\bar{\cP},\bar{Y},\bar{A})$ in terms of constrained evolution flows for tuples $( h_{\tau} , a_{\tau} , b_{\tau} , \phi_{\tau})$ on the reduction $(\cP, Y , A_{\tau}  , \Psi_{\tau})$ of $(\bar{\cP},\bar{A},\bar{Y})$.  
\begin{prop}
\label{prop:CylindricalGeneralizedSolitonSystem}
A reducible tuple $(g , \bar{b}, \bar{\phi}) \in \Conf(\bar{\cP} , \bar{A}, \bar{Y})$ is a gradient Ricci soliton if and only if its reduction $( h_{\tau} , a_{\tau} , b_{\tau} , \phi_{\tau})$ on $(\cP, Y , A_{\tau}  , \Psi_{\tau})$ satisfies the following evolution equations:
\begin{align}
\partial_{\tau}^2 h_{\tau} & = 2 \Ric^{h_{\tau}} +  \partial_{\tau} h_{\tau} \circ_{h_{\tau}} \partial_{\tau} h_{\tau} - \frac{1}{2} \Tr_{h_{\tau}}(\partial_{\tau} h_{\tau}) \partial_{\tau} h_{\tau}   - \frac{2}{n-1}\,\dd_{\Sigma}\phi_{\tau} \otimes \dd_{\Sigma}\phi_{\tau} \nonumber \\
& + \frac{2}{\,n-1\,} (e^{-\tfrac{4\phi}{n-1}} (\vert \psi_{\tau}\vert_{h_{\tau}}^2 + \vert H_{b_{\tau}}\vert_{h_{\tau}}^2) +   \lambda \, e^{ \frac{2\phi}{n-1}})\, h_{\tau} -  e^{-\tfrac{4\phi}{n-1}}\,(\psi_{\tau} \circ_{h_{\tau}} \psi_{\tau} + H_{b_{\tau}} \circ_{h_{\tau}} H_{b_{\tau}}) \nonumber \\
\partial_{\tau} \psi_{\tau} & = \delta^{h_{\tau}} H_{b_{\tau}} + \frac{4}{n-1} H_{b_{\tau}}(\dd_{\Sigma} \phi^{\sharp_{h_{\tau}}}) + \partial_{\tau}h_{\tau} \Delta_1 \psi_{\tau} + (\frac{4}{n-1} \partial_{\tau} \phi_{\tau} - \frac{1}{2} \Tr_{h_{\tau}}(\partial_{\tau} h_{\tau}) )\, \psi_{\tau}   \label{eq:evolutionequations} \\
\partial_{\tau}^2 \phi_{\tau} & = \delta^{h_{\tau}} \dd_{\Sigma}\phi_{\tau} - \frac{1}{2} \partial_{\tau} \phi_{\tau} \Tr_{h_{\tau}}(\partial_{\tau} h_{\tau}) - e^{-\frac{4\phi_{\tau}}{n-1}} (\vert \psi_{\tau}\vert^2_{h_{\tau}} + \vert H_{b_\tau}\vert^2_{h_{\tau}}) - \lambda e^{\frac{2\phi_{\tau}}{n-1}}  \nonumber
\end{align}
together with the following system of time-dependent constraints:
\begin{align}
& s^{h_{\tau}} +  \frac{1}{4} \vert \partial_{\tau} h_{\tau} \vert_{h_{\tau}}^2 - \frac{1}{4} \Tr_{h_{\tau}}( \partial_{\tau} h_{\tau})^2  + \lambda    e^{ \frac{2\phi_{\tau}}{n-1}} + \frac{1}{n-1} ( (\partial_{\tau} \phi_{\tau})^2 - \vert \dd_{\Sigma}\phi_{\tau} \vert^2_{h_{\tau}})  +  \frac{1}{2} ( \vert \psi_{\tau} \vert^2_{h_{\tau}}   -  \vert H_{b_{\tau}} \vert^2_{h_{\tau}} ) e^{-\tfrac{4\phi_{\tau}}{n-1}}   = 0 \nonumber\\
&   \delta^{h_{\tau}}\partial_{\tau} h_{\tau} + \dd_\Sigma \Tr_{h_{\tau}}(\partial_{\tau} h_{\tau}) + \frac{2}{n-1}\, \partial_{\tau} \phi_{\tau}\, \dd_{\Sigma} \phi_{\tau} + e^{-\tfrac{4\phi_{\tau}}{n-1}}\,\psi_{\tau} \lrcorner_{h_{\tau}} H_{b_{\tau}} = 0 \label{eq:constraints} \\
& \delta^{h_{\tau}} \psi_{\tau} + \frac{4}{n-1} \psi_{\tau}(\dd_{\Sigma} \phi^{\sharp_{h_{\tau}}}) = 0\nonumber
\end{align}

\noindent
where $\psi_{\tau}$ is the derived family of two-forms associated to $\bar{b} \in \Omega^2(\bar{Y})$.
\end{prop}

\noindent
The previous proposition presents the evolution and time-dependent constraint equations, given respectively in  \eqref{eq:evolutionequations} and  \eqref{eq:constraints}, of the Cauchy problem for the gradient generalized Ricci soliton system on a bundle gerbe $(\cP,Y,A)$ defined on $\Sigma$. By restricting the time-dependent constraint equations to $\tau = 0$, we obtain the \emph{constraint equations} of this Cauchy problem, namely:
\begin{align}
& s^{h} +  \vert \Theta \vert_{h}^2 -  \Tr_{h}( \Theta)^2  + \lambda   \, e^{ \frac{2\phi}{n-1}} + \frac{1}{n-1} ( \rho^2 - \vert \dd \phi \vert^2_{h})  +  \frac{1}{2} ( \vert \psi \vert^2_{h}   -  \vert H_{b} \vert^2_{h} ) e^{-\tfrac{4\phi}{n-1}}   = 0 \label{eq:constraints1}\\
&   \delta^{h}\Theta + \dd \Tr_{h} (\Theta) + \frac{1}{n-1}\, \rho \, \dd \phi  + \frac{1}{2} e^{-\tfrac{4\phi}{n-1}}\,\psi  \lrcorner_{h} H_{b} = 0  \, , \qquad  \delta^{h} \psi + \frac{4}{n-1} \psi(\dd \phi^{\sharp_{h}}) = 0 \label{eq:constraints2}
\end{align}

\noindent
We consider these equations as a differential system for tuples $(h,\Theta , b,\psi,\phi,\rho)$ consisting of a Riemannian metric $h$ on $\Sigma$, a symmetric two-tensor $\Theta \in \Gamma(T^{\ast}\Sigma \odot T^{\ast}\Sigma)$, a curving $b$ on $(\cP,Y,A)$, a two-form $\psi \in \Omega^2(\Sigma)$ and a pair of functions $\phi , \rho \in C^{\infty}(\Sigma)$. Let $\Met(\Sigma)$ denote the convex cone of Riemannian metrics on $\Sigma$, and denote by $\Omega^2_{\cC}(Y)$ the affine space of curvings on $(\cP,A,\pi,\mu)$. Then, given a tuple $(h,\Theta , b,\psi,\phi,\rho)$ we observe:
\begin{eqnarray*}
\Theta \in T_h \Met(\Sigma) \, , \qquad \psi \in T_{b} \Omega^2_{\cC}(Y)\, , \qquad \rho \in T_{\phi} C^{\infty}(\Sigma)
\end{eqnarray*}

\noindent
where $T_h \Met(\Sigma)$ denotes the tangent space of $\Met(\Sigma)$ at $h\in \Met(M)$, and similarly for $\in T_{b} \Omega^2_{\cC}(Y)$ and $T_{\phi} C^{\infty}(\Sigma)$. Hence, we can give a geometric meaning to the tuples $(h,\Theta , b,\psi,\phi,\rho)$ by identifying them as elements of the following space:
\begin{equation*}
(h,\Theta) \oplus (b,\psi)\oplus (\phi,\rho) \in T(\Met(\Sigma) \times   \Omega^2_{\cC}(Y) \times   C^{\infty}(\Sigma))
\end{equation*}

\noindent
In particular, we conclude that the \emph{configuration space} of the constraint equations of the gradient generalized Ricci soliton system is $T(\Met(\Sigma) \times   \Omega^2_{\cC}(Y) \times   C^{\infty}(\Sigma))$. Clearly, equations \eqref{eq:constraints1} and \eqref{eq:constraints2} are necessary conditions for $(\cP, Y , A_{\tau}  , \Psi_{\tau})$ to admit a solution to the equations of Proposition \ref{prop:CylindricalGeneralizedSolitonSystem}.  We now show that, by the Cauchy-Kovalevskaya theorem, the system of evolution equations \eqref{eq:evolutionequations} has a unique solution in $\Sigma$ for short times. 

\begin{lemma}
\label{lemma:existencevolution}
Let $a_{\tau}$ be a fixed analytic family of one-forms on $Y$ satisfying \eqref{eq:conditionat}. Then, for every real analytic tuple $(h,\Theta,b,\psi,\phi,\rho) \in T(\Met(\Sigma) \times   \Omega^2_{\cC}(Y) \times   C^{\infty}(\Sigma))$ there exists a unique analytic germ  $( h_{\tau} , b_{\tau} , \phi_{\tau})$ such that $( h_{\tau} , a_{\tau}, b_{\tau} , \phi_{\tau})$ satisfies the evolution equations \eqref{eq:evolutionequations} on $(\cP, Y , A_{\tau}  , \Psi_{\tau})$ and restricts to $(h,\Theta,b,\psi,\phi,\rho)$ at $\tau = 0$, that is:
\begin{equation*}
h = h_{\tau}\vert_{\tau=0} \, , \quad \Theta = \frac{1}{2}\partial_{\tau} h_{\tau}\vert_{\tau = 0}   \, , \quad b = b_{\tau}\vert_{\tau = 0}\, , \quad \psi = \psi_{\tau}\vert_{\tau = 0}\, , \quad \phi = \phi_{\tau}\vert_{\tau = 0}\, , \quad \rho = \partial_{\tau}\phi_{\tau}\vert_{\tau = 0} 
\end{equation*}
\end{lemma}

\begin{proof}
We fix real analytic initial data  $(h,\Theta,b,\psi,\phi,\rho) \in T(\Met(\Sigma) \times   \Omega^2_{\cC}(Y) \times   C^{\infty}(\Sigma))$. Let $a_{\tau}$ be a fixed real-analytic family of one-forms on $Y$ satisfying \eqref{eq:conditionat}. For every real-analytic family $(h_{\tau}, b_{\tau}, \phi_{\tau})$, we introduce the new global variable:
\begin{equation*}
\theta_{\tau} = \int_0^{\tau} \psi_{s} \, \dd s 
\end{equation*}

\noindent
where $\psi_s$ denotes the derived family of two-forms associated with $(h_{\tau}, a_{\tau}, b_{\tau}, \phi_{\tau})$. We have:
\begin{equation*}
H_{b_{\tau}}  = H_b + \dd_{\Sigma}\theta_{\tau}
\end{equation*}

\noindent
where $H_b$ is the curvature of the initial curving, which does not depend on $\tau$ and hence can be considered as a global \emph{background} three-form on $\Sigma$ determined by the given choice of initial data. We fix a good open cover $(U_r)$ of $\Sigma$ for some index $r$, and we consider the evolution equations \eqref{eq:evolutionequations} restricted to each $U_r$. Using the fact that $\psi_{\tau} = \partial_{\tau} \theta_{\tau}$, on each contractible patch $U_r$ we can consider \eqref{eq:evolutionequations} as a system of evolution equations of the form:
\begin{align*}
\partial_{\tau}^2 h_{\tau} &= \cF_1(h_{\tau}, \partial_{\tau} h_{\tau}, \theta_{\tau}, \partial_{\tau} \theta_{\tau}, \phi_{\tau}, \partial_{\tau} \phi_{\tau}),\\
\partial_{\tau}^2 \theta_{\tau} &= \cF_2(h_{\tau}, \partial_{\tau} h_{\tau}, \theta_{\tau}, \partial_{\tau} \theta_{\tau}, \phi_{\tau}, \partial_{\tau} \phi_{\tau}),\\
\partial_{\tau}^2 \phi_{\tau} &= \cF_3(h_{\tau}, \partial_{\tau} h_{\tau}, \theta_{\tau}, \partial_{\tau} \theta_{\tau}, \phi_{\tau}, \partial_{\tau} \phi_{\tau}),
\end{align*}

\noindent
for certain real-analytic functions $\cF_1$, $\cF_2$, and $\cF_3$ of the variables $(h_{\tau}, \theta_{\tau}, \phi_{\tau})$ with given initial data $(h,\Theta,0,\psi,\phi,\rho)$. By the Cauchy-Kovalevskaya theorem (see, for example, \cite{Folland}), for this real-analytic initial data we obtain a unique real-analytic solution $(h_{\tau}, \theta_{\tau}, \phi_{\tau})_r$ on each $U_r$ for some interval $I_r = (-\epsilon_r, \epsilon_r)$ containing zero. By the uniqueness of each solution on $U_r$, the variables $(h_{\tau}, \theta_{\tau}, \phi_{\tau})$ must coincide on the overlaps of pairs of open sets in $(U_r)$. Note that, crucially, the two-form $\psi$ contained in the fixed initial data is globally defined on $\Sigma$.\medskip 

The agreement on overlaps allows us to glue the local solutions to obtain a globally defined solution $(h_{\tau}, \theta_{\tau}, \phi_{\tau})$ on a neighborhood $\mathcal{W} \subset \Sigma \times \mathbb{R}$ containing the initial hypersurface $\Sigma \times \{0\}$. Since $\Sigma$ is a paracompact manifold, we may assume the cover $(U_r)$ is locally finite, in which case it follows that $\mathcal{W} \subset \Sigma \times \mathbb{R}$ is an open neighborhood of $\Sigma \times \{0\}$. While the length of the $\tau$-interval may depend on the index $r$ labeling the contractible open set $U_r$, th gluing uniquely defines a global real-analytic germ around $\Sigma$. Given such a global germ $(h_{\tau}, \theta_{\tau}, \phi_{\tau})$, and the fixed real-analytic family $a_{\tau}$, we set: 
\begin{equation*}
b_{\tau} = \int_0^{\tau} (\pi^{\ast} \psi_s + \dd_{Y} a_{s})\, \dd s + b \in \Omega^2(Y),
\end{equation*}

\noindent
and consequently we obtain a solution $(h_{\tau}, a_{\tau}, b_{\tau}, \phi_{\tau})$ solving \eqref{eq:evolutionequations} that restricts to $(h,\Theta,b,\psi,\phi,\rho)$ at $\tau = 0$. Due to the uniqueness of the local solutions on each open set $U_r$, together with the previous formula determining $b_{\tau}$ uniquely, we conclude that the germ of the solution $(h_{\tau}, a_{\tau}, b_{\tau}, \phi_{\tau})$ is unique for a fixed $a_{\tau}$.
\end{proof}

\noindent
We now study the compatibility of solutions to the evolution equations \eqref{eq:evolutionequations} with the time-dependent constraints \eqref{eq:constraints} to establish the well-posedness of the analytic Cauchy problem for the gradient generalized Ricci soliton system.

\begin{thm}
\label{thm:CylindricalWellPosedness}
Let $\Sigma$ be an oriented $n$-dimensional manifold equipped with a bundle gerbe $(\cP,Y)$ with analytic connective structure $A$. For every solution $(h,\Theta , b,\psi,\phi,\rho)$ of the constraint equations \eqref{eq:constraints1} and \eqref{eq:constraints2} on $(\cP,Y,A)$, there exists an $(n+1)$-dimensional oriented manifold $M$, a bundle gerbe with connective structure $(\bar{\cP} , \bar{Y}, \bar{\cA} )$ on $M$ inducing $A$ on $(\cP,Y)$, an embedding $\Sigma  \subset M$ and a gradient generalized Ricci soliton $(g,\bar{b},\bar{\phi})$ around $\Sigma \subset M$ such that:
\begin{equation*}
h =  g\vert_{\Sigma}\, , \qquad b = \bar{b}\vert_Y\, , \qquad \phi = \bar{\phi} \vert_{\Sigma}
\end{equation*}

\noindent
and such that $\Theta$ is the second fundamental form of the embedding $(\Sigma , h)\subset (M,g)$.
\end{thm}

\begin{remark}
Identifying a small enough tubular neighborhood of $\Sigma\subset M$ with $\cI\times\Sigma$ and fixing families $a_{\tau}$, $A_{\tau}$ and $\Psi_{\tau}$ as introduced in Lemma \ref{lemma:existencevolution}, in the previous theorem we obtain a unique germ of a gradient generalized Ricci soliton that recovers $(h,\Theta , b,\psi,\phi,\rho)$ at $\tau = 0$.
\end{remark}

\begin{proof}
On $(\cP,Y,A)$ we fix an analytic family $\Psi_{\tau}$ of functions, an analytic family of connective structures $A_{\tau}$ and analytic family of one-forms $a_{\tau} \in \Omega^1(Y)$, such that $A = A_{\tau}\vert_{\tau = 0}$ and such that Equation \eqref{eq:conditionat} holds. By Lemma \ref{lemma:existencevolution} we can assume that there exists a family $(h_{\tau},a_{\tau},b_{\tau},\phi_{\tau})$ solving the evolution equations \eqref{eq:evolutionequations} with initial data $(h,\Theta , b,\psi,\phi,\rho)$. By Proposition \ref{prop:CylindricalGeneralizedSolitonSystem} it is enough to prove that the time-dependent constraints \eqref{eq:constraints} are satisfied by such $(h_{\tau},a_{\tau},b_{\tau},\phi_{\tau})$, in which case we obtain the desired gradient generalized Ricci soliton by taking $M = \cI\times \Sigma$, pulling back $(\cP,Y,A_{\tau},\Psi_{\tau})$ to $M$ and setting:
\begin{equation*}
g = \dd\tau\otimes \dd \tau + h_{\tau}\, , \qquad \bar{b} = \dd \tau \wedge a_{\tau} + b_{\tau}\, , \qquad \bar{\phi} = \phi_{\tau}
\end{equation*}

\noindent
Given $\Psi_{\tau}$, $A_{\tau}$, and $a_{\tau}$ as specified above, for every $(h_{\tau},b_{\tau},\phi_{\tau})$ such that $(h_{\tau},a_{\tau},b_{\tau},\phi_{\tau})$ satisfies the evolution equations \eqref{eq:evolutionequations}, we use the time-dependent constraint equations \eqref{eq:constraints} to define: 
\begin{align*}
C_{1\tau} &:= \delta^{h_{\tau}}\partial_{\tau} h_{\tau} + \dd_\Sigma \Tr_{h_{\tau}}(\partial_{\tau} h_{\tau}) + \frac{2}{n-1}\, \partial_{\tau} \phi_{\tau}\, \dd_{\Sigma} \phi_{\tau} + e^{-\tfrac{4\phi_{\tau}}{n-1}}\,\psi_{\tau} \lrcorner_{h_{\tau}} H_{b_{\tau}} \\
C_{2\tau} &:= s^{h_{\tau}} +  \frac{1}{4} \vert \partial_{\tau} h_{\tau} \vert_{h_{\tau}}^2 - \frac{1}{4} \Tr_{h_{\tau}}( \partial_{\tau} h_{\tau})^2  + \lambda    e^{ \frac{2\phi_{\tau}}{n-1}} + \frac{1}{n-1} ( (\partial_{\tau} \phi_{\tau})^2 - \vert \dd_{\Sigma}\phi_{\tau} \vert^2_{h_{\tau}})  +  \frac{1}{2} ( \vert \psi_{\tau} \vert^2_{h_{\tau}}   -  \vert H_{b_{\tau}} \vert^2_{h_{\tau}} ) e^{-\tfrac{4\phi_{\tau}}{n-1}}\\
C_{3\tau} &:= \delta^{h_{\tau}} \psi_{\tau} + \frac{4}{n-1} \psi_{\tau}(\dd_{\Sigma} \phi^{\sharp_{h_{\tau}}})
\end{align*}

\noindent
Hence, we need to show that $C_{1\tau} = C_{2\tau} = C_{3\tau} = 0$. A series of long computations, which we defer to Appendix \ref{app:computations}, shows that $C_{1\tau}$, $C_{2\tau}$, and $C_{3\tau}$ satisfy the following differential system: 
\begin{align*}
\partial_{\tau} C_{1\tau} &= -\frac{1}{2}\Tr_{h_{\tau}}(\partial_{\tau} h_{\tau})C_{1\tau} + \frac{1}{2} \dd_\Sigma C_{2\tau} - \frac{e^{-\frac{4\phi_{\tau}}{n-1}}}{2}  C_{3\tau}\lrcorner_{h_{\tau}} \partial_{\tau}\psi_{\tau}  \\
\partial_{\tau} C_{2\tau} &= - 2 \delta^{h_{\tau}} C_{1\tau} - \Tr_{h_{\tau}}(\partial_{\tau} h_{\tau})C_{2\tau}  \\
\partial_{\tau} C_{3\tau} &= \partial_{\tau} h_{\tau}(C_{3\tau}^{\sharp_{h_{\tau}}}) - \left(\frac{1}{2}\Tr_{h_{\tau}}(\partial_{\tau} h_{\tau}) - \frac{4}{n-1}\partial_{\tau} \phi_{\tau} \right)C_{3\tau} 
\end{align*}

\noindent
Since, by assumption, we have:
\begin{equation*}
C_{1\tau}\vert_{\tau=0} = C_{2\tau}\vert_{\tau=0} = C_{3\tau}\vert_{\tau=0} = 0
\end{equation*}

\noindent
by Cauchy-Kovalevskaya it follows that $C_{1\tau}  = C_{2\tau}  = C_{3\tau}  = 0$ for every $\tau \in \cI$.
\end{proof}

\begin{remark}
As it happens with the Cauchy problem for Riemannian Einstein metrics \cite{Koiso1981}, analyticity of the initial data is a necessary condition for the previous theorem to hold and hence cannot be relaxed.  
\end{remark}


\subsection{The Cauchy problem in three dimensions} 


In what follows, we consider initial data in two dimensions, with the aim of proving an existence result for the NS generalized Ricci soliton system in three dimensions. In two dimensions the constraint equations \eqref{eq:constraints1} and \eqref{eq:constraints2} with $\lambda = 0$ simplify to:
\begin{align*}
& s^{h} +  \vert \Theta \vert_{h}^2 -  \Tr_{h}( \Theta)^2  +   \rho^2 - \vert \dd \phi \vert^2_{h}  +  \frac{1}{2}  \vert \psi \vert^2_{h}   e^{-4\phi}   = 0 \\
&   \delta^{h}\Theta + \dd \Tr_{h} (\Theta) +  \rho \, \dd \phi  = 0  \, , \qquad  \delta^{h} \psi + 4 \psi(\dd \phi^{\sharp_{h}}) = 0 
\end{align*}

\noindent
The general solution to the equation $\delta^{h} \psi + 4 \psi(\dd \phi^{\sharp_{h}}) = 0$ is readily found to be:
\begin{equation*}
\psi = c e^{4\phi} \nu_h\, , \qquad c\in \mathbb{R}
\end{equation*}

\noindent
where $\nu_h\in \Omega^2(\Sigma)$ is the Riemannian volume form on $(\Sigma,h)$. Plugging this expression into the remaining equations above, the system reduces to:
\begin{align}
\label{eq:reducedconstraints3d}
s^{h} +  \vert \Theta \vert_{h}^2 -  \Tr_{h}( \Theta)^2  +   \rho^2 - \vert \dd \phi \vert^2_{h} +  \frac{c^2}{2}     e^{4\phi}   = 0 \, , \qquad \delta^{h}\Theta + \dd \Tr_{h} (\Theta) +  \rho \, \dd \phi  = 0  
\end{align}

\noindent
We shall use the following lemma.
\begin{lemma}[Lemma 3.4, \cite{BryanWentworth}]
\label{lemma:conformalsecondconstraint}
Let $(N,g)$ be a compact Riemannian manifold of any dimension and let $A$, $B$, and $w$ be smooth functions with $A$ and $B$ non-negative, $\int_N (A-B)\nu_g > 0$, and $\int_N w\,\nu_g > 0$. Then the equation
\begin{align*}
    \delta^g\dd u + Ae^u - Be^{-u} - w = 0
\end{align*}
has a unique $C^\infty(N)$ solution.
\end{lemma}

\noindent 
We now prove the existence of solutions to the constraint equations \eqref{eq:constraints1} and \eqref{eq:constraints2} on every compact Riemann surface.
\begin{thm}
\label{thm:wellposednessconstraints3d}
Let $\Sigma$ be a compact and oriented two-dimensional manifold. For every conformal class $[h]$ of Riemannian metrics on $\Sigma$ there exists a solution $(h,\Theta,b,\psi,\phi,\rho)$ to the constraint equations \eqref{eq:constraints1} and \eqref{eq:constraints2} on $\Sigma$ with $h\in [h]$.  
\end{thm}

\begin{proof}
It is enough to solve equations \eqref{eq:reducedconstraints3d}. We consider first the second equation in \eqref{eq:reducedconstraints3d}. To solve it, we set:
\begin{equation*}
\Theta = f\, h + \Theta_o \, , \qquad \rho = - F^{\prime}(\phi) 
\end{equation*}

\noindent
for a function $f\in C^{\infty}(\Sigma)$, a traceless and divergence-free two-tensor $\Theta_o \in \Gamma(T^{\ast}\Sigma\odot T^{\ast}\Sigma)$ and a real smooth function $F^{\prime}\colon \mathbb{R}\to \mathbb{R}$. Substituting in the second equation of \eqref{eq:reducedconstraints3d}, we obtain:
\begin{equation*}
\dd f - \dd F (\phi) = 0
\end{equation*}

\noindent
Hence, $f = F(\phi) + k$ for a real constant $k\in\mathbb{R}$, where $F = \int F^{\prime}$. Substituting this into the first equation in \eqref{eq:reducedconstraints3d}, we obtain:
\begin{equation*}
s^h + \vert \Theta_o \vert^2_h - 2 (k + F(\phi))^2  +   F^{\prime}(\phi) ^2 - \vert \dd \phi \vert^2_{h} +  \frac{c^2}{2}     e^{4\phi}   = 0
\end{equation*}

\noindent
Given a conformal class of metrics $[h]$ denote by $h_{\mu}\in [h]$ the unique representative with constant scalar curvature $\mu \in \{ - 1 , 0 , 1 \}$. Any other representative in $[h]$ is of the form $h = e^u h_{\mu}$ for a unique smooth function $u\in C^{\infty}(\Sigma)$. Plugging $h = e^u h_{\mu}$ into the previous equation, the first equation in \eqref{eq:reducedconstraints3d} becomes equivalent to:
\begin{equation*}
 \delta^{h_{\mu}}\dd u +   (F^{\prime}(\phi)^2  +  \frac{c^2}{2}     e^{4\phi} - 2(k + F(\phi))^2) e^u   + \vert \Theta_o \vert_{h_{\mu}}^2 e^{-u}   - \vert \dd \phi \vert^2_{h_{\mu}} + s^{h_{\mu}}  = 0     
\end{equation*}

\noindent
Comparing with the statement of Lemma \ref{lemma:conformalsecondconstraint}, we have:
\begin{equation*}
A = F^{\prime}(\phi)^2  +  \frac{c^2}{2} e^{4\phi} - 2(k + F(\phi))^2 \, , \qquad B= - \vert \Theta_o \vert_{h_{\mu}}^2\, , \qquad w = \vert \dd \phi \vert^2_{h} - s^{h_{\mu}}
\end{equation*}

\noindent
Assume momentarily that $\mu = 0 , -1$, that is, assume that $(\Sigma,h_{\mu})$ is either flat or hyperbolic, and choose $F$, $c$ and $k$ such that:
\begin{equation*}
F^{\prime}(\phi)^2  +  \frac{c^2}{2}     e^{4\phi} > 2 (k + F(\phi))^2\, , \qquad \Theta_o = 0
\end{equation*}

\noindent
Then by Lemma \ref{lemma:conformalsecondconstraint} there exists a unique representative $h\in [h_{\mu}]$ solving the first equation in \eqref{eq:reducedconstraints3d}, and therefore we obtain a solution to the constraint equations \eqref{eq:constraints1} and \eqref{eq:constraints2} on $\Sigma$ with $h\in [h]$ in the given conformal class.  If $\mu = 1$, then appropriately choosing $\phi$ we can guarantee that:
\begin{equation*}
\int (\vert \dd \phi \vert^2_{h} - 1) \nu_{h_1} \ge 0
\end{equation*}

\noindent
and hence we similarly obtain a solution on the sphere.
\end{proof}
 
\noindent
As a consequence of Theorem \ref{thm:wellposednessconstraints3d}, we obtain the following corollary on the topology of NS gradient generalized Ricci solitons in three dimensions. 
\begin{cor}
Let $(\Sigma,h_{\mu})$ be a punctured Riemann surface. Then, there exists an NS gradient generalized Ricci soliton $(g,b,\phi)$ on a bundle gerbe over a three-manifold $M$, and an embedding $\iota \colon \Sigma \hookrightarrow M$ such that $\iota^{\ast}g \in [h_{\mu}]$. In particular, there exist infinitely many different topological types of three-dimensional NS gradient generalized Ricci solitons.
\end{cor}

\noindent
A much more difficult problem would be to prove the previous corollary for \emph{complete} NS generalized Ricci solitons, as studied in \cite{PodestaRaffero} for a certain class of three-dimensional generalized Ricci solitons.


\appendix


\section{Differential relations satisfied by the constraints}
\label{app:computations}

 
\noindent
Recall the following definitions in terms of the constraint equations \eqref{eq:constraints}:
\begin{align*}
C_{1\tau} &:= \delta^{h_{\tau}}\partial_{\tau} h_{\tau} + \dd_\Sigma \Tr_{h_{\tau}}(\partial_{\tau} h_{\tau}) + \frac{2}{n-1}\, \partial_{\tau} \phi_{\tau}\, \dd_{\Sigma} \phi_{\tau} + e^{-\tfrac{4\phi_{\tau}}{n-1}}\,\psi_{\tau} \lrcorner_{h_{\tau}} H_{b_{\tau}} \\
C_{2\tau} &:= s^{h_{\tau}} +  \frac{1}{4} \vert \partial_{\tau} h_{\tau} \vert_{h_{\tau}}^2 - \frac{1}{4} \Tr_{h_{\tau}}( \partial_{\tau} h_{\tau})^2  + \lambda    e^{ \frac{2\phi_{\tau}}{n-1}} + \frac{1}{n-1} ( (\partial_{\tau} \phi_{\tau})^2 - \vert \dd_{\Sigma}\phi_{\tau} \vert^2_{h_{\tau}})  +  \frac{1}{2} ( \vert \psi_{\tau} \vert^2_{h_{\tau}}   -  \vert H_{b_{\tau}} \vert^2_{h_{\tau}} ) e^{-\tfrac{4\phi_{\tau}}{n-1}}\\
C_{3\tau} &:= \delta^{h_{\tau}} \psi_{\tau} + \frac{4}{n-1} \psi_{\tau}(\dd_{\Sigma} \phi^{\sharp_{h_{\tau}}})
\end{align*}

\noindent
We will show that the $\tau$-derivatives of $C_{1\tau}$, $C_{2\tau}$ and $C_{3\tau}$ ($i \in \{1,2,3\}$) are specific functions of $C_{1\tau}$, $C_{2\tau}$ and $C_{3\tau}$. This is the key result used in Theorem \ref{thm:CylindricalWellPosedness} to prove that if all $C_{i\tau}$ are zero at $\tau = 0$, then they are all zero for every $\tau \in \cI$, and thus are \emph{preserved} by the evolution equations \eqref{eq:evolutionequations}.
\begin{prop}
\label{prop:firstconstraint}
Let $(g , \bar{b}, \bar{\phi}) \in \Conf(\bar{\cP} , \bar{A}, \bar{Y})$ be a gradient Ricci soliton with reduction $( h_{\tau} , a_{\tau} , b_{\tau} , \phi_{\tau})$ on $(\cP, Y , A_{\tau}  , \Psi_{\tau})$. Then, the following equation holds:
\begin{equation*}
\partial_{\tau} C_{1\tau} = - \tfrac{1}{2}\Tr_{h_{\tau}}(\partial_{\tau}h_{\tau})C_{1\tau} + \dd_\Sigma C_{2\tau} - e^{-\tfrac{4\phi_{\tau}}{n-1}} C_{3\tau} \lrcorner_{h_{\tau}} \psi_{\tau}
\end{equation*}
\end{prop}

\begin{proof}
The result is obtained through a long but direct computation. First, use the following identity, which is proven in \cite[Proposition 6]{Koiso1981}: 
\begin{equation}
\begin{gathered}
\partial_{\tau} \left( \nabla^{h_{\tau}*}(\partial_{\tau}h_{\tau}) + \dd_\Sigma\Tr_{h_{\tau}}(\partial_{\tau}h_{\tau}) \right) = - \tfrac{1}{2} \Tr_{h_{\tau}}(\partial_{\tau}h_{\tau})\left( \nabla^{h_{\tau} \ast}(\partial_{\tau}h_{\tau}) + \dd_\Sigma\Tr_{h_{\tau}}(\partial_{\tau}h_{\tau}) \right) \\
+\, \dd_\Sigma\left(s^{h_{\tau}} + \tfrac{1}{4}\vert \partial_{\tau}h_{\tau}\vert_{h_{\tau}}^2 - \tfrac{1}{4}(\Tr_{h_{\tau}}\partial_{\tau}h_{\tau})^2 \right) - 2 \dd_\Sigma\Tr_{h_{\tau}} \big(\Ric^g\vert _\Sigma\big) - 2 \nabla^{h_{\tau} \ast} (\Ric^g\vert _\Sigma) 
\end{gathered}
\end{equation}

\noindent
Substituting the expression of $\Ric^g\vert_\Sigma$ given in \eqref{eq:NSNSEinsteinframeII}, one obtains:
\begin{equation}
\label{eq:gausscodazziconstraintNSNS1}
\begin{gathered}
\partial_{\tau} \left( \nabla^{h_{\tau} \ast}(\partial_{\tau}h_{\tau}) + \dd_\Sigma\Tr_{h_{\tau}}(\partial_{\tau}h_{\tau}) \right) = -\tfrac{1}{2} \Tr_{h_{\tau}}\big(\partial_{\tau}h_{\tau}\big)\left( \nabla^{h_{\tau} \ast} (\partial_{\tau}h_{\tau}) + \dd_\Sigma\Tr_{h_{\tau}}(\partial_{\tau}h_{\tau}) \right) \\
+\, \dd_\Sigma\bigg( s^{h_{\tau}} + \tfrac{1}{4}\vert \partial_{\tau}h_{\tau}\vert^2_{h_{\tau}} - \tfrac{1}{4}(\Tr_{h_{\tau}}\partial_{\tau}h_{\tau})^2 - \tfrac{2}{n-1}\vert \dd_\Sigma\phi_{\tau}\vert _{h_{\tau}}^2 - e^{-\tfrac{4\phi_{\tau}}{n-1}}\vert H_{b_{\tau}}\vert _{h_{\tau}}^2  - 2\lambda e^{\tfrac{2\phi_{\tau}}{n-1}} \bigg) \\
-\, 2\nabla^{h_{\tau} \ast}\bigg( \tfrac{1}{n-1}\dd_\Sigma\phi_{\tau} \otimes \dd_\Sigma\phi_{\tau} + \tfrac{1}{2}e^{-\tfrac{4\phi_{\tau}}{n-1}}\big( \psi_{\tau}\circ_{h_{\tau}} \psi_{\tau} + H_{b_{\tau}}\circ_{h_{\tau}} H_{b_{\tau}} \big) \bigg)  
\end{gathered}
\end{equation}

\noindent
Second, we compute the following quantity:
\begin{equation*}
\partial_{\tau}\left( \frac{2}{n-1}\, \partial_{\tau} \phi_{\tau}\, \dd_{\Sigma} \phi_{\tau} + e^{-\tfrac{4\phi_{\tau}}{n-1}}\,\psi_{\tau} \lrcorner_{h_{\tau}} H_{b_{\tau}} \right)
\end{equation*}

\noindent
in two steps. On the one hand, we have:
\begin{equation*}
\begin{gathered}
\partial_{\tau} \bigg(\tfrac{2}{n-1}\partial_{\tau}\phi_{\tau}\, \dd_\Sigma\phi_{\tau}\bigg) = \tfrac{2}{n-1}\left( \partial_{\tau}^2\phi_{\tau} \dd_\Sigma\phi_{\tau} + \partial_{\tau}\phi_{\tau}\, \dd_\Sigma\partial_{\tau}\phi_{\tau} \right)\\
= \dd_\Sigma\bigg( \tfrac{1}{n-1}\big((\partial_{\tau}\phi_{\tau})^2 + \vert \dd_\Sigma\phi_{\tau}\vert_{h_{\tau}}^2\big) + \tfrac{1}{2}e^{-\tfrac{4\phi_{\tau}}{n-1}}\big(\vert\psi_{\tau}\vert_{h_{\tau}}^2 + \vert H_{b_{\tau}}\vert_{h_{\tau}}^2\big) + \lambda e^{\tfrac{2\phi_{\tau}}{n-1}} \bigg) \\
-\, \tfrac{1}{2}\Tr_{h_{\tau}}(\partial_{\tau}h_{\tau})\left( \tfrac{2}{n-1}\partial_{\tau}\phi_{\tau}\, \dd_\Sigma\phi_{\tau} \right) - \tfrac{1}{2}e^{-\tfrac{4\phi_{\tau}}{n-1}}\dd_\Sigma\big(\vert \psi_{\tau} \vert _{h_{\tau}}^2+\vert H_{b_{\tau}}\vert _{h_{\tau}}^2\big) + 2\nabla^{h_{\tau} \ast}\left( \tfrac{1}{n-1}\dd_\Sigma\phi_{\tau}\otimes \dd_\Sigma\phi_{\tau} \right)
\end{gathered}
\end{equation*}

\noindent
where we have substituted the evolution equation for $\phi_{\tau}$. On the other hand:
\begin{equation*}
\begin{gathered}
\partial_{\tau} \left( e^{-\tfrac{4\phi_{\tau}}{n-1}} \psi_{\tau} \lrcorner_{h_{\tau}} H_{b_{\tau}} \right) = \\
= - e^{-\tfrac{4\phi_{\tau}}{n-1}} \left( \tfrac{4\partial_{\tau}\phi_{\tau}}{n-1} \psi_{\tau} \lrcorner_{h_{\tau}} H_{b_{\tau}} - (\partial_{\tau}h_{\tau} \circ_{h_{\tau}} \psi_{\tau}) \lrcorner_{h_{\tau}} H_{b_{\tau}} - \partial_{\tau}\psi_{\tau} \lrcorner_{h_{\tau}} H_{b_{\tau}} - \psi_{\tau} \lrcorner_{h_{\tau}} \dd_\Sigma\psi_{\tau} \right) \\
= - e^{-\tfrac{4\phi_{\tau}}{n-1}} \left\{ \tfrac{1}{2}\Tr_{h_{\tau}}(\partial_{\tau}h_{\tau})\, \psi_{\tau} \lrcorner_{h_{\tau}} H_{b_{\tau}} - \left( \psi_{\tau} \lrcorner_{h_{\tau}} \dd_\Sigma\psi_{\tau} + \left( \delta^{h_{\tau}}H_{b_{\tau}} + \tfrac{4}{n-1} \dd_\Sigma\phi_{\tau}\lrcorner_{h_{\tau}} H_{b_{\tau}} \right) \lrcorner_{h_{\tau}} H_{b_{\tau}} \right) \right\} \\
= - e^{-\tfrac{4\phi_{\tau}}{n-1}} \left( \tfrac{1}{2}\Tr_{h_{\tau}}(\partial_{\tau}h_{\tau})\, \psi_{\tau} \lrcorner_{h_{\tau}} H_{b_{\tau}} + C_{3\tau} \lrcorner_{h_{\tau}} \psi_{\tau} - \tfrac{1}{2}\dd_\Sigma\big( \vert \psi_{\tau}\vert _{h_{\tau}}^2 + \vert H_{b_{\tau}}\vert _{h_{\tau}}^2 \big) \right) \\
+\, \nabla^{h_{\tau} \ast}\left( e^{-\tfrac{4\phi_{\tau}}{n-1}}\psi_{\tau}\circ_{h_{\tau}} \psi_{\tau} + e^{-\tfrac{4\phi_{\tau}}{n-1}}H_{b_{\tau}}\circ_{h_{\tau}} H_{b_{\tau}} \right)
\end{gathered}
\end{equation*}

\noindent
where we have substituted the evolution equation for $\psi_{\tau}$ and we have calculated:
\begin{equation*}
\begin{gathered}
\nabla^{h_{\tau} \ast}\left( e^{-\tfrac{4\phi_{\tau}}{n-1}}\psi_{\tau}\circ_{h_{\tau}} \psi_{\tau} \right) = e^{-\tfrac{4\phi_{\tau}}{n-1}}\left(  C_{3\tau} \lrcorner_{h_{\tau}} \psi_{\tau} - \tfrac{1}{2} \dd_\Sigma\vert \psi_{\tau}\vert_{h_{\tau}}^2 + \psi_{\tau}\lrcorner_{h_{\tau}} H_{b_{\tau}} \right) \\
\nabla^{h_{\tau} \ast}\left( e^{-\tfrac{4\phi_{\tau}}{n-1}}H_{b_{\tau}}\circ_{h_{\tau}} H_{b_{\tau}} \right) = e^{-\tfrac{4\phi_{\tau}}{n-1}}\left( \left( \delta^{h_{\tau}}H_{b_{\tau}} + \tfrac{4}{n-1} \dd_\Sigma\phi_{\tau}\lrcorner_{h_{\tau}} H_{b_{\tau}} \right) \lrcorner_{h_{\tau}} H_{b_{\tau}} - \tfrac{1}{2}\dd_\Sigma\vert H_{b_{\tau}}\vert_{h_{\tau}}^2 \right)
\end{gathered}
\end{equation*}

\noindent
Hence:
\begin{equation}
\label{eq:NSNSconstraint1}
\begin{gathered}
\partial_{\tau} \bigg( \tfrac{2}{n-1}\partial_{\tau}\phi_{\tau}\, \dd_\Sigma\phi_{\tau} + e^{-\tfrac{4\phi_{\tau}}{n-1}} \psi_{\tau} \lrcorner_{h_{\tau}} H_{b_{\tau}} \bigg) \\
= \dd_\Sigma\bigg( \tfrac{1}{n-1}\big( (\partial_{\tau}\phi_{\tau})^2 +\vert\dd_\Sigma\phi_{\tau}\vert_{h_{\tau}}^2 \big) + \tfrac{1}{2}e^{-\tfrac{4\phi_{\tau}}{n-1}}\big( \vert\psi_{\tau}\vert_{h_{\tau}}^2 + \vert H_{b_{\tau}}\vert_{h_{\tau}}^2 \big) + \lambda e^{\tfrac{2\phi_{\tau}}{n-1}} \bigg) \\
-\, \tfrac{1}{2}\Tr_{h_{\tau}}(\partial_{\tau}h_{\tau})\bigg( \tfrac{2}{n-1}\partial_{\tau}\phi_{\tau}\, \dd_\Sigma\phi_{\tau} + e^{-\tfrac{4\phi_{\tau}}{n-1}}\psi_{\tau} \lrcorner_{h_{\tau}} H_{b_{\tau}} \bigg) \\
+\, 2\nabla^{h_{\tau} \ast}\bigg( \tfrac{1}{n-1}\dd_\Sigma\phi_{\tau} \otimes \dd_\Sigma\phi_{\tau} + \tfrac{1}{2}e^{-\tfrac{4\phi_{\tau}}{n-1}}\big( \psi_{\tau}\circ_{h_{\tau}} \psi_{\tau} + H_{b_{\tau}}\circ_{h_{\tau}} H_{b_{\tau}} \big) \bigg) - e^{-\tfrac{4\phi_{\tau}}{n-1}} C_{3\tau} \lrcorner_{h_{\tau}} \psi_{\tau}   
\end{gathered}
\end{equation}

\noindent
Altogether, summing \eqref{eq:gausscodazziconstraintNSNS1} and \eqref{eq:NSNSconstraint1}:
\begin{equation*}
\begin{aligned}
\partial_{\tau} C_{1\tau} &= \partial_{\tau} \bigg( \nabla^{h_{\tau} \ast}(\partial_{\tau}h_{\tau}) + \dd_\Sigma\Tr_{h_{\tau}}(\partial_{\tau}h_{\tau})  + \tfrac{2}{n-1}\partial_{\tau}{\phi}_{\tau}\dd_\Sigma\phi_{\tau} + e^{-\tfrac{4\phi_{\tau}}{n-1}} \psi_{\tau}\lrcorner_{h_{\tau}} H_{b_{\tau}} \bigg) \\
&= - \tfrac{1}{2}\Tr_{h_{\tau}}(\partial_{\tau}h_{\tau})C_{1\tau} + \dd_\Sigma C_{2\tau} - e^{-\tfrac{4\phi_{\tau}}{n-1}} C_{3\tau}\lrcorner_{h_{\tau}} \psi_{\tau}
\end{aligned}
\end{equation*}

\noindent
and the result follows.
\end{proof}

\begin{prop}
\label{prop:secondconstraint}
Let $(g , \bar{b}, \bar{\phi}) \in \Conf(\bar{\cP} , \bar{A}, \bar{Y})$ be a gradient Ricci soliton with reduction $( h_{\tau} , a_{\tau} , b_{\tau} , \phi_{\tau})$ on $(\cP, Y , A_{\tau}  , \Psi_{\tau})$. Then, the following equation holds:
\begin{equation*}
\partial_{\tau} C_{2\tau} = 2\delta^{h_{\tau}} C_{1\tau} - \Tr_{h_{\tau}}(\partial_{\tau}h_{\tau}) C_{2\tau}
\end{equation*}
\end{prop}

\begin{proof}
First, use the following identity proven in \cite[Proposition 6]{Koiso1981}:
\begin{equation}
\label{eq:gausscodazziconstraint2}
\begin{gathered}
\partial_{\tau} \left( s^{h_{\tau}} - \tfrac{1}{4}(\Tr_{h_{\tau}}\partial_{\tau}h_{\tau})^2 + \tfrac{1}{4}\vert \partial_{\tau}h_{\tau} \vert^2_{h_{\tau}} \right) = - \Tr_{h_{\tau}}(\partial_{\tau}h_{\tau}) \bigg( s^{h_{\tau}} - \tfrac{1}{4}(\Tr_{h_{\tau}}\partial_{\tau}h_{\tau})^2 + \tfrac{1}{4}\vert \partial_{\tau}h_{\tau} \vert^2_{h_{\tau}} \bigg) \\
- \delta^{h_{\tau}}\bigg( \nabla^{h_{\tau} \ast}\partial_{\tau}h_{\tau} + \dd_\Sigma\Tr_{h_{\tau}}(\partial_{\tau}h_{\tau}) \bigg) + \Tr_{h_{\tau}}(\partial_{\tau}h_{\tau})\, \Tr_{h_{\tau}}\big(\Ric^g\vert_{\Sigma}\big) - h_{\tau}(\partial_{\tau}h_{\tau}  ,\Ric^g\vert_{\Sigma})
\end{gathered}
\end{equation}

\noindent
Substituting the expression of $\Ric^g\vert_\Sigma$ in \eqref{eq:NSNSEinsteinframeII}, one obtains:
\begin{equation}
\begin{gathered}
\label{eq:gausscodazziconstraintNSNS2}
\partial_{\tau} \left( s^{h_{\tau}} - \tfrac{1}{4}(\Tr_{h_{\tau}}\partial_{\tau}h_{\tau})^2 + \tfrac{1}{4}\vert \partial_{\tau}h_{\tau} \vert^2_{h_{\tau}} \right) = \\
= - \Tr_{h_{\tau}}(\partial_{\tau}h_{\tau}) \left( s^{h_{\tau}} - \tfrac{1}{4}(\Tr_{h_{\tau}}\partial_{\tau}h_{\tau})^2 + \tfrac{1}{4}\vert \partial_{\tau}h_{\tau} \vert^2_{h_{\tau}} - \tfrac{1}{n-1}\vert\dd_\Sigma\phi_{\tau}\vert_{h_{\tau}}^2 - \tfrac{1}{2}e^{-\tfrac{4\phi_{\tau}}{n-1}}\vert H_{b_{\tau}}\vert_{h_{\tau}}^2 - \lambda e^{\tfrac{2\phi_{\tau}}{n-1}} \right) \\
- \delta^{h_{\tau}}\bigg( \nabla^{h_{\tau} \ast}\partial_{\tau}h_{\tau} + \dd_\Sigma\Tr_{h_{\tau}}(\partial_{\tau}h_{\tau}) \bigg) - h_{\tau}\bigg( \partial_{\tau}h_{\tau}  ,\tfrac{1}{n-1}\dd_\Sigma\phi_{\tau} \otimes \dd_\Sigma\phi_{\tau} + \tfrac{1}{2}e^{-\tfrac{4\phi_{\tau}}{n-1}}\big(\psi_{\tau}\circ_{h_{\tau}} \psi_{\tau} + H_{b_{\tau}}\circ_{h_{\tau}} H_{b_{\tau}}\big) \bigg)
\end{gathered}
\end{equation}

\noindent
Second, we compute the following quantity:
\begin{equation*}
\partial_{\tau} \left( \frac{1}{n-1}\big( (\partial_{\tau}{\phi}_{\tau})^2 - \vert\dd_\Sigma\phi_{\tau}\vert_{h_{\tau}}^2 \big) + \frac{1}{2}e^{-\tfrac{4\phi_{\tau}}{n-1}}\big( \vert\psi_{\tau}\vert_{h_{\tau}}^2 - \vert H_{b_{\tau}}\vert_{h_{\tau}}^2 \big) + \lambda e^{\tfrac{2\phi_{\tau}}{n-1}} \right)
\end{equation*}

\noindent
On the one hand, we have:
\begin{equation*}
\begin{gathered}
\partial_{\tau} \bigg( \frac{1}{n-1}\big( (\partial_{\tau}{\phi}_{\tau})^2 - \vert\dd_\Sigma\phi_{\tau}\vert_{h_{\tau}}^2 \big) \bigg) = \frac{1}{n-1}\bigg( 2\partial_{\tau}^2{\phi}_{\tau}\, \partial_{\tau}\phi_{\tau} - 2\langle \dd_\Sigma\partial_{\tau}\phi_{\tau}, \dd_\Sigma\phi_{\tau} \rangle_{h_{\tau}} + h_{\tau}(\partial_{\tau}h_{\tau}, \dd_\Sigma\phi_{\tau} \otimes \dd_\Sigma\phi_{\tau}) \bigg) \\
= - \frac{2}{n-1}e^{-\tfrac{4\phi_{\tau}}{n-1}}\partial_{\tau}\phi_{\tau}\big( \vert\psi_{\tau}\vert_{h_{\tau}}^2 + \vert H_{b_{\tau}}\vert_{h_{\tau}}^2 \big) + \frac{2}{n-1}\partial_{\tau}\phi_{\tau}\, \lambda e^{\tfrac{2\phi_{\tau}}{n-1}} - \frac{1}{n-1}\Tr_{h_{\tau}}(\partial_{\tau}h_{\tau})(\partial_{\tau}\phi_{\tau})^2 \\
+ 2 \delta^{h_{\tau}}\left( \frac{1}{n-1}\partial_{\tau}\phi_{\tau}\, \dd_\Sigma\phi_{\tau}  \right) + \frac{1}{n-1} h_{\tau}\big( \partial_{\tau}h_{\tau}, \dd_\Sigma\phi_{\tau} \otimes \dd_\Sigma\phi_{\tau} \big)    
\end{gathered}
\end{equation*}

\noindent
where we have substituted the evolution equation for $\phi_{\tau}$. On the other hand, we obtain:
\begin{equation*}
\begin{gathered}
\partial_{\tau} \bigg( \frac{1}{2}e^{-\tfrac{4\phi_{\tau}}{n-1}} \big(\vert\psi_{\tau}\vert_{h_{\tau}}^2 - \vert H_{b_{\tau}} \vert_{h_{\tau}}^2\big) \bigg) = e^{-\tfrac{4\phi_{\tau}}{n-1}}\bigg\{ \frac{2}{n-1}\partial_{\tau}\phi_{\tau}\big(-\vert\psi_{\tau}\vert_{h_{\tau}}^2 + \vert H_{b_{\tau}}\vert_{h_{\tau}}^2\big) \\
+ \frac{1}{2}\Big[ - h_{\tau}\big(\partial_{\tau}h_{\tau},\psi_{\tau}\circ_{h_{\tau}} \psi_{\tau}\big) + h_{\tau}\big(\partial_{\tau}h_{\tau},H_{b_{\tau}}\circ_{h_{\tau}} H_{b_{\tau}}\big) \Big] + \langle \partial_{\tau}^2{\psi}_{\tau},\psi_{\tau} \rangle_{h_{\tau}} - \langle \dd_\Sigma\psi_{\tau},H_{b_{\tau}}\rangle_{h_{\tau}} \bigg\} = \\
= e^{-\tfrac{4\phi_{\tau}}{n-1}}\bigg\{ \frac{2}{n-1}\partial_{\tau}\phi_{\tau}\big(\vert\psi_{\tau}\vert_{h_{\tau}}^2 + \vert H_{b_{\tau}}\vert_{h_{\tau}}^2\big) + \frac{1}{2}h_{\tau}\big(\partial_{\tau}h_{\tau},\psi_{\tau}\circ_{h_{\tau}} \psi_{\tau} + H_{b_{\tau}}\circ_{h_{\tau}} H_{b_{\tau}} \big)  \\
- \frac{1}{2}\Tr_{h_{\tau}}(\partial_{\tau}h_{\tau})\, \vert\psi_{\tau}\vert_{h_{\tau}}^2 + \left\langle \delta^{h_{\tau}}H_{b_{\tau}} + \frac{4}{n-1}\, \dd_\Sigma\phi_{\tau}\lrcorner_{h_{\tau}}H_{b_{\tau}},\psi_{\tau}  \right\rangle_{h_{\tau}} - \langle \dd_\Sigma\psi_{\tau} ,H_{b_{\tau}} \rangle_{h_{\tau}} \bigg\}   
\end{gathered}
\end{equation*}

\noindent
where we have substituted the evolution equation for $\psi_{\tau}$ and we have computed:
\begin{equation*}
\begin{aligned}
\langle\psi_{\tau}, \partial_{\tau}h_{\tau} \Delta_1\psi_{\tau}\rangle_{h_{\tau}} &= - \frac{1}{2}h_{\tau}\bigg( \psi_{\tau}, \sum_{i=1}^n \partial_{\tau}h_{\tau}(e_i) \wedge \psi_{\tau}(e_i) \bigg) \\
&= - \frac{1}{2}\sum_{i,k=1}^n \partial_{\tau}h_{\tau}(e_k,e_i)\, h_{\tau}\big(\psi_{\tau}(e_k), \psi_{\tau}(e_i) \big) = - \frac{1}{2}h_{\tau}\big(\partial_{\tau}h_{\tau}, \psi_{\tau}\circ_{h_{\tau}} \psi_{\tau}\big)    
\end{aligned}
\end{equation*}

\noindent
and (see Proposition \ref{prop:firstconstraint}):
\begin{equation*}
2 \delta^{h_{\tau}}\left( - \frac{1}{2}e^{-\tfrac{4\phi_{\tau}}{n-1}}\psi_{\tau} \lrcorner_{h_{\tau}} H_{b_{\tau}} \right)
= e^{-\tfrac{4\phi_{\tau}}{n-1}}\left( - \left\langle \psi_{\tau}, \delta^{h_{\tau}}H_{b_{\tau}} + \frac{4}{n-1} \dd_\Sigma\phi_{\tau}\lrcorner_{h_{\tau}} H_{b_{\tau}} \right\rangle_{h_{\tau}} + \langle \dd_\Sigma\psi_{\tau},H_{b_{\tau}}\rangle_{h_{\tau}} \right)
\end{equation*}

\noindent
Rearranging:
\begin{equation*}
\begin{gathered}
\partial_{\tau} \bigg( \frac{1}{2}e^{-\tfrac{4\phi_{\tau}}{n-1}}\big( \vert\psi_{\tau}\vert_{h_{\tau}}^2 - \vert H_{b_{\tau}}\vert_{h_{\tau}}^2 \big) \bigg) = \delta^{h_{\tau}} \big( e^{-\tfrac{4\phi_{\tau}}{n-1}}\psi_{\tau} \lrcorner_{h_{\tau}} H_{b_{\tau}} \big) \\
+ e^{-\tfrac{4\phi_{\tau}}{n-1}}\bigg[ \frac{2}{n-1}\partial_{\tau}\phi_{\tau}\big( \vert\psi_{\tau}\vert_{h_{\tau}}^2 + \vert H_{b_{\tau}}\vert_{h_{\tau}}^2 \big) + \frac{1}{2}h_{\tau}\big(\partial_{\tau}h_{\tau}, \psi_{\tau}\circ_{h_{\tau}} \psi_{\tau} + H_{b_{\tau}}\circ_{h_{\tau}} H_{b_{\tau}} \big) - \frac{1}{2}\Tr_{h_{\tau}}(\partial_{\tau}h_{\tau})\, \vert\psi_{\tau}\vert_{h_{\tau}}^2 \bigg]
\end{gathered}
\end{equation*}

\noindent
Then:
\begin{equation}
\label{eq:NSNSconstraint2}
\begin{gathered}
\partial_{\tau} \bigg( \frac{1}{n-1}\big( (\partial_{\tau}{\phi}_{\tau})^2 - \vert\dd_\Sigma\phi_{\tau}\vert_{h_{\tau}}^2 \big) + \frac{1}{2}e^{-\tfrac{4\phi_{\tau}}{n-1}}\big( \vert\psi_{\tau}\vert_{h_{\tau}}^2 - \vert H_{b_{\tau}}\vert_{h_{\tau}}^2 \big) + \lambda e^{\tfrac{2\phi_{\tau}}{n-1}} \bigg) = \\
= - \Tr_{h_{\tau}}(\partial_{\tau}h_{\tau})\bigg( \frac{1}{n-1}(\partial_{\tau}\phi_{\tau})^2 + \frac{1}{2}e^{-\tfrac{4\phi_{\tau}}{n-1}}\vert\psi_{\tau}\vert_{h_{\tau}}^2 \bigg) + 2 \delta^{h_{\tau}}\bigg( \frac{1}{n-1}\partial_{\tau}\phi_{\tau}\, \dd_\Sigma\phi_{\tau} + \frac{1}{2}e^{-\tfrac{4\phi_{\tau}}{n-1}}\psi_{\tau} \lrcorner_{h_{\tau}} H_{b_{\tau}} \bigg) \\
+ h_{\tau}\bigg( \partial_{\tau}h_{\tau}, \frac{1}{n-1} \dd_\Sigma\phi_{\tau} \otimes \dd_\Sigma\phi_{\tau} + \frac{1}{2}e^{-\tfrac{4\phi_{\tau}}{n-1}}\big( \psi_{\tau}\circ_{h_{\tau}} \psi_{\tau} + H_{b_{\tau}}\circ_{h_{\tau}} H_{b_{\tau}} \big) \bigg) + \frac{4}{n-1}\partial_{\tau}\phi_{\tau}\, \lambda e^{\tfrac{2\phi_{\tau}}{n-1}}
\end{gathered}
\end{equation}

\noindent
Altogether:
\begin{equation*}
\begin{aligned}
\partial_{\tau} C_{2\tau} &= \partial_{\tau} \bigg( s^{h_{\tau}} - \tfrac{1}{4}(\Tr_{h_{\tau}}\partial_{\tau}h_{\tau})^2 + \tfrac{1}{4}\vert \partial_{\tau}h_{\tau} \vert^2_{h_{\tau}} + \frac{1}{n-1}\big( (\partial_{\tau}{\phi}_{\tau})^2 - \vert\dd_\Sigma\phi_{\tau}\vert_{h_{\tau}}^2 \big) \\
&\qquad\quad + \frac{1}{2}e^{-\tfrac{4\phi_{\tau}}{n-1}}\big( \vert\psi_{\tau}\vert_{h_{\tau}}^2 - \vert H_{b_{\tau}}\vert_{h_{\tau}}^2 \big) + \lambda e^{\tfrac{2\phi_{\tau}}{n-1}} \bigg) = 2 \delta^{h_{\tau}} C_{1\tau} - \Tr_{h_{\tau}}(\partial_{\tau}h_{\tau})C_{2\tau}    
\end{aligned}   
\end{equation*}
and the result follows.
\end{proof}

\begin{prop}
\label{prop:thirdconstraint}
Let $(g , \bar{b}, \bar{\phi}) \in \Conf(\bar{\cP} , \bar{A}, \bar{Y})$ be a gradient Ricci soliton with reduction $( h_{\tau} , a_{\tau} , b_{\tau} , \phi_{\tau})$ on $(\cP, Y , A_{\tau}  , \Psi_{\tau})$. Then, the following equation holds:
\begin{equation*}
\partial_{\tau} C_{3\tau} = - \left( \tfrac{1}{2}\Tr_{h_{\tau}}(\partial_{\tau}h_{\tau}) - \tfrac{4}{n-1}\partial_{\tau}\phi_{\tau} \right) C_{3\tau} - \partial_{\tau}h_{\tau}(C_{3\tau}^{\sharp_{h_{\tau}}})
\end{equation*}
\end{prop}

\begin{proof}
On the one hand (see, \cite[Lemma 5.3]{BunkMunozShahbazi2023}), we have:
\begin{equation*}
\partial_{\tau} (\delta^{h_{\tau}} \psi_{\tau}) = \partial_{\tau}h_{\tau}(\delta^{h_{\tau}}\psi_{\tau}) - \tfrac{1}{2}\psi_{\tau}\left(\dd_\Sigma\Tr_{h_{\tau}}(\partial_{\tau}h_{\tau})^{\sharp_{h_{\tau}}}\right) + \delta^{h_{\tau}}(\partial_{\tau}h_{\tau}\Delta_1\psi_{\tau})  + \delta^{h_{\tau}}\partial_{\tau}\psi_{\tau}      
\end{equation*}

\noindent
On the other hand:
\begin{equation*}
\partial_{\tau} \left( \frac{4}{n-1} \dd_\Sigma\phi_{\tau}\lrcorner_{h_{\tau}} \psi_{\tau} \right) = \frac{4}{n-1}\left( - \partial_{\tau}h_{\tau}\big( (\dd_\Sigma\phi_{\tau}\lrcorner_{h_{\tau}} \psi_{\tau})^{\sharp_{h_{\tau}}} \big) + \dd_\Sigma\partial_{\tau}\phi_{\tau}\lrcorner_{h_{\tau}} \psi_{\tau} + \dd_\Sigma\phi_{\tau}\lrcorner_{h_{\tau}} \partial_{\tau}\psi_{\tau} \right)
\end{equation*}

\noindent
Therefore:
\begin{equation*}
\begin{gathered}
\partial_{\tau} \left( \delta^{h_{\tau}} \psi_{\tau} + \frac{4}{n-1} \dd_\Sigma\phi_{\tau}\lrcorner_{h_{\tau}} \psi_{\tau} \right) = - \partial_{\tau}h_{\tau}\left( (\delta^{h_{\tau}}\psi_{\tau})^{\sharp_{h_{\tau}}} + \frac{4}{n-1}(\dd_\Sigma\phi_{\tau}\lrcorner_{h_{\tau}} \psi_{\tau})^{\sharp_{h_{\tau}}} \right) \\
- \left(\tfrac{1}{2}\dd_\Sigma\Tr_{h_{\tau}}(\partial_{\tau}h_{\tau}) - \tfrac{4}{n-1}\dd_\Sigma\partial_{\tau}\phi_{\tau}\right)\lrcorner_{h_{\tau}} \psi_{\tau} + \delta^{h_{\tau}}(\partial_{\tau}h_{\tau}\Delta_1\psi_{\tau}) + \left( \delta^{h_{\tau}}\partial_{\tau}\psi_{\tau} + \frac{4}{n-1} \dd_\Sigma\phi_{\tau}\lrcorner_{h_{\tau}} \partial_{\tau}\psi_{\tau} \right)
\end{gathered}
\end{equation*}

\noindent
Furthermore:
\begin{equation*}
\begin{gathered}
\delta^{h_{\tau}}\partial_{\tau}\psi_{\tau} + \frac{4}{n-1} \dd_\Sigma\phi_{\tau}\lrcorner_{h_{\tau}} \partial_{\tau}\psi_{\tau} = \frac{4}{n-1} \partial_{\tau}h_{\tau}((\dd_\Sigma\phi_{\tau})\lrcorner_{h_{\tau}} \psi_{\tau} + \frac{4}{n-1} \partial_{\tau}h_{\tau}\big( (\dd_\Sigma\phi_{\tau}\lrcorner_{h_{\tau}} \psi_{\tau})^{\sharp_{h_{\tau}}} \big) \\
+ \left( \tfrac{1}{2}\dd_\Sigma\Tr_{h_{\tau}}(\partial_{\tau}h_{\tau}) - \tfrac{4}{n-1}\dd_\Sigma\partial_{\tau}\phi_{\tau} \right)\lrcorner_{h_{\tau}} \psi_{\tau} - \delta^{h_{\tau}}(\partial_{\tau}h_{\tau}\Delta_1\psi_{\tau}) - \left( \tfrac{1}{2}\Tr_{h_{\tau}}(\partial_{\tau}h_{\tau}) - \tfrac{4}{n-1}\partial_{\tau}\phi_{\tau} \right)C_{3\tau}
\end{gathered}
\end{equation*}

\noindent
where we have used the evolution equation for $\psi_{\tau}$ and we have computed:
\begin{equation*}
\begin{gathered}
\delta^{h_{\tau}}(\partial_{\tau}h_{\tau}\Delta_1\psi_{\tau}) + \frac{4}{n-1} \dd_\Sigma\phi_{\tau}\lrcorner_{h_{\tau}} (\partial_{\tau}h_{\tau}\Delta_1\psi_{\tau}) = \\
= \delta^{h_{\tau}}(\partial_{\tau}h_{\tau}\Delta_1\psi_{\tau}) - \frac{4}{n-1} \partial_{\tau}h_{\tau}((\dd_\Sigma\phi_{\tau})\lrcorner_{h_{\tau}} \psi_{\tau} - \frac{4}{n-1} \partial_{\tau}h_{\tau}\big( (\dd_\Sigma\phi_{\tau}\lrcorner_{h_{\tau}} \psi_{\tau})^{\sharp_{h_{\tau}}} \big) \\
\delta^{h_{\tau}}\bigg( \Big( \tfrac{1}{2}\Tr_{h_{\tau}}(\partial_{\tau}h_{\tau}) - \tfrac{4}{n-1}\partial_{\tau}\phi_{\tau} \Big) \psi_{\tau} \bigg) + \frac{4}{n-1} \dd_\Sigma\phi_{\tau}\lrcorner_{h_{\tau}} \bigg( \Big( \tfrac{1}{2}\Tr_{h_{\tau}}(\partial_{\tau}h_{\tau}) - \tfrac{4}{n-1}\partial_{\tau}\phi_{\tau} \Big) \psi_{\tau} \bigg) = \\
= \left( \tfrac{1}{2}\Tr_{h_{\tau}}(\partial_{\tau}h_{\tau}) - \tfrac{4}{n-1}\partial_{\tau}\phi_{\tau} \right) C_{3\tau} - \left( \tfrac{1}{2}\dd_\Sigma\Tr_{h_{\tau}}(\partial_{\tau}h_{\tau}) - \tfrac{4}{n-1}\dd_\Sigma\partial_{\tau}\phi_{\tau} \right)\lrcorner_{h_{\tau}} \psi_{\tau} \\
\left( \delta^{h_{\tau}} + \frac{4}{n-1} \dd_\Sigma\phi_{\tau}\lrcorner_{h_{\tau}} \right)^2 H_{b_{\tau}} = 0
\end{gathered}
\end{equation*}

\noindent
Altogether:
\begin{equation*}
\begin{aligned}
\partial_{\tau} C_{3\tau} = \partial_{\tau} \left( \delta^{h_{\tau}}\psi_{\tau} + \frac{4}{n-1} \dd_\Sigma\phi_{\tau}\lrcorner_{h_{\tau}} \psi_{\tau} \right) = - \left( \tfrac{1}{2}\Tr_{h_{\tau}}(\partial_{\tau}h_{\tau}) - \tfrac{4}{n-1}\partial_{\tau}\phi_{\tau} \right) C_{3\tau} - \partial_{\tau}h_{\tau}(C_{3\tau}^{\sharp_{h_{\tau}}})
\end{aligned}
\end{equation*}

and the result follows.
\end{proof}

\phantomsection
\bibliographystyle{JHEP}

\end{document}